\documentclass[10pt,a4paper,reqno,oneside]{amsart}

\usepackage{microtype}

\usepackage{setspace}
\usepackage[totalwidth=14cm,totalheight=22cm]{geometry}
\usepackage[T1]{fontenc}
\usepackage[dvips]{graphicx}
\usepackage{epsfig}
\usepackage{amsmath,amssymb,amscd,amsthm}
\usepackage{subfigure}
\usepackage[latin1]{inputenc}
\usepackage{latexsym}
\usepackage{graphics}
\usepackage{colortbl}
\usepackage{color}
\usepackage{ae}
\usepackage{enumitem}
\usepackage[all]{xy}
\usepackage{scalerel}
\usepackage{tikz}
\usepackage{tikz-cd}
\usepackage{cancel}
\usepackage{bbm,amsbsy}
\usepackage{mathrsfs} 
\usepackage{nicefrac}
\usepackage{comment}

    \graphicspath{{../}}
    \makeatletter
    \providecommand*{\input@path}{}
    \g@addto@macro\input@path{{../}}
    \makeatother
    \newcommand{\red}{\textcolor{red}} %
    \newcommand{\blue}{\textcolor{blue}} %
    \newcommand{\magenta}{\textcolor{magenta}} 
    \newcommand{\orange}{\textcolor{orange}}
    \newcommand{\purple}{\textcolor{purple}}
    \usepackage[normalem]{ulem}

\newlist{steps}{enumerate}{1}
\setlist[steps, 1]{itemsep=3pt,leftmargin=0cm,itemindent=.5cm,labelwidth=\itemindent,labelsep=0cm,align=left,label = \emph{Step \arabic*}:\,}

\makeatletter
\newtheorem*{rep@theorem}{\rep@title}
\newcommand{\newreptheorem}[2]{%
\newenvironment{rep#1}[1]{%
 \def\rep@title{#2 \ref{##1}}%
 \begin{rep@theorem}}%
 {\end{rep@theorem}}}
\makeatother

\makeatletter
\newtheorem*{rep@cor}{\rep@title}
\newcommand{\newrepcor}[2]{%
\newenvironment{rep#1}[1]{%
 \def\rep@title{#2 \ref{##1}}%
 \begin{rep@cor}}%
 {\end{rep@cor}}}
\makeatother

\makeatletter
\newtheorem*{rep@prop}{\rep@title}
\newcommand{\newrepprop}[2]{%
\newenvironment{rep#1}[1]{%
 \def\rep@title{#2 \ref{##1}}%
 \begin{rep@prop}}%
 {\end{rep@prop}}}
\makeatother

\newtheorem{cor}{Corollary}[section]
\newtheorem{corx}{Corollary}

\newtheorem{thmx}[corx]{Theorem}

\newtheorem{prop}[cor]{Proposition}
\newtheorem{propx}[corx]{Proposition}

\newrepcor{cor}{Corollary}
\newreptheorem{theorem}{Theorem}
\newrepprop{prop}{Proposition}

\newtheorem{lemma}[cor]{Lemma}

\newtheorem*{problem*}{Problem}

\theoremstyle{definition}
\newtheorem{defi}[cor]{Definition}
\theoremstyle{remark}
\newtheorem{remark}[cor]{Remark}
\newtheorem*{remark*}{Remark}
\newtheorem{example}[cor]{Example}

\theoremstyle{plain}
\newcommand{\thistheoremname}{}
\newtheorem*{genericthm}{\thistheoremname}

\newcommand{\ginv}{\Gamma}

\newcommand{\R}{{\mathbb R}}
\newcommand{\Z}{{\mathbb Z}}

\newcommand{\Hyp}{\mathbb{H}}

\newcommand{\ph}{\varphi}

\newcommand{\arcsinh}{\mathrm{arcsinh}}

\newcommand{\Isom}{\mathrm{Isom}}

\newcommand{\grad}{\operatorname{grad}}

\newcommand{\D}{\mathbb{D}}

\newcommand{\I}{\mathrm{I}}
\newcommand{\II}{\mathrm{I}\hspace{-0.04cm}\mathrm{I}}
\newcommand{\III}{\mathrm{I}\hspace{-0.04cm}\mathrm{I}\hspace{-0.04cm}\mathrm{I}}

\DeclareRobustCommand{\SkipTocEntry}[4]{}

\begin{document}

\setcounter{secnumdepth}{3}
\setcounter{tocdepth}{2}

\title[Completeness of convex entire surfaces in Minkowski 3-space]{Completeness of convex entire surfaces\\in Minkowski 3-space}

\author[Francesco Bonsante]{Francesco Bonsante}
\address{Francesco Bonsante: Dipartimento di Matematica ``Felice Casorati", Universit\`{a} degli Studi di Pavia, Via Ferrata 5, 27100, Pavia, Italy.} \email{bonfra07@unipv.it} 
\author[Andrea Seppi]{Andrea Seppi}
\address{Andrea Seppi: Institut Fourier, UMR 5582, Laboratoire de Math\'ematiques,
Universit\'e Grenoble Alpes, CS 40700, 38058 Grenoble cedex 9, France.} \email{andrea.seppi@univ-grenoble-alpes.fr}
\author[Peter Smillie]{Peter Smillie}
\address{Peter Smillie: Mathematisches Institut, Universit\"at Heidelberg, Im Neuenheimer Feld 205, 69120 Heidelberg, Germany.} \email{psmillie@mathi.uni-heidelberg.de}


\thanks{The first aurthor was partially supported by Blue Sky Research project ``Analytic and geometric properties of low-dimensional manifolds" . The first two authors are members of the national research group GNSAGA}

\begin{abstract}
We prove four results towards a description, in terms of the null support function, of the set of isometric embeddings of the hyperbolic plane into Minkowski 3-space. We show that for sufficiently tame null support function, the corresponding entire surface of constant curvature -1 is complete, and for sufficiently sharp null support function, it is incomplete. Our results apply also to entire surfaces whose curvature is merely bounded.
\end{abstract}

\maketitle
%

\tableofcontents

\section{Introduction}

The purpose of this paper is to study complete spacelike surfaces of constant intrinsic curvature $-1$ (which we call hyperbolic), or more generally surfaces whose curvature is non-positive and is bounded above or below by negative constants, in Minkowski three-space $\R^{2,1}$. 

\subsection*{Examples and statement of the problem} It is easily shown that every complete spacelike surface in $\R^{2,1}$ is entire, meaning that it is the graph of a function globally defined on the horizontal plane. Moreover, if the intrinsic curvature is non-positive, then (up to a reflection in a horizontal plane) such function is convex. The very first example is the well-known hyperboloid model of the hyperbolic plane $\Hyp^2$, namely the future unit sphere in $\R^{2,1}$, see Figure \ref{fig:hyperboloid}.

\begin{figure}[htb]
\centering
\includegraphics[height=4cm]{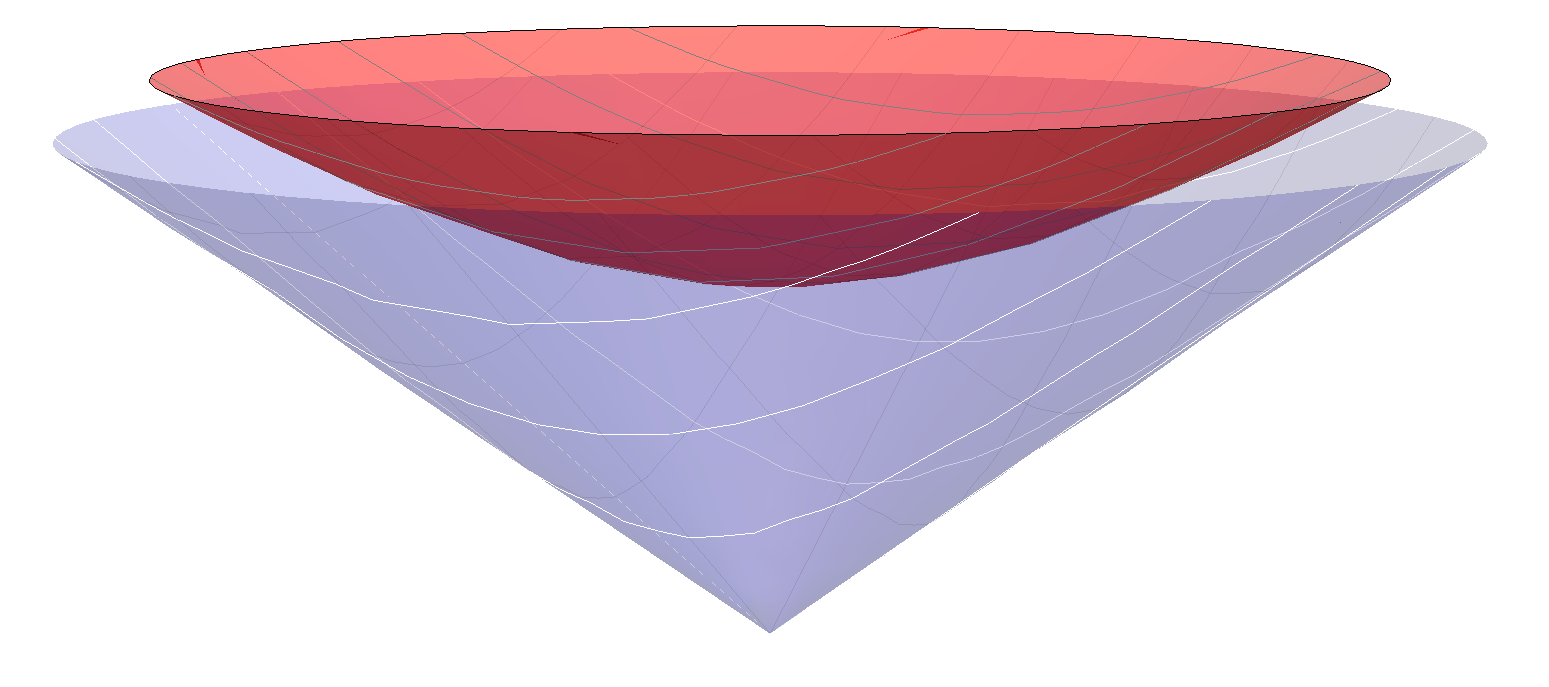}
\caption{The hyperboloid model of the hyperbolic plane (in red). In this and the following figures, the blue surface is the boundary of the domain of dependence of the red surface. \label{fig:hyperboloid}}
\end{figure}

The first non-umbilic examples of complete hyperbolic surfaces in $\R^{2,1}$ have been obtained by Hano and Nomizu in 1983 (\cite{hanonomizu}). By considering ``surfaces of revolution'' with spacelike axis, they reduced the problem of finding a function whose graph has constant intrinsic curvature $-1$ to an ordinary differential equation, whose maximal solutions give a one-parameter family of non-equivalent  surfaces. It turns out that these surfaces are complete, and thus intrinsically isometric to $\Hyp^2$. See Figures \ref{fig:hanonomizu} and \ref{fig:hanonomizu2} (on the left).

Another large source of examples arose from the work \cite{barbotzeghib} of Barbot-B\'eguin-Zeghib: they showed that for every $g\geq 2$ and every representation $\rho:\pi_1(S_g)\to\Isom(\Hyp^3)\cong \mathrm O(2,1)\ltimes\R^{2,1}$ ($S_g$ being a closed orientable surface of genus $g$) whose linear part is Fuchsian, there exists a unique $\rho$-equivariant embedding of constant curvature $-1$. By cocompactness, the induced metric is complete and therefore isometric to $\Hyp^2$; these surfaces are non-equivalent to the hyperboloid unless $\rho$ has a global fixed point.

\begin{figure}[htb]
\centering
\begin{minipage}[c]{.65\textwidth}
\centering
\includegraphics[height=3.3cm]{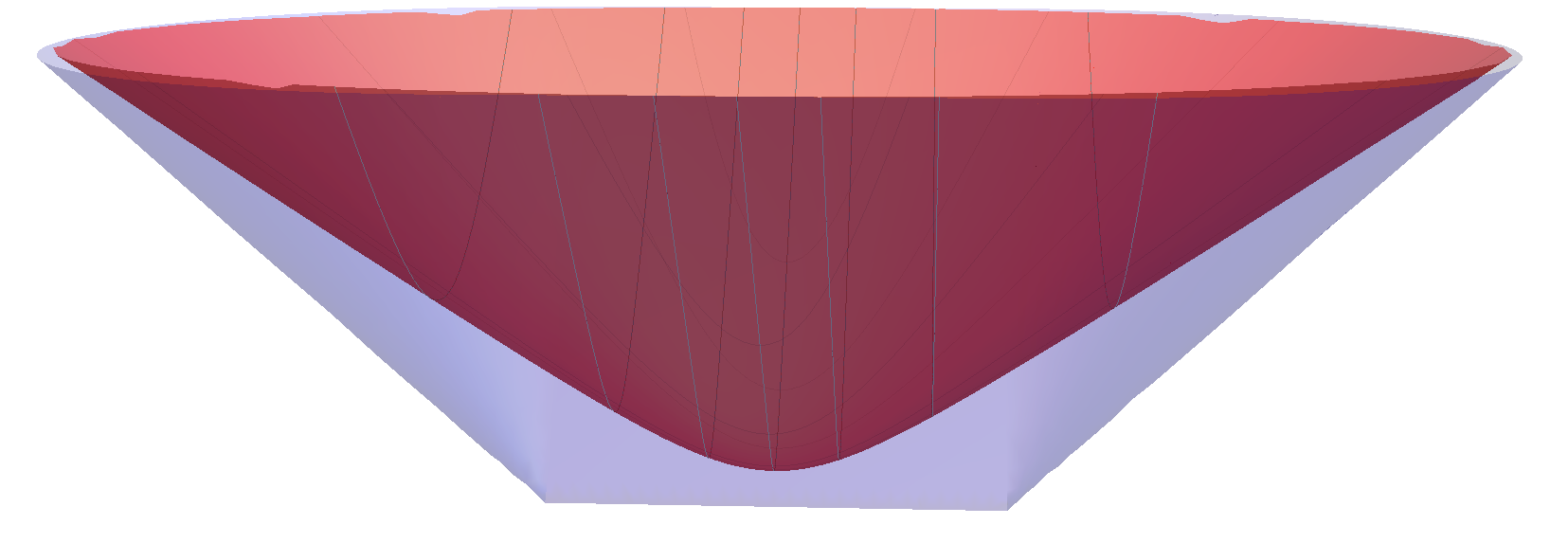}
\end{minipage}%
\begin{minipage}[c]{.35\textwidth}
\centering
\includegraphics[height=4.5cm]{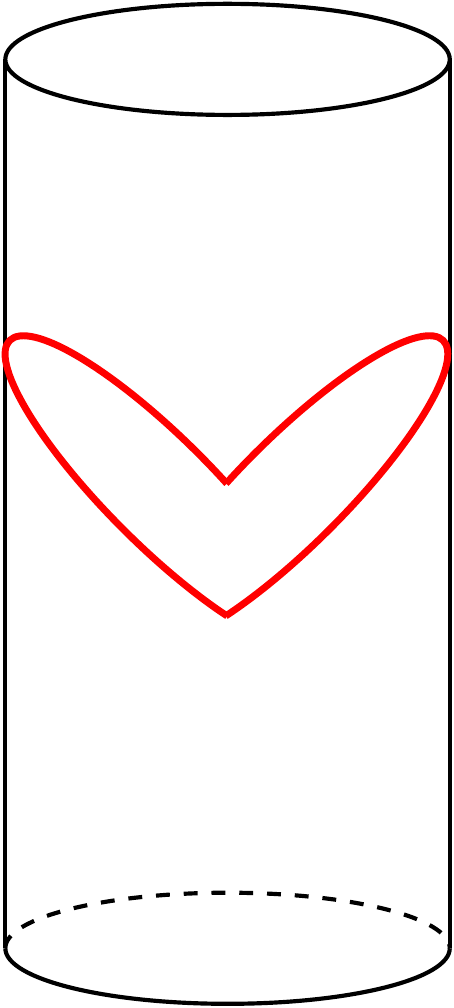}
\end{minipage}%

\caption{A complete hyperbolic surfaces in the Hano-Nomizu family, obtained as surface of revolution with spacelike axis. Its null support is a continuous piecewise affine function on the circle.}
\label{fig:hanonomizu}
\end{figure}

\begin{figure}[htb]
\centering
\begin{minipage}[c]{.65\textwidth}
\centering
\includegraphics[height=3.3cm]{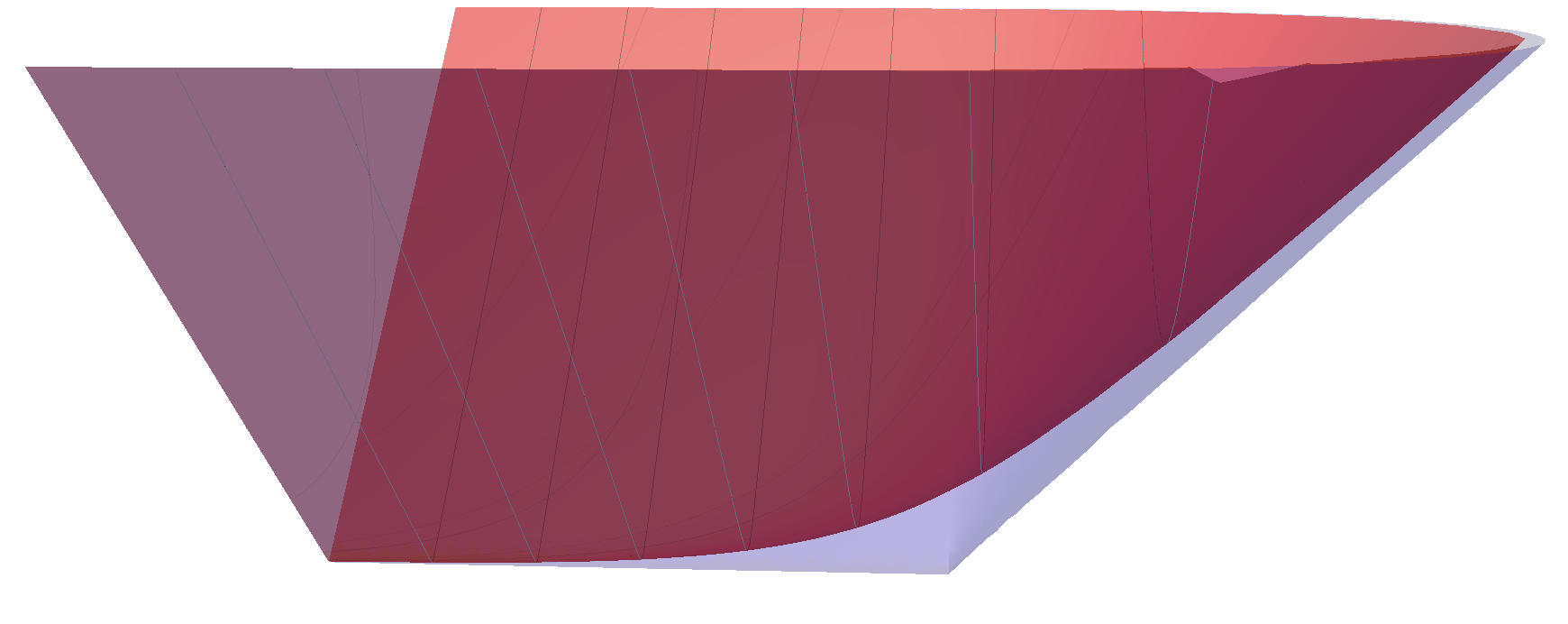}
\end{minipage}%
\begin{minipage}[c]{.35\textwidth}
\centering
\includegraphics[height=4.5cm]{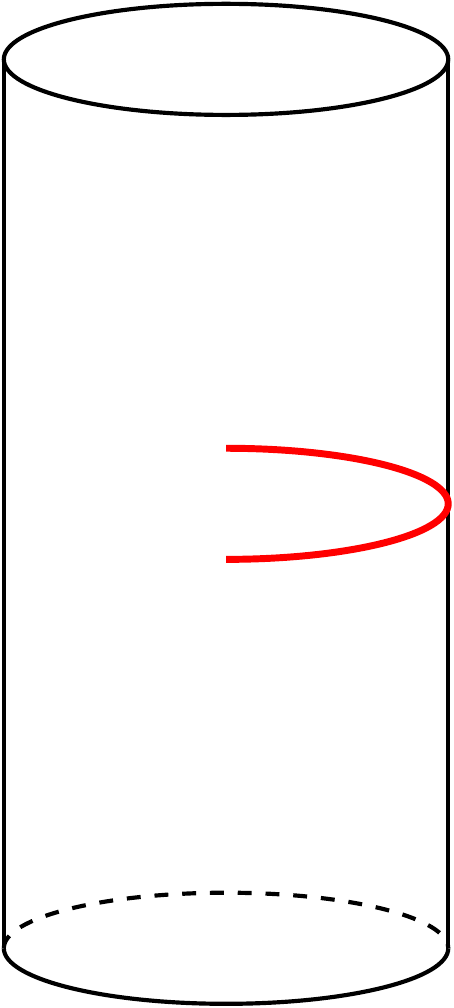}
\end{minipage}%
\caption{Another example in the Hano-Nomizu family, called the \emph{semitrough}. In this case the null support function equals zero on a closed arc, and $+\infty$ on its complement.}
\label{fig:hanonomizu2}
\end{figure}

A third source of examples arises from the study of the asymptotics of entire hyperbolic surfaces. To explain this, let us introduce a fundamental object for the present work, namely the \emph{null support function}. Given an entire convex spacelike surface $\Sigma$ in $\R^{2,1}$, expressed as the graph of a convex, 1-Lipschitz function $f:\R^2\to\R$, we define the function $\phi:\mathbb S^1\to\R\cup\{+\infty\}$ by
$$\phi(\theta)=\sup_{(x,y)\in\R^2}(x\cos\theta+y\sin\theta-f(x,y))=\sup_{p\in\Sigma}\langle p,\vec\theta\rangle~,$$
which is easily seen to be well-defined and lower-semicontinuous as a consequence of convexity of $f$. Here and in the following, we use the notation $\vec\theta:=(\cos\theta,\sin\theta,1)$. 

The geometric interpretation of the null support function is the following. Given $\theta\in \mathbb S^1$, the null affine plane $P$ defined by the equation $\langle \vec\theta,\cdot\rangle=\phi(\theta)$ is a support plane for $\Sigma$, meaning that every translate of $P$ in the future intersects $\Sigma$, while every translate in the past is disjoint from $\Sigma$. We also remark that the intersection of the future half-spaces bounded by the null support planes of equation $\langle \vec\theta,\cdot\rangle=\phi(\theta)$, as $\theta$ varies in $\mathbb S^1$, is called \emph{domain of dependence} $\mathcal D$ of $\Sigma$, and is pictured in blue in Figures \ref{fig:hyperboloid}, \ref{fig:hanonomizu} and \ref{fig:hanonomizu2}.
Some concrete examples: the null support function of the hyperboloid of Figure \ref{fig:hyperboloid} is identically zero and its domain of dependence is the future cone over the origin; the null support functions of some of the Hano-Nomizu examples are pictured in Figures \ref{fig:hanonomizu} and \ref{fig:hanonomizu2} (on the right). 

In \cite{Li}, Li showed that for every smooth function $\phi$ on the circle, there exists a hyperbolic surface having null support function $\phi$, which is moreover complete. The results of \cite{Li} actually hold in higher dimension, for hypersurfaces of constant Gauss-Kronecker curvature.  In dimension three, the existence part of Li's results has been improved in \cite{schoenetal}, for Lipschitz continuous $\phi$, and in \cite{Bonsante:2015vi} for $\phi$ lower semicontinuous and bounded. Furthermore, in \cite{Bonsante:2015vi} the first two authors proved that this construction gives a bijection between the set of (convex, which we always assume here) entire hyperbolic surfaces in $\R^{2,1}$ with bounded second fundamental form and the (infinite-dimensional!) vector space of functions $\phi:\mathbb S^1\to\R\cup\{+\infty\}$ having the Zygmund regularity. Since an entire hyperbolic surface with bounded second fundamental form is necessarily complete, the aforementioned results provide another large class of isometrically embedded copies of the hyperbolic plane, non-equivalent to one another.

Finally, in \cite{Bonsante:2019aa}, we characterized \emph{all} entire hyperbolic surfaces in terms of their null support functions, by showing that the same construction gives 
a bijection between the set of \emph{entire} hyperbolic surfaces in $\R^{2,1}$ and the set of lower semicontinuous functions $\phi:\mathbb S^1\to\R\cup\{+\infty\}$ which are finite on at least three points. In the present paper we address the question of characterizing all smooth isometric embeddings of the hyperbolic plane into $\R^{2,1}$. Equivalently, our aim is to determine those lower semicontinuous functions $\phi$ which correspond to \emph{complete} hyperbolic surfaces. 

It is worth pausing to point out, once more, that an entire spacelike surface may be incomplete. Roughly speaking, this may happen if the surface approaches a null direction quickly enough. For instance, in Figure \ref{fig:parabolic} we see an entire hyperbolic surface, obtained as a surface of revolution with respect to a lightlike axis, which is incomplete, being intrinsically isometric to a half-plane in $\Hyp^2$ (\cite[Appendix A]{Bonsante:2015vi}). Its null support function is equal to minus the characteristic function of the point in $\mathbb S^1$ corresponding to the  lightlike axis. 

\begin{figure}[htb]
\centering
\begin{minipage}[c]{.6\textwidth}
\centering
\includegraphics[height=4.5cm]{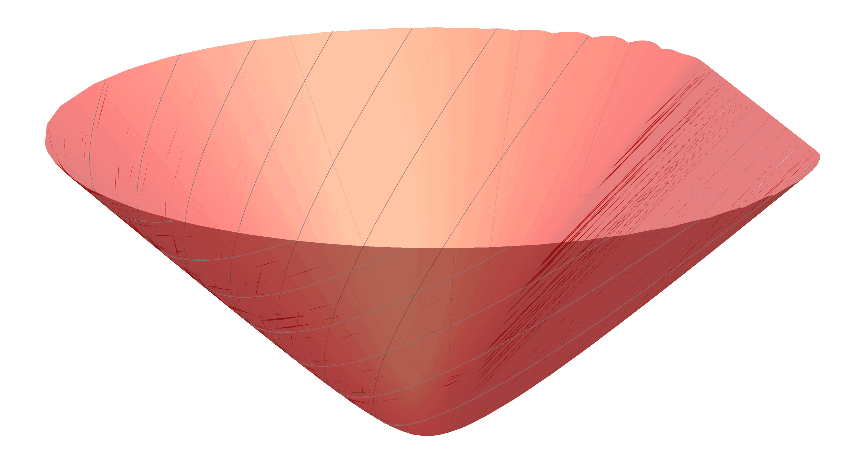}
\end{minipage}%
\begin{minipage}[c]{.4\textwidth}
\centering
\includegraphics[height=4.5cm]{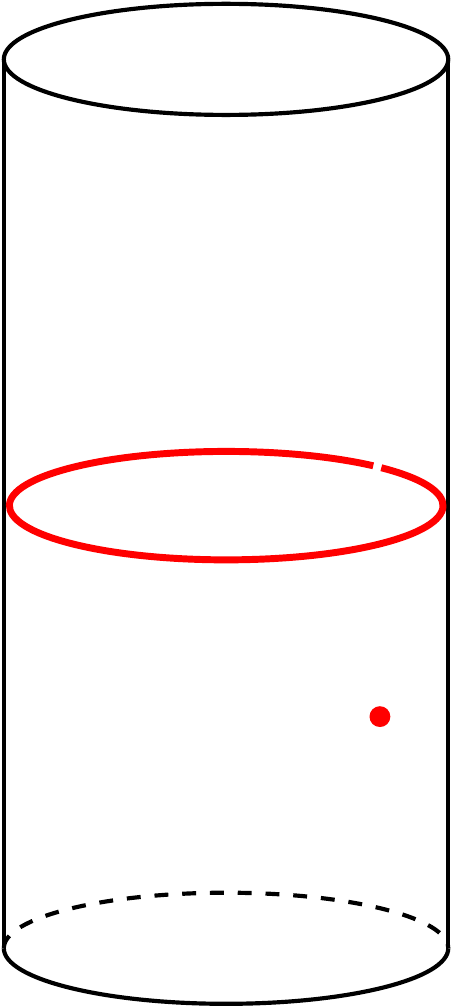}
\end{minipage}%

\caption{The entire hyperbolic surface of revolution with lightlike axis, and its null support function. The surface is incomplete, because it has a Cauchy sequence escaping to the right. \label{fig:parabolic}}

\end{figure}

We remark that in \cite[Corollary E]{bss2}, we completely described the possible intrinsic geometry of entire hyperbolic surfaces, showing that they are isometric to the interior of the convex hull in $\Hyp^2$ of a subset of $\partial_\infty\Hyp^2$, and moreover any such metric can be realized. Let us now state our results.

\subsection*{Main results}

Our first results are two conditions on the null support function which guarantee completeness of a convex entire spacelike surface. Although in the above discussion we focused on \emph{hyperbolic} surfaces, our results hold  more generally for non-positively curved surfaces satisfying certain curvature bounds. 

\begin{thmx}[Sequentially sublinear condition] \label{thm: Lipschitz completeness intro} 
Let $\phi:\mathbb S^1\to\R\cup\{+\infty\}$ be lower semicontinuous and finite on at least three points. Suppose that for each $\theta_0 \in \mathbb S^1$ at which $\phi$ is finite, there exists $M > 0$ and a sequence $\theta_i \to \theta_0$ such that
\begin{equation} \label{eqn: M}
\phi(\theta_i) < \phi(\theta_0) + M |\theta_i - \theta_0|~. \tag{Comp}
\end{equation}
If $\Sigma$ is a convex entire spacelike surface in $\R^{2,1}$ with curvature bounded below and null support function $\phi$, then $\Sigma$ is complete.
\end{thmx}

Throughout, we identify $\mathbb S^1$ with $\R/2\pi\Z$ in the standard way, and define $|\theta - \theta'|$ to be distance in $\R/2\pi\Z$. This is not an essential point, as the statement remains true for any reasonable notion of distance in the circle. 

We remark the condition \eqref{eqn: M} has a simple geometric interpretation in terms of the domain of dependence $\mathcal D$ of the surface $\Sigma$ which, as explained above, is entirely determined by the function $\phi$ on $\mathbb S^1$. Namely,  \eqref{eqn: M} is equivalent to the condition that, whenever $\mathcal D$ has a null support plane $P$ (which is of the form $\langle \vec\theta,\cdot\rangle=\phi(\theta)$ for some  $\theta\in\mathbb S^1$),  there exists a null line in $P$ which does not intersect $\partial\mathcal D$. See Proposition \ref{prop: contact set completeness}  and Corollary \ref{cor: contact set completeness}  for more details.

If we restrict to constant curvature $-1$, Theorem \ref{thm: Lipschitz completeness intro} implies that for any function $\phi$ satisfying \eqref{eqn: M}, the entire hyperbolic surface with null support function $\phi$ (which exists and is unique by the results in \cite{Bonsante:2019aa}) is complete, and hence gives an isometric embeddings of the hyperbolic plane in $\R^{2,1}$.

See Figure \ref{fig:condition1} for a schematic picture of the condition \eqref{eqn: M} of Theorem \ref{thm: Lipschitz completeness intro}. We stress that a function $\phi$ satisfying the condition \eqref{eqn: M} at every point can be highly discontinuous, and can take value $+\infty$ on large portions of the circle. For instance, by virtue of Theorem \ref{thm: Lipschitz completeness intro} the function taking value $0$ on any Cantor set in $\mathbb S^1$, and $+\infty$ elsewhere, is the null support function of a complete hyperbolic surface.

\begin{figure}[htb]
\centering
\includegraphics[height=5cm]{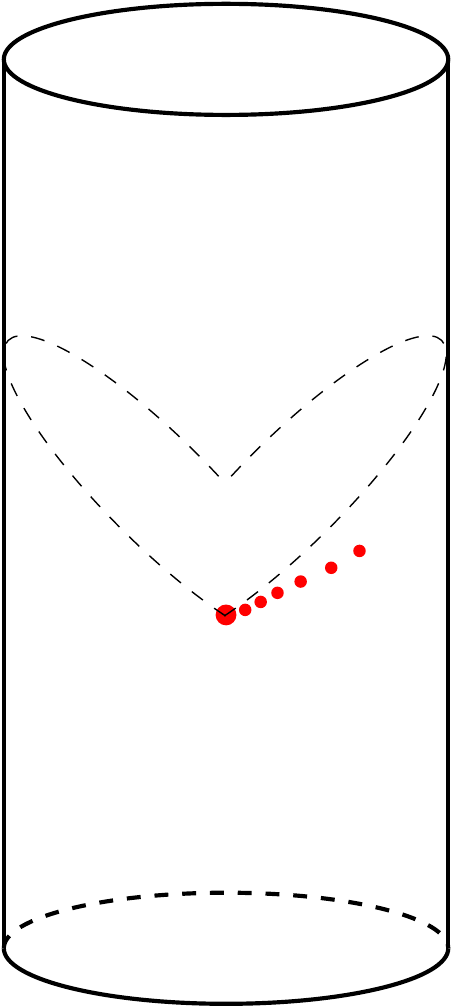}
\caption{The condition \eqref{eqn: M} in Theorem \ref{thm: Lipschitz completeness intro}. \label{fig:condition1}}
\end{figure}

Hence Theorem \ref{thm: Lipschitz completeness intro} is a remarkable improvement with respect to the state-of-the-art, since so far the most general result in this direction, which follows from \cite{Bonsante:2015vi}, is that an entire hyperbolic surface with \emph{Zygmund continuous} null support function is complete.  
Nevertheless, Theorem \ref{thm: Lipschitz completeness intro} is not sharp, as showed by the next statement, which gives another criterion for completeness.

\begin{thmx}[Subloglogarithmic condition] \label{thm: xloglogx completeness}
Let $\phi:\mathbb S^1\to\R\cup\{+\infty\}$ be lower semicontinuous and  finite on at least three points, and let $\lambda>0$. Suppose that that for each $\theta_0 \in \mathbb S^1$ at which $\phi$ is finite, there is a one-sided neighbourhood $U$ of $\theta_0$ and such that
\begin{equation} \label{eqn: lxloglogx}
\phi(\theta) \leq \phi(\theta_0) + \frac{\lambda}{4}|\theta - \theta_0| \log(-\log|\theta-\theta_0|) \tag{Comp'}
\end{equation}
for every $\theta \in U$. If $\Sigma$ is a convex entire spacelike surface in $\R^{2,1}$ with curvature bounded below by $-\lambda^2$ and null support function $\phi$, then $\Sigma$ is complete.
\end{thmx}

By one-sided neighbourhood of $\theta_0$, we mean that $U$ contains an interval either of the form $(\theta_0-\epsilon,\theta_0]$ or $[\theta_0,\theta_0+\epsilon)$ for $\epsilon>0$.

In the other direction, we state now two conditions which are sufficient to guarantee incompleteness.

\begin{thmx}[Power function condition] \label{thm: Holder incomplete} Let $\phi: \mathbb S^1 \to \R \cup \{+\infty\}$ be lower semicontinuous and finite on at least three points. Suppose that there exist $\theta_0 \in \mathbb S^1$ at which $\phi$ is finite, a neighborhood $U$ of $\theta_0$, and constants $\epsilon>0$ and $0 < \alpha < 1$ such that
\begin{equation}\label{eq:inc power}
\phi(\theta) - \phi(\theta_0) > \epsilon|\theta - \theta_0|^{\alpha} \tag{Inc}
\end{equation}
for every $\theta \in U$. If $\Sigma$ is a convex entire spacelike  surface in $\R^{2,1}$ with null support function $\phi$ and curvature bounded above by a negative constant, then $\Sigma$ is incomplete.
\end{thmx}

For example, taking any $\alpha$, the conclusion holds if $\phi$ has a \emph{two-sided jump} at $\theta_0$, meaning that $\phi(\theta_0) < \liminf_{\theta \to \theta_0} \phi(\theta)$, which covers the example of Figure \ref{fig:parabolic}.

\begin{thmx}[One-sided superlogarithmic condition] \label{thm: xlogx incomplete} Let $\phi: \mathbb S^1 \to \R \cup \{+\infty\}$ be lower semicontinuous and finite on at least three points. Suppose that there exist $\theta_0 \in \mathbb S^1$ at which $\phi$ is finite, a neighbourhood  $U$ of $\theta_0$, and $\epsilon > 0$ such that 
\begin{equation} \label{eq:one sided log}\tag{Inc'}
\begin{cases}
\phi(\theta)= +\infty & \text{ if }\theta\text{ is on one side of }\theta_0  \\ 
\phi(\theta)\geq \phi(\theta_0) + \epsilon|(\theta - \theta_0)\log|\theta - \theta_0|| &  \text{ if }\theta\text{ is on the other one side of }\theta_0.
\end{cases}
\end{equation}
for every $\theta \in U$. If $\Sigma$ is a convex entire spacelike surface in $\R^{2,1}$ with null support function $\phi$ and curvature bounded above by a negative constant, then $\Sigma$ is incomplete.
\end{thmx}

\begin{figure}[htb]
\centering
\includegraphics[height=5cm]{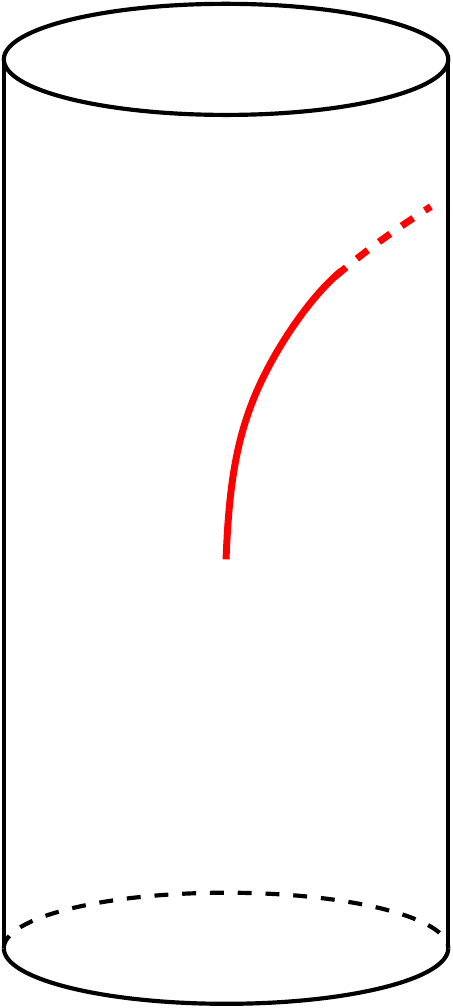}
\caption{The condition \eqref{eq:one sided log} in Theorem \ref{thm: xlogx incomplete}. \label{fig:glide}}

\end{figure}

See Figure \ref{fig:glide} for a schematic picture of the condition \eqref{eq:one sided log} of Theorem \ref{thm: xlogx incomplete}.

Taking together our completeness and incompleteness theorems, we have a narrow window of local behaviors of $\phi$ for which we cannot yet determine completeness, such as $\phi(\theta) = |\theta \log|\theta||$ near $\theta = 0$. From our perspective, the still-open problem of classifying isometric immersions of the hyperbolic plane by their asymptotic behavior amounts to closing this window.

Let us now briefly outline the methods involved in the proofs.

\subsection*{Directional completeness}

The first important step in our investigation consists in establishing a notion of ``directional completeness". We say a convex entire surface $\Sigma$ with null support function $\phi$ is \emph{incomplete at $\theta \in \mathbb S^1$} (Definition \ref{defi:incomplete at theta}) if there is a proper finite-length geodesic $\gamma(t)$ on $\Sigma$ with $\langle \gamma(t), \vec{\theta} \rangle$ converging to $\phi(\theta)$, where we recall that $\vec\theta$ is the null vector $(\cos\theta,\sin\theta,1)$. Geometrically, this means that $\gamma(t)$ is asymptotic to a null support plane which is defined by the equation $\langle \vec\theta,\cdot\rangle=\phi(\theta)$. Otherwise, we say that $\Sigma$ is \emph{complete at} $\theta$.

A first fundamental result about this notion is the following:

\begin{repprop}{prop: exists theta}[Rough version] Let $\Sigma$ be a convex entire spacelike surface in $\R^{2,1}$. If $\gamma: [0,T) \to \Sigma$ is a proper geodesic with finite length, then there exists a unique direction $\theta_0\in\mathbb S^1$ such that 
$\langle \gamma(t), \vec{\theta_0} \rangle$ converges to $\phi(\theta_0)$, and $\phi(\theta_0)$ is finite.
\end{repprop}

In other words, $\gamma(t)$ is asymptotic to a unique null support plane. For the most complete version of Proposition \ref{prop: exists theta}, see Section \ref{subsec:asymptotics finite length}. An immediate consequence of Proposition \ref{prop: exists theta}, together with the Hopf-Rinow theorem, is the following.

\begin{repcor}{thm: incomplete at theta} Let $\Sigma$ be a convex entire spacelike surface in $\R^{2,1}$. If $\Sigma$ is complete at every $\theta\in \mathbb S^1$, then $\Sigma$ is complete.
\end{repcor}

In fact, all the theorems stated above are proved in the paper in a stronger version, in terms of directional completeness. Namely, Theorems \ref{thm: Lipschitz completeness intro}  and \ref{thm: xloglogx completeness} actually state that, if the null support function $\phi$ satisfies the conditions \eqref{eqn: M}  or \eqref{eqn: lxloglogx} \emph{for a fixed }$\theta_0\in\mathbb S^1$, then the convex entire surface $\Sigma$ (satisfying some curvature bounds) is complete \emph{at }$\theta_0$. The statements of Theorems \ref{thm: Lipschitz completeness intro}  and \ref{thm: xloglogx completeness}  are then an immediate consequence using Corollary \ref{thm: incomplete at theta}. Similarly, the conclusion of the stronger version of Theorems \ref{thm: Holder incomplete}  and \ref{thm: xlogx incomplete} is that the surface $\Sigma$ is incomplete \emph{at }$\theta_0$.

Having established such a well-behaved notion of directional completeness, we now explain the main tools in the proofs of our results, keeping in mind that, based on Corollary \ref{thm: incomplete at theta} above, it now suffices  to deal with completeness and incompleteness \emph{in a fixed direction $\theta_0$}. 

\subsection*{Completeness criteria: ``finding long segments"}

In order to prove Theorems \ref{thm: Lipschitz completeness intro}  and \ref{thm: xloglogx completeness}, the rough strategy is the following. Assume $\gamma(t)$ is a proper finite-length geodesic on $\Sigma$ which is asymptotic to a null plane $P$ in the direction of $\theta_0$. We need to use this geodesic in order to find lower bounds for the null support function $\phi$ around $\theta_0$, thus showing the contrapositive of Theorems \ref{thm: Lipschitz completeness intro}  and \ref{thm: xloglogx completeness}. 

Those lower bounds will be achieved by producing certain surfaces contained in the future of the surface $\Sigma$. Indeed, if a surface $\Sigma'$ is in the future of $\Sigma$, then the null support function of $\Sigma'$ is smaller than that of $\Sigma$. Those surfaces $\Sigma'$ that we consider need not  be spacelike, though: roughly speaking, they contain ``lightlike parts". 
For instance, in the proof of Theorem \ref{thm: xloglogx completeness}, these surfaces are the boundary of the future of a spacelike segments, and look exactly like the blue surface that appears in Figure \ref{fig:hanonomizu}. 

More concretely, the strategy involves two main steps:

\begin{enumerate}
\item If $\gamma(t)$ is a proper finite-length geodesic on $\Sigma$ asymptotic to a null support plane $P$, and $L$ is a spacelike line in $P$, then the timelike distance of $\gamma(t)$ from $L$ tends to zero.  (See Lemma \ref{lem: timelike distance}.)
\item If $L$ is a spacelike line containing $\Sigma$ in its future and $\Sigma$ contains a point $p$ whose timelike distance from $L$ is smaller than $\delta$, then one can find a ``long'' spacelike segment parallel to $L$ in the future of $\Sigma$, whose length is roughly $-\log\delta$.  (See Proposition \ref{prop: L tilde}.)
\end{enumerate}

The proof of Theorem \ref{thm: Lipschitz completeness intro} essentially puts together these two ingredients: we produce a spacelike strip in the future of $\Sigma$ which is ``swept-out'' by the segments of step (2), as the point $p$ varies in the tail of the finite-length geodesic $\gamma$. This spacelike strip is very ``wide'' since the distance of $p$ by the line $L$ goes to zero by step (1).

Theorem \ref{thm: xloglogx completeness} is proved by applying similar ideas, and by making those estimates more quantitative. The price to pay is that the ``long segments'' from item (2) that we consider here are constructed over a sequence of points $\gamma(t_n)$ in $\Sigma$, not over the entire tail of $\gamma$. Taking the boundary of the future of these segments provides entire surfaces in the future of $\Sigma$ and leads to a better bound (indeed, the function $|\theta - \theta_0| \log(-\log|\theta-\theta_0|)$ that appears in the condition \eqref{eqn: lxloglogx} is much larger than the linear function $M|\theta - \theta_0|$ that appears in \eqref{eqn: M}), but at the price that condition \eqref{eqn: lxloglogx} must be supposed to hold in a whole one-sided neighbourhood of $\theta_0$, not just along a subsequence as in Theorem \ref{thm: xloglogx completeness}.

\subsection*{Incompleteness criteria: ``finding barriers"}

Let us finally move on to an outline of the proofs of Theorems \ref{thm: Holder incomplete}  and \ref{thm: xlogx incomplete}. In this setting, we need to \emph{assume} a lower bound on the null support function $\phi$ of $\Sigma$ near $\theta_0$, and infer the incompleteness of $\Sigma$ in the direction of $\theta_0$. For both theorems, this is done by means of ``barriers''. 

More precisely, the general tool that we use is the following ``comparison'' statement, that allows to infer, given two surfaces, incompleteness of one from incompleteness of the other.

\begin{repcor}{cor:complete comparison} 
Let $\Sigma_+$ and $\Sigma_-$ be entire spacelike  surfaces in $\R^{2,1}$ with null support functions $\phi_+$ and $\phi_-$, and let ${\theta}\in \mathbb S^1$. Suppose that the curvature functions of $\Sigma_+$ and $\Sigma_-$ satisfy:
\begin{equation}\label{curvature inequality}
K_{\Sigma_-}\leq -C\leq K_{\Sigma_+}\leq 0
\end{equation}
for some constant $C>0$. If $\phi_+$ touches $\phi_-$ from below at ${\theta}$ and $\Sigma_+$ is incomplete at $\theta$, then $\Sigma_-$ is also incomplete at $\theta$.
\end{repcor}

We say that $\phi_+$ \emph{touches $\phi_-$ from below at $\theta\in \mathbb S^1$} if $\phi_+\leq \phi_-$ and $\phi_+(\theta)=\phi_-(\theta) < \infty$.
Corollary \ref{cor:complete comparison}  is essentially a consequence of:
\begin{enumerate}
\item The comparison principle for Gaussian curvature, saying that if the curvature functions of the surfaces $\Sigma_+$ and $\Sigma_-$ satisfy the inequality \eqref{curvature inequality} and $\phi_+\leq \phi_-$, then $\Sigma_+$ is in the future of $\Sigma_-$ (Proposition \ref {prop:comp_principle});
\item The statement that one can find a 1-Lipschitz map from $\Sigma_+$ to $\Sigma_-$, and that, since $\Sigma_+$ and $\Sigma_-$ have the same null support plane $P$ in the direction $\theta$, this 1-Lipschitz map sends a diverging Cauchy surface of $\Sigma_+$ asymptotic to $P$ to a diverging Cauchy surface of $\Sigma_-$, again asymptotic to $P$ (Lemma \ref{prop:contraction}).
\end{enumerate}

Based on Corollary \ref{cor:complete comparison}, we then prove  Theorems \ref{thm: Holder incomplete}  and \ref{thm: xlogx incomplete} using explicit ``barriers". As a first remark, using the hyperbolic surface of Figure \ref{fig:parabolic}, one can immediately deduce that if a convex entire surface $\Sigma$ has curvature bounded above by a negative constant and its null support function has a two-sided jump at a point $\theta_0$, then $\Sigma$ is incomplete at $\theta_0$ (see Corollary \ref{cor easy}). Theorem \ref{thm: Holder incomplete}  is a much stronger statement: it shows that the condition \eqref{eq:inc power}, which is much weaker than having a two-sided jump, implies incompleteness at $\theta_0$.

To prove Theorem \ref{thm: Holder incomplete} we produce a ``barrier", namely an entire spacelike surface whose curvature is non-constant, but is non-positive and bounded below. Moreover, this surface is incomplete in a direction $\theta_0$, around which the null support function behaves like $\epsilon|\theta-\theta_0|^\alpha$. This is done by constructing an explicit support function in the so-called \emph{parabolic coordinates}. For Theorem \ref{thm: xlogx incomplete}, instead, the ``barrier'' that we use has constant curvature and is obtained by imposing the invariance by a one-parameter group of isometries with hyperbolic linear part, similarly to the examples of Hano and Nomizu, except that in this case the translation part is allowed to be  non-trivial. The surfaces obtained by this construction are incomplete (except when the trivial part vanishes, thus recovering the Hano and Nomizu surfaces), and therefore can be successfully applied as the ``barrier" $\Sigma_+$ in Corollary \ref{cor:complete comparison}. 

\subsection*{Organization of the paper}

In Section \ref{sec:prel} we present the necessary preliminaries, and some examples. In Section \ref{sec:finitelengthgeo} we discuss the asymptotic properties of proper geodesics of finite length (including Proposition \ref{prop: exists theta} and Corollary \ref{thm: incomplete at theta}), leading to the notion of directional completeness. In Section \ref{sec:Completeness I}
 we prove Theorem \ref{thm: Lipschitz completeness intro}. In Section \ref{sec:comparison} we develop some comparison tools in order to prove incompleteness, for instance Corollary \ref{cor:complete comparison}. In Section \ref{sec:Completeness II} we prove Theorem \ref{thm: xloglogx completeness}. In Section \ref{sec:Holder incomplete} we prove Theorem \ref{thm: Holder incomplete}. 
In Section \ref{sec:xlogx incomplete} we prove Theorem \ref{thm: xlogx incomplete}.

\section{Preliminaries}\label{sec:prel}

Let $\R^{2,1}$ be the vector space $\R^3$ endowed with the standard Lorentzian metric 
\[
\langle (x, y, z), (x',y',z') \rangle = xx' + yy' - zz'.
\]
 A vector $v\in\R^{2,1}$ is said to be:
\begin{itemize}
\item \emph{spacelike} if $\langle v,v\rangle>0$;
\item \emph{lightlike}, or \emph{null}, if $\langle v,v\rangle = 0$ and $v \neq 0$;
\item \emph{timelike} if $\langle v,v\rangle<0$.
\end{itemize}
and \emph{causal} if it is either timelike, lightlike, or zero. A causal vector $v$ is called \emph{future-directed} if the $z$-coordinate of $v$ is non-negative.

The three-dimensional Minkowski space is the underlying oriented and time-oriented Lorentzian manifold of $\R^{2,1}$.
Its group of isometries  is the semidirect product $\mathrm{O}(2,1)\ltimes\R^{2,1}$, and its identity component $\mathrm{O}_0(2,1)\ltimes\R^{2,1}$ consists of the isometries that preserve both orientation and time-orientation.
We will not distinguish notationally between Minkowski space and the vector space $\R^{2,1}$.

\subsection{Causality for curves and surfaces}\label{subsec:causality}

A curve or surface in $\R^{2,1}$ is a locally Lipschitz regular submanifold. Most of the time, our surfaces will be $C^2$. A curve in $\R^{2,1}$ is called timelike (resp. causal) if every pair of points in it is timelike (resp. causally) separated. We say that a surface $\Sigma$ in $\R^{2,1}$ is:
\begin{itemize}
\item  \emph{achronal} if it meets every timelike curve at most once;
\item \emph{acausal} if it meets every causal curve at most once;
\item \emph{spacelike} if it is $C^1$ and its tangent plane at each point is acausal.
\end{itemize}

Clearly acausal is stronger than achronal. We will see below that if $\Sigma$ is either convex or closed, then spacelike does imply acausal. However, since spacelike is a local condition, it does not in general imply acausal or achronal. 

If $P$ is an affine plane, let $v$ be a nonzero vector orthogonal to the linear plane parallel to $P$. Then $P$ is spacelike if and only if it is acausal, if and only if $v$ is timelike. It is achronal if and only if $v$ is causal. 
An affine plane which is achronal but not acausal (that is, such that $v$ is lightlike) will be called a \emph{null plane}. 

In the following, let $\pi:\R^{2,1}\to\R^2$ be the orthogonal projection to the plane $z=0$.

\begin{remark}\label{rmk locally graphs}
 If $\Sigma$ is an achronal surface, then $\pi|_\Sigma:\Sigma\to\R^2$ is a homeomorphism onto its image. Indeed $\pi|_\Sigma$ is locally injective because $\Sigma$ meets every vertical line at most once, so the conclusion follows from invariance of domain. Hence $\Sigma$ is the graph of a continuous function over a domain in $\R^2$.
 We can thus give the following characterization:
\begin{itemize}
\item $\Sigma$ is achronal if and only if it is the graph of a $1$-Lipschitz function from a subset of $\R^2$ to $\R$;
\item $\Sigma$ is acausal if and only if it is  the graph of a contractive function (that is, it satisfies $|f(x,y)-f(x',y')|< \|(x-x',y-y')\|$);
\item $\Sigma$ is spacelike if and only if it is locally the graph of a $C^1$ function with $\|Df\|<1$. 
\end{itemize}
\end{remark}

\begin{remark}
If $\Sigma$ is spacelike, then its first fundamental form defines a $C^0$ Riemannian metric on $\Sigma$. 
\end{remark}

\begin{defi} We will denote by $J^+$ be the cone of future-directed causal vectors (which includes the zero vector), and by $I^+$ the open cone of future-directed timelike vectors. If $X$ is a subset of $\R^{2,1}$, we will call the set 
\[
J^+(X) = \bigcup_{p \in X} \{q \in \R^{2,1} | q - p \in J^+\}
\]
the \emph{causal future} of $X$, or simply the \emph{future} of $X$. We will call the analogous set $I^+(X)$ the \emph{timelike future} of $X$. The pasts $J^-$, $I^-$, $J^-(X)$, and $I^-(X)$ are defined analogously.
\end{defi}

Note that according to our conventions, $X\subseteq J^+(X)$.

We also introduce the notion of \emph{domain of dependence}, which we use primarily in section \ref{sec: dod}. Since it gives a useful tool for visualizing achronal surfaces, we also draw the domain of dependence, shaded in blue, in Figures \ref{fig:hyperboloid}, \ref{fig:hanonomizu}, \ref{fig:hanonomizu2}.
\begin{defi}\label{defi: dod} The \emph{domain of dependence} of an achronal surface $\Sigma \subset \R^{2,1}$ is the set of points $p \in \R^{2,1}$ such that every causal line through $p$ meets $\Sigma$.
\end{defi}

\subsection{Convexity and support planes}
We now turn our attention to notions concerning convexity, with particular interest in achronal surfaces. The definitions below take advantage of the splitting $\R^{2,1} \cong \R^2 \times \R$, but in fact the notions are all invariant under the  group of  isometries of Minkowski space that preserve the time-orientation.

\begin{defi}\label{def:convex}
An achronal surface $\Sigma$ in $\R^{2,1}$ is \emph{locally (strictly) convex} it is locally the graph of a (strictly) convex function from a subset of $\R^2$ to $\R$. It is \emph{(strictly) convex} if it is the graph of a (strictly) convex function (which is therefore defined on a convex subset of $\R^2$).
\end{defi}

Some authors define a surface as in Definition \ref{def:convex} above to be \emph{future-convex}, since it is convex with respect to the choice of the future-directed vertical vector (or any other future-directed vector). We will simply say (locally) convex to simplify the terminology. 

Let us now introduce support planes and study them in relation with the causality of Minkowski space.

\begin{defi}\label{defi:support plane}
Let $\Sigma$ be a convex achronal surface in $\R^{2,1}$. A spacelike or null plane $P\subset\R^{2,1}$ is a \emph{support plane} for $\Sigma$ if $\Sigma$ is contained in the (causal) future of $P$, but for every $\epsilon > 0$, $\Sigma$ is not contained in the future of $P + (0,0,\epsilon)$.
\end{defi}

\begin{remark}\label{rmk:support plane unique}
From the definition, it follows that a convex achronal surface $\Sigma$ does not admit two distinct parallel support planes. That is, if $P$ is a support plane, then every  plane of the form $P+v$ and distinct from $P$ is not a support plane.
\end{remark}

We remark that a support plane $P$ may or may not intersect $\Sigma$. We say that a support plane $P$ \emph{touches} $\Sigma$ at a point $p$ if $p\in\Sigma\cap P$. By convexity, if $\Sigma$ is a convex surface, for every $p\in\Sigma$ there exists at least one support plane $P$ that touches $\Sigma$ at $p$. Let us now develop this point of view further.

{\begin{prop} \label{prop achronal acausal} 
A convex achronal surface in $\R^{2,1}$ can only be touched by spacelike or null support planes. A convex acausal surface in $\R^{2,1}$ can only be touched by spacelike support planes.
\end{prop}}

\begin{proof} {Let $\Sigma$ be an achronal convex surface, and let $p \in \Sigma$. Since $\Sigma$ is achronal, it does not intersect $I^+(p)$. Since $\Sigma$ is the graph of a convex function on $\R^2$, it is touched from below by the graph of some {affine} function $f(x,y)$ at $p$. To prove the first statement, it suffices to show that the graph of $f$ is not timelike. 

Suppose by contradiction that the graph of $f$ is timelike, and let $R$ be the open ray in $\R^2$ starting from $p$ in the direction of the gradient of $f$. If $\pi: \R^{2,1} \to \R^2$ denotes the vertical projection, then every point of $\pi^{-1}(R)$ is either in $I^+(p)$ or lies below the graph of $f$. Therefore, the restriction $\pi_\Sigma$ of $\pi$ to $\Sigma$ cannot meet $R$. But by Remark \ref{rmk locally graphs}, $\pi_\Sigma$ is a local homeomorphism at $p$. This contradiction shows that $p$ is touched by a spacelike or null support plane.}

{If $\Sigma$ is moreover acausal, then it does not intersect $J^+(p)$ except at $p$, and the same argument with $I^+$ replaced by $J^+$ shows that it can have no null support planes.}
\end{proof}

%

We immediately conclude: 
\begin{cor}\label{cor:spacelike iff acausal} Every $C^1$ convex acausal surface in $\R^{2,1}$ is spacelike.
\end{cor}
\begin{proof} Since $\Sigma$ is $C^1$ and convex, the tangent planes are precisely the support planes that touch $\Sigma$. By Proposition \ref{prop achronal acausal}, each tangent  plane must be spacelike, that is, $\Sigma$ is spacelike. 
\end{proof}

Finally, let us relate the local convexity with the Gauss map and the curvature of the first fundamental form. Recall that the \emph{Gauss map} of a spacelike surface $\Sigma$ is the map $G:\Sigma\to\Hyp^{2}$, where $\Hyp^{2}$ is the one-sheeted hyperboloid consisting of all those future-directed vectors $v$ with $\langle v,v\rangle=-1$, sending $p\in \Sigma$ to the unique $v\in\Hyp^{2}$ normal to $\Sigma$ at $p$. The \emph{shape operator} of $\Sigma$ is the section $B$ of the bundle $\mathrm{End}(T\Sigma)$ defined by $B=P \circ dG$ where $P: T_{G(p)}\Hyp^{2} \to T_p\Sigma$ is the parallel transport in $\R^{2,1}$. The \emph{Gauss equation} in Minkoswki space is the following identity:
$$K_\I=-\det B~,$$
where $\I$ is the first fundamental form of $\Sigma$ and $K_\I$ its curvature.

\begin{prop}\label{prop:Gauss map properties}
Let $\Sigma$ be a spacelike surface in $\R^{2,1}$.
\begin{itemize}
\item If $\Sigma$ is locally strictly convex, then $G$ is a local homeomorphism.
\item If $\Sigma$ is $C^2$, then it is locally convex (up to a reflection in the plane $z=0$) if and only if the first fundamental form is non-positively curved.
\end{itemize}
\end{prop}
\begin{proof}
The first item follows from the definition of locally strictly convex, since in a small neighbourhood $U$, no two points can have the same normal vector. Hence $G$ is locally injective, and a local homeomorphism by the invariance of domain. The second item is a consequence of the Gauss equation, since local convexity/concavity is equivalent to $B$ being non-negative definite. Applying a reflection in the horizontal plane, we can transform a locally concave map into a locally convex one.
\end{proof}

\subsection{Entire surfaces}

In this paper we will study \emph{entire} surfaces in $\R^{2,1}$, which we now define.

\begin{defi} \label{defi entire} 
An achronal surface $\Sigma$ in $\R^{2,1}$ is \emph{entire} if it meets every timelike line exactly once. 
\end{defi}

Equivalently, an achronal surface in $\R^{2,1}$ is entire if and only if it is the graph of a (globally defined) 1-Lipschitz function from $\R^2$ to $\R$. The following lemma shows that it suffices for $\Sigma$ to be closed, and moreover that we can replace achronal by locally achronal (i.e. covered by open achronal sets), which is in general weaker than both achronal and spacelike.

\begin{lemma}[{\cite[Proposition 1.10]{Bonsante:2019aa}}] \label{lem: entire} Let $\Sigma$ be a closed locally achronal surface in $\R^{2,1}$. Then $\Sigma$ is an entire achronal surface.
\end{lemma}


Another easy criterion that ensures entireness is the following:
 
\begin{prop}[{\cite[Lemma 3.1]{Bonsante}}] \label{prop:complete entire} Let $\Sigma$ be a spacelike surface in $\R^{2,1}$ such that the first fundamental form is complete.  Then $\Sigma$ is an entire acausal surface.
\end{prop}
 
The converse of Proposition \ref{prop:complete entire} does not hold, that is, there exist examples of entire spacelike surfaces which are not complete. Those examples can also be produced of constant negative curvature, see Example \ref{ex:parabolic}. This is exactly the starting point of the present work. Our results address necessary and sufficient conditions for an entire spacelike surface satisfying certain curvature bounds (as a  particular case, when it has constant negative curvature) to be complete. 

\begin{remark}\label{rmk convention} 
Recall that, from Section \ref{subsec:causality}, a spacelike surface is supposed to be  $C^1$. When $\Sigma$ is assumed to satisfy certain hypotheses on the curvature, like in the statements of all our main theorems, then we implicitly assume without further mention that $\Sigma$ is $C^2$.

Moreover, all our surfaces will always be  convex and entire. From Proposition \ref{prop:Gauss map properties}, an entire spacelike surface with non-positive curvature is convex up to reflection across a spacelike plane. By our convention of choosing the future-directed normal vector, the second fundamental form, which is defined as $\II(\cdot,\cdot)=\I(B(\cdot),\cdot)$, will always be non-negative definite.
\end{remark}

\subsection{Support functions}

We now focus on convex entire achronal surfaces, which can be described dually by its \emph{homogeneous support function} $s$.

\begin{defi}
Given a convex entire achronal surface $\Sigma$ in $\R^{2,1}$, the \emph{homogeneous support function} of $\Sigma$ is the function $s:J^+\to\R\cup\{+\infty\}$ defined by:
\[
s({v}) = \sup_{{p} \in \Sigma} \langle {p}, {v}\rangle~.
\]
\end{defi}

The homogeneous support function is a proper closed convex function, which means that the supergraph $\{({v}, z) \in J^+ \times \R \,|\, z \geq s({v})\}$ is nonempty, closed, and convex. (We remark that this meaning of proper is unrelated to the topological condition of properness as in Section \ref{sec:finitelengthgeo}.) It is also homogeneous in the sense that $s(\lambda {v}) = \lambda s({v})$ for any $\lambda\geq 0$.

\begin{remark}\label{rmk directional support plane 1}
Fix $v\in J^+\setminus\{0\}$. Then $s(v)=c\in\R$ if and only if the affine plane defined by the equation $\langle \cdot,v\rangle=c$ is a support plane of $\Sigma$. In this case, it is the unique support plane parallel to $v^\perp$, by Remark \ref{rmk:support plane unique}. On the other hand, $s(v)=+\infty$ means that there exists no support plane parallel to $v^\perp$. 
\end{remark}

A standard fact in the theory of convex functions is the following.

\begin{prop} \cite[Proposition 2.4]{Bonsante:2019aa} \label{prop: convex duality} The map sending $\Sigma$ to its homogeneous support function $s$ gives a bijection between the set 
of convex entire achronal surfaces in $\R^{2,1}$ and the set 
of homogeneous proper closed convex functions on $J^+$.
\end{prop}

It is often more convenient to work with a dehomogenized support function. We will use two different dehomogenizations in this paper, either restricting to an affine spacelike plane, which we call an elliptic dehomogenization, or restricting to an affine null plane, which we call a parabolic dehomogenization. 

Let us first introduce the elliptic dehomogenization. Denote by $\Xi:\overline\D\to J^+$ be the inclusion of the closed unit disc at height one, namely:
\begin{equation}\label{eq:inclusion disc}
\Xi(x,y)=(x,y,1)~,
\end{equation}
for $x^2+y^2\leq 1$.

\begin{defi}\label{defi elliptic support}
The \emph{elliptic dehomogenization} $s_\mathrm{ell}:\overline\D\to\R\cup\{+\infty\}$ of the homogeneous support function $s$ is the map:
$$s_\mathrm{ell}(x,y)=s\circ\Xi(x,y)= \sup_{{p} \in \Sigma} \langle {p}, \Xi(x,y)\rangle~,$$
and is called the \emph{elliptic support function}, or simply the \emph{support function}. Its restriction to the unit circle $\mathbb S^1=\partial\D$, that we will always denote by $\phi$, is called \emph{(elliptic) null support function}. 
\end{defi}

Clearly the homogeneous support function $s$ is uniquely determined by $s_\mathrm{ell}$, and it is recovered as the homogeneous extension of $s_\mathrm{ell}$ to $J^+$. As a consequence of Proposition \ref{prop: convex duality}, we have:

\begin{cor} \label{cor: convex duality 1} The map sending $\Sigma$ to its (elliptic) support function $s_\mathrm{ell}$ gives a bijection between the set 
of convex entire achronal surfaces in $\R^{2,1}$ and the set 
of  proper closed convex functions on $\overline \D$.
\end{cor}

We remark that the unit circle $\mathbb S^1=\partial\D$ contains exactly one representative for each null direction in $\R^{2,1}$. 
To simplify the notation, we will identify the unit circle $\mathbb S^1$ in the standard way by $\mathbb S^1 = \R/2\pi\Z$, and write the null support function of $\Sigma$ as $\phi(\theta)$ (where by a small abuse of notation, we write a real number $\theta$ to actually mean the class of $\theta$ modulo $2\pi\Z$). This choice of notation and terminology is consistent with the paper \cite{Bonsante:2019aa}. 

We will also occasionally find it convenient to use the alternative notation 
\begin{equation}\label{eq:vectheta}
\vec{\theta}:=\Xi(\cos(\theta), \sin(\theta))=(\cos(\theta), \sin(\theta),1)~,
\end{equation}
 so that 
\begin{equation}\label{eq: ell null supp fun}
\phi(\theta) = \sup_{p\in\Sigma } \langle p,\vec\theta\rangle~.
\end{equation}


The null support function at a point $\theta$ encodes the null support plane of $\Sigma$ \emph{in a given direction}, a notion that we now define formally and that will be of fundamental importance.

\begin{defi}\label{defi:support plane in direction of}
Let $\Sigma$ be a convex entire achronal surface in $\R^{2,1}$, and let $\theta\in \mathbb S^1$. We say that $P$ is a null support plane \emph{in the direction} $\theta$ if $P$ is a null support plane and the translate of $P$ through the origin is $\vec \theta^\perp$.
\end{defi}

\begin{remark}
By Remark \ref{rmk directional support plane 1}, $\Sigma$ has a (necessarily unique) support plane in the direction of $\theta$ exactly when $\phi(\theta)<+\infty$, and in this case it is the plane defined by the equation $\langle \cdot,\vec\theta\rangle=\phi(\theta)$. 
\end{remark}

\subsection{Parabolic support functions}

The parabolic dehomogenization of the homogeneous support function, depends on a choice of a null direction. In most of the paper, we will choose the direction of the null vector $\vec{\pi} = (-1,0,1)$, but in Section \ref{sec:xlogx incomplete} we will have to make different choices too. 

The idea is to consider the restriction of the homogeneous support function to a null plane which is a translate of the linear plane $\vec\pi^\perp$. However, we choose a particular parametrization of this plane, that we now introduce. 
 Let $\overline{\mathcal{H}} = \{(\mathsf x,\mathsf y){\in\mathbb R^2}\, |\, \mathsf y \geq 0\}$ be the closed upper half-plane, and let $\zeta:\overline{\mathcal{H}}\to\R^{2,1}$ as follows:
\begin{equation}\label{eq:defi p}
\begin{split}
{\zeta}(\mathsf x,\mathsf y) = \begin{bmatrix}  1- \mathsf x^2 - \mathsf y^2 \\ 2\mathsf x \\ 1+ \mathsf x^2 + \mathsf y^2 \end{bmatrix}~.
\end{split}
\end{equation}

\begin{remark}
We will use everywhere the symbols $\mathsf x,\mathsf y$ for the coordinates on $\overline{\mathcal H}$, to distinguish them from the coordinates $(x,y)$ on $\overline\D$ and $(x,y,z)$ on $\R^{2,1}$.
\end{remark}

Note that the image of ${\zeta}$ is the intersection of $J^+$ with the null plane $x+z =2$. This intersection meets every causal line through the origin except for the single null direction $\vec{\pi}$. The map $\zeta$ is well-behaved with respect to the group of linear parabolic isometries fixing $\vec{\pi}$. This is formulated more precisely in the next two remarks.

\begin{remark}\label{rmk:parabolic equivariance}
The map $\zeta$ above has the useful property that the action of the one-parameter group of parabolic isometries of $\R^{2,1}$ fixing $\vec\pi=(-1,0,1)$ corresponds to the action by translation in the $\mathsf x$-coordinate in the upper half-plane $\overline{\mathcal{H}} $. 
Indeed, let $\{L_t\}_{t\in\R}$ the one-parameter parabolic group defined by:
\begin{equation}\label{eq:linear parabolic}
L_t\begin{bmatrix}  -1\\0\\1 \end{bmatrix}=\begin{bmatrix}  -1\\0\\1 \end{bmatrix}\qquad L_t\begin{bmatrix}  0\\1\\0 \end{bmatrix}=\begin{bmatrix}  t\\1\\-t \end{bmatrix}\qquad L_t\begin{bmatrix}  1\\0\\1 \end{bmatrix}=\begin{bmatrix}  1-t^2\\2t\\1+t^2 \end{bmatrix}~.
\end{equation}
Then it is immediate to check that
$
L_t\circ \zeta(\mathsf x,\mathsf y)=\zeta(\mathsf x+t,\mathsf y)$.
\end{remark}

\begin{remark}\label{rmk:metric halfspace}
Moreover, our choice of coordinates ${\zeta}$ has the property that the projection 
$$(\mathsf x,\mathsf y) \mapsto {\zeta}(\mathsf x,\mathsf y) / |{\zeta}(\mathsf x, \mathsf y)|={\zeta}(\mathsf x,\mathsf y)/2\mathsf y\in\Hyp^{2,1}$$
from the upper half-plane to the unit hyperboloid $\Hyp^{2}$ is conformal (by a direct computation we have $|{\zeta}| := \sqrt{-\langle {\zeta}, {\zeta} \rangle} = 2\mathsf y$). Indeed the pull-back of the hyperbolic metric of $\Hyp^{2}$ by this map is the standard metric $(d\mathsf x^2+d\mathsf y^2)/\mathsf y^2$ on the upper half-plane $\mathcal H$.
\end{remark}

We are now ready to introduce the parabolic (null) support function.

\begin{defi}\label{defi elliptic support}
The \emph{parabolic dehomogenization} $s_\mathrm{par}:\overline{\mathcal H}\to\R\cup\{+\infty\}$ of the homogeneous support function $s$ is the map :
\begin{equation} \label{eq: par supp fun}
\begin{split}
s_{\mathrm{par}}(\mathsf x,\mathsf y) =s\circ\zeta(\mathsf x,\mathsf y)= \sup_{{p} \in \Sigma} \langle {p}, {\zeta}(\mathsf x,\mathsf y) \rangle~,
\end{split}
\end{equation}
and is called the \emph{parabolic support function}. Its restriction to $\R=\partial\mathcal H=\{\mathsf y=0\}$, which we will always denote by $\psi$, is called the \emph{parabolic null support function}. 
\end{defi}

Similarly to Equation \eqref{eq: ell null supp fun}, we thus have the following identity for the parabolic null support function: 
\begin{equation} \label{eq: par null supp fun}
\psi(\mathsf x) = \sup_{{p} \in \Sigma} \langle {p}, {\zeta}(\mathsf x,0) \rangle~.
\end{equation}

\begin{remark}
We will sometimes need to specify the direction at infinity, we will say that $s_{\mathrm{par}}$ and $\psi$ are the parabolic (null) support functions \emph{with point at infinity $\vec{\pi}$}. We will occasionally use the terminology ``parabolic (null) support function with point at infinity $\vec\theta_0$" to mean the functions obtained as in \eqref{eq: par supp fun} and \eqref{eq: par null supp fun}, having modified the map $\zeta$ by post-composition with an element of the group $\mathrm{O}(2)$ of time-orientation isometries fixing the $z$-axis, sending $\vec \pi$ to $\vec\theta$.
\end{remark}

The following proposition collects formulas relating the elliptic and parabolic support functions.

\begin{prop} \label{prop: ell/par}
Let $s_{\mathrm{ell}}$ and $s_{\mathrm{par}}$ be the elliptic and parabolic dehomogenizations of the homogeneous support function $s$. Then
\begin{equation}\label{eq:affine vs parabolic}
s_{\mathrm{ell}}\left(\frac{1- \mathsf x^2-\mathsf y^2}{1+ \mathsf x^2+\mathsf y^2},\frac{2\mathsf x}{1+ \mathsf x^2+\mathsf y^2}\right) = \frac{s_{\mathrm{par}}(\mathsf x,\mathsf y)}{1+ \mathsf x^2 + \mathsf y^2} ~.
\end{equation}
In particular, if $\theta \in \mathbb S^1 - \{\pi\}$, the elliptic and parabolic null support functions satisfy:
\begin{equation}\label{eq:null affine vs parabolic}
\phi(\theta) = \frac{\psi(\mathsf x)}{1+ \mathsf x^2} 
\end{equation}
where $\mathsf x = \tan(\theta/2)$. Moreover, 
\begin{equation} \label{eqn: u of infinity}
\phi(\pi) = \liminf_{\mathsf x^2 + \mathsf y^2 \to \infty} \frac{s_\mathrm{par}(\mathsf x,\mathsf y)}{1 + \mathsf x^2 + \mathsf y^2}~.
\end{equation}
\end{prop}
\begin{proof}
The proof of \eqref{eq:affine vs parabolic} is immediate by observing that
$$\zeta(\mathsf x,\mathsf y)=(1+ \mathsf x^2 + \mathsf y^2)\begin{bmatrix}  \frac{1- \mathsf x^2 - \mathsf y^2}{1+ \mathsf x^2 + \mathsf y^2} \\ \frac{2\mathsf x}{1+ \mathsf x^2 + \mathsf y^2} \\ 1 \end{bmatrix}$$
and applying the homogeneity of $s$ and the definitions of $s_{\mathrm{ell}}$ and $s_{\mathrm{par}}$. The proof of \eqref{eq:null affine vs parabolic} then follows from setting $\mathsf y=0$ and noting that 
$$\vec\theta=\begin{bmatrix} \cos(\theta) \\ \sin(\theta) \\ 1  \end{bmatrix}=\begin{bmatrix}  \frac{1- \mathsf x^2}{1+ \mathsf x^2} \\ \frac{2\mathsf x}{1+ \mathsf x^2} \\ 1 \end{bmatrix}$$
for $\mathsf x = \tan(\theta/2)$. Finally, to show \eqref{eqn: u of infinity}, observe that the radial projection from $\zeta(\overline{\mathcal H})$ to the closed unit disc $\overline\D$ at height one induces a homeomorphism between the one-point compactification $\overline{\mathcal{H}}\cup\{\infty\}$ and $\overline\D$. Moreover, since $s_\mathrm{ell}$ is a closed convex function, its value at any boundary point equals the liminf there:
$$s_\mathrm{ell}(\cos(\theta),\sin(\theta))=\liminf_{(x,y)\to (\cos(\theta),\sin(\theta))}s_\mathrm{ell}(x,y)~.$$
Together with \eqref{eq:affine vs parabolic}, we deduce
$$\phi(\pi)=s_\mathrm{ell}(-1,0)=\liminf_{\mathsf x^2 + \mathsf y^2 \to \infty} \frac{s_\mathrm{par}(\mathsf x,\mathsf y)}{1 + \mathsf x^2 + \mathsf y^2}$$
as claimed.
\end{proof}

Since the map $\zeta$ in the definition of the parabolic support function omits the null direction spanned by $\vec\pi$, the parabolic null support function $\psi$ does not quite contain as much information as the (elliptic) null support function $\phi$. However, the previous proposition shows that the missing information, namely the value of $\phi$ at $\pi$, can be recovered from the value of $s_{\mathrm{par}}$ in the interior of $\mathcal H$. In particular, we recover an analog of convex duality for parabolic support functions:

\begin{cor} \label{cor: convex duality 2} \label{lem: convex duality} The map sending $\Sigma$ to its parabolic support function $s_{\mathrm{par}}$ gives a bijection between the set 
of convex entire achronal surfaces in $\R^{2,1}$ and the set 
of  proper closed functions on $\overline {\mathcal{H}}$ convex with respect to the affine structure of the plane $x+z=2$.
\end{cor}

We say that a function $u: \overline{\mathcal{H}} \to \R \cup \{+\infty\}$ is convex with respect to the affine structure of the plane $x+z=2$ if $u\circ \zeta^{-1}:J^+\cap \{x+z=2\}\to\R \cup \{+\infty\}$ is convex.

\begin{proof}
We have to show that, given $u: \overline{\mathcal{H}} \to \{\R \cup +\infty\}$ proper, closed and convex with respect to the affine structure of the plane $x + z = 2$, there is a unique  convex entire achronal surface $\Sigma$ with parabolic support function $u$.
To see this, recall that the affine plane $x+z = 2$ meets every line through the origin in $J^+$ except for $\mathrm{Span}(\vec{\pi})$. Hence we can extend $u$ to a homogeneous function $s$ on $J^+\setminus \mathrm{Span}(\vec{\pi})$. The 1-homogeneous extension $s$ 
is still convex. 
Finally, the function $s$ has a unique closed convex extension to $\langle \vec{\pi}\rangle$, which is defined at the point $(-1,0,1)$ by the liminf as in \eqref{eqn: u of infinity}.
The existence and uniqueness of $\Sigma$ then follow from Proposition \ref{prop: convex duality}.
\end{proof}

\subsection{Examples}

We conclude this preliminary section by two examples of entire spacelike surfaces of constant curvature $-1$, which will be used in the following. 

\begin{example}\label{ex:semitrough}
The first surface we consider is the so-called \emph{semitrough}, which is pictured in Figure \ref{fig:hanonomizu2}. It can be parameterized by the  map $X:(0,+\infty)\times\R\to\R^{2,1}$:
\[
     X(t,s)=\begin{bmatrix}   t -\coth(t) \\ {\sinh(s)}/{\sinh(t)} \\ {\cosh(s)}/{\sinh(t)} \end{bmatrix}~.
\]
One can see immediately that $X$ is invariant under the one-parameter linear \emph{hyperbolic} group $M_s$ fixing $(1,0,0)$ and translating by $s$ along the geodesic $\{z^2-y^2=1\,,y=0\}=\Hyp^2\cap \{x=0\}$. More precisely, $X(t,s+s_0)=M_{s_0}X(t,s)$. It can be shown that $X$ parameterizes an entire spacelike hyperbolic surface $\Sigma$. This surface $\Sigma$ has been first studied by Hano and Nomizu in \cite{hanonomizu}, where it was also shown that it is complete. More precisely,  Hano and Nomizu studied a one-parameter family of complete hyperbolic surfaces in $\R^{2,1}$ which are invariant under the same linear hyperbolic group, which includes $\Sigma$; see also Figure \ref{fig:hanonomizu}.

We omit the proofs of our claims here, since in Section \ref{sec:xlogx incomplete} we will  recover $X$ as part of a different one-parameter family of surfaces,  invariant under one-parameter groups of isometries now with non-zero translation parts. 

The null support function of this entire surface is very simple, namely:
$$\phi(\theta)=\begin{cases} 0 & \textrm{if }\theta\in [-\pi/2,\pi/2] \\
+\infty & \textrm{otherwise} \end{cases}
~.$$
See again Figure \ref{fig:hanonomizu2}. Moreover, the surface is \emph{complete}, which shows that the null support function of complete hyperbolic surfaces in $\R^{2,1}$ need not be finite. \end{example}

It is instructive to consider another example, which is easily expressed in terms of the parabolic support function.

\begin{example}\label{ex:parabolic}
We now consider a family of surfaces which are invariant under the linear \emph{parabolic} group fixing $\vec\pi=(-1,0,1)$, which is defined in \eqref{eq:linear parabolic}. These surfaces have been studied in detail in \cite[Appendix A]{Bonsante:2015vi}. Here we will describe such a surface by means of its parabolic support function, using Corollary \ref{cor: convex duality 2}. By Remark \ref{rmk:parabolic equivariance}, we should consider a parabolic support function of the form 
$$u(\mathsf x,\mathsf y)=f(\mathsf y)~.$$ It turns out that the choice
$$f(\mathsf y)=\int_0^\mathsf y g(\mathsf t)d\mathsf t\qquad g(\mathsf y)=-2\sqrt{1+\epsilon^2 \mathsf y^2}~,$$
for any $\epsilon\geq 0$,
gives rise to an entire spacelike surface of curvature $-1$. For more details, see Section \ref{sec:Holder incomplete}. 
By applying Lemma \ref{lem: u symmetric incomplete}, the corresponding entire spacelike surface provided by Corollary \ref{cor: convex duality 2}  is incomplete unless $\epsilon=0$. In fact, one sees readily that the case $\epsilon=0$ corresponds to the hyperboloid $\Hyp^2$, by definition of parabolic support function applied to $\Sigma=\Hyp^2$:
$$s_{\mathrm{par}}^{\Hyp^2}(\mathsf x,\mathsf y)=\sup_{p\in\Hyp^2}\langle p,\zeta(\mathsf x,\mathsf y)\rangle=\langle \zeta(\mathsf x,\mathsf y),\zeta(\mathsf x,\mathsf y)/|\zeta(\mathsf x,\mathsf y)|\rangle=-|\zeta(\mathsf x,\mathsf y)|=-2\mathsf y~.$$

Finally, we can compute explicitly the elliptic or parabolic null support function, using Proposition \ref{prop: ell/par}. Indeed, for $\mathsf x\in\R$, we immediately get $\psi(\mathsf x)=u(\mathsf x,0)=0$, and therefore  $\phi(\theta)=0$ for all $\theta\neq\pi$, by \eqref{eq:null affine vs parabolic}. To compute $\phi(\pi)$, we can explicitly compute
$$f(\mathsf y)=-\mathsf y \sqrt{1+\epsilon^2\mathsf y^2} - (1/\epsilon)\arcsinh(\epsilon\mathsf y)$$
and therefore $\phi(\pi)=-\epsilon$ by \eqref{eqn: u of infinity}. In conclusion, $\phi$ is a function that is constantly zero on the complement of $\pi$, and has a negative jump (depending on the parameter $\epsilon>0$) at $\pi$. See Figure \ref{fig:parabolic}. We will meet again this family of surfaces in Subsection \ref{subsec:easy}.
\end{example}

\section{Proper finite-length geodesics}\label{sec:finitelengthgeo}

In this section, we begin to develop some tools to understand the geodesic completeness of convex entire spacelike surfaces. Without further mention, we will work with $C^2$ surfaces. The most important result is Proposition \ref{prop: exists theta}, where we associate a unique null direction to each proper finite-length geodesic on any convex entire spacelike surface. We remark that the results in this section use the convexity assumption, but no further assumption on the curvature of the surface (e.g. curvature bounded above or below by negative constants, which are instead assumed in the main results of this paper).

\subsection{Geodesic convexity}

In this subsection we prove that convex entire spacelike surfaces $\Sigma$ in $\R^{2,1}$ are geodesically convex, in the sense that there is a unique geodesic in $\Sigma$ between any two points.

\begin{remark}
When the curvature of $\Sigma$ is a negative constant, the result follows (up to applying a homothety) by \cite[Corollary E]{bss2}, where it has been proved that an entire spacelike surface of curvature $-1$ is isometric to the interior of a domain in the hyperbolic plane bounded by complete geodesics.
\end{remark}

As a tool to study the behavior of geodesics in $\Sigma$, we will use the extrinsic distance. 
Given $p\in\R^{2,1}$ and $q$ in the complement of  $I^+(p)\cup I^-(p)$, we define
\[
d^{\mathrm{ext}}_p(q) = \sqrt{\langle q-p,q-p\rangle}~.
\]
To simplify the notation, we will sometimes write $d^{\mathrm{ext}}_p(q)=|q - p|$.
We first establish some important properties of this function.
\begin{lemma} \label{lem: d ext proper}
Let $\Sigma$ be an entire spacelike surface in $\R^{2,1}$. Fix any point $p \in \Sigma$. Then $d^{\mathrm{ext}}_p:\Sigma\to\R$ is a well-defined proper function on $\Sigma$.
\end{lemma}

\begin{proof} This is Proposition 1 of \cite{chengyaumax}. {Here is a recap of the proof. Write $\Sigma$ as the graph of an entire 1-Lipschitz function $f$, as assume after an isometry that $p$ is the origin and $d f(0,0) = 0$. We will write $x$ for a point in $\R^2$. Using only that $f(x) \leq |x|$, we see that $d^{\mathrm{ext}}_p$ is well-defined. By Taylor's theorem, there is an $\epsilon$ such that $|f(x)| \leq \epsilon / 2$ on the circle of radius $\epsilon$. Since $f$ is 1-Lipschitz, it follows that $|f(x)| \leq |x| - \epsilon/2$ outside the circle of radius $\epsilon$. Therefore for $|x| \geq \epsilon$, $d^{\mathrm{ext}}_p((x, f(x))) = \sqrt{|x|^2 - f(x)^2} \geq \sqrt{\epsilon |x|-\epsilon^2/4}$, and so $d^{\mathrm{ext}}_p$ is proper.}
\end{proof}

\begin{lemma} \label{lem: phi concave}
Let $\Sigma$ be a convex spacelike surface in $\R^{2,1}$ and let $\gamma(t)$ be any  geodesic of $\Sigma$. For any future-directed causal vector $v$, the function
$$
\Phi_v(t) = \langle \gamma(t), v\rangle~.
$$
is concave in $t$.
\end{lemma}
\begin{proof}
The acceleration $\ddot{\gamma}$ in $\R^{2,1}$ is normal to $\Sigma$ because $\gamma$ is a geodesic. Moreover {(recalling our conventions from Remark \ref{rmk convention})} since $\Sigma$ is convex, $\ddot{\gamma}$ must be a nonnegative multiple of the future normal vector to $\Sigma$.  Since the inner product of any two future-directed causal vectors is nonpositive,  we have $\ddot\Phi_v(t) = \langle \ddot{\gamma}(t), v \rangle \leq 0$.
\end{proof}

\begin{lemma} \label{lem: d ext increasing}
Let $\Sigma$ be a convex spacelike surface in $\R^{2,1}$, and let $\gamma(t)$ be any unit speed geodesic of $\Sigma$ with $\gamma(0) = p$. Then for every $t \geq 0$,
\[
\frac{d}{dt} d^{\mathrm{ext}}_p( \gamma(t)) \geq 1~.
\]
\end{lemma}
\begin{proof}
By the chain rule, 
\[
\frac{d}{dt} d^{\mathrm{ext}}_p(\gamma(t)) = \langle \mathrm{grad}~d^{\mathrm{ext}}_p(\gamma(t)), \dot{\gamma}(t) \rangle~.
\]
The gradient in $\R^{2,1}$ of $d^{\mathrm{ext}}_p$ at $q$ is just the unit vector in the direction of $q-p$, so the right hand side is the inner product of two unit spacelike vectors in $\R^{2,1}$. 
If $u$ is a unit spacelike vector, define
\begin{equation} \label{eqn: It}
\begin{split}
I_u :=  \{\theta \in \mathbb S^1 | \langle u, \vec{\theta}\rangle \geq 0 \}~,
\end{split}
\end{equation}
which a nonempty closed interval.

 If ${u}$ and ${v}$ are two unit spacelike vectors, one can check that $\langle {u}, {v} \rangle \geq 1$ exactly when either $I_u \subseteq I_v$ or $I_v \subseteq I_u$. We will prove the lemma by showing that $I_{\dot{\gamma}(t)} \subseteq I_{\mathrm{grad}\,d^{\mathrm{ext}}_p(\gamma(t))}$  for all $t$.

For a given $\theta$, define  
\begin{equation} \label{eqn: phitheta}
\begin{split}
\Phi_\theta(t) = \langle \gamma(t), \vec{\theta}\rangle~,
\end{split}
\end{equation}
which is concave in $t$ by Lemma \ref{lem: phi concave}.

Now, suppose $\theta \in I_{\dot{\gamma}(t)}$. This implies $\dot{\Phi}_\theta(t) \geq 0$.  
By concavity of $\Phi_\theta$, we have:
\[
\Phi_\theta(t) - \Phi_\theta(0) \geq t\dot{\Phi}_\theta(t)\geq 0~.
\]
This implies that $\theta \in I_{\mathrm{grad}\,d^{\mathrm{ext}}_p(\gamma(t))}$ because
\[
\langle \mathrm{grad}~d^{\mathrm{ext}}_p(\gamma(t)), \vec{\theta} \rangle = \langle \frac{\gamma(t)-p}{|\gamma(t)-p|},\vec\theta\rangle= \frac{\Phi_\theta(t) - \Phi_\theta(0)}{|\gamma(t) - p|} \geq 0~.
\]
Hence, $I_{\dot{\gamma}(t)} \subseteq I_{\mathrm{grad}\,d^{\mathrm{ext}}_p(\gamma(t))}$, which concludes the proof.
\end{proof}

For the next lemma, we suppose that $\Sigma$ is both convex and entire. By convexity, its Gauss curvature is non-positive, so no geodesic on $\Sigma$ can have conjugate points, which is to say that the exponential map of $\Sigma$ from any point $p$ is a local diffeomorphism. The next lemma shows that it is also proper.

\begin{lemma} \label{lem: exp proper}
Let $\Sigma$ be a convex entire spacelike surface in $\R^{2,1}$. For any $p \in \Sigma$, the exponential map at $p$ is a proper map from its domain of definition in $T_p \Sigma$ to $\Sigma$.
\end{lemma}

\begin{proof}
Let $U \subset T_p \Sigma$ be the domain of definition of the exponential map, meaning the star-shaped domain consisting of vectors $v$ in $T_p \Sigma$ such that the geodesic through $p$ with derivative $v$ exists for time 1. It suffices to show that the composition $r(v) := d^{\mathrm{ext}}_p \circ \mathrm{exp}_p$ is proper on $U$. 

First, let $u$ be an arbitrary unit vector in $U$ and consider the restriction of $r$ to the maximal ray $\{tu: t \in [0, t_\mathrm{max})\}$. We will show that this is a diffeomorphism from $[0, t_\mathrm{max})$ to $[0, \infty)$. Obviously, its value at $t=0$ is 0. According to Lemma \ref{lem: d ext increasing}, $r(tu)$ is a strictly increasing function of $t$ and, integrating from 0, $r(tu) \geq t$. In the case that $t_\mathrm{max} = \infty$, it follows that $r(tu)$ tends to infinity as $t \to t_\mathrm{max}$, and hence that $r(tu)$ is a diffeomorphism from $[0, t_\mathrm{max})$ to $[0, \infty)$. On the other hand, if $t_\mathrm{max} < \infty$, then the geodesic $\mathrm{exp}_p(tu)$ must be proper in $\Sigma$ since finite length geodesics that stay in a compact set are always extendable. Hence the properness of $d^{\mathrm{ext}}_p$ on $\Sigma$ (Lemma \ref{lem: d ext proper}) also implies in this case that $r(tu)$ tends to infinity as $t \to t_\mathrm{max}$, and the claim follows.

To complete the proof, for any $a \in \R$, we need to show that the sublevel set $\{r(v) \leq a\} \subset U$ is compact in $U$. According to the previous paragraph, there is a unique function $t_a(u)$ on the unit sphere in $T_p\Sigma$ such that $r(t_a u) = a$, and the sublevel set of $r$ is the radial subgraph $\{tu\, |\, t \leq t_a(u) \}$. Moreover, since $r$ is a smooth function on $U\setminus\{0\}$ and its radial derivative is nonzero by Lemma \ref{lem: d ext increasing}, the implicit function theorem shows that $t_a$ is a smooth --- in particular, continuous --- function of $u$. Therefore the sublevel set of $r$ is a compact subset of $U$, completing the proof.
\end{proof}

\begin{prop}\label{prop:geoconv}
Let $\Sigma$ be a convex entire spacelike surface in $\R^{2,1}$. Then $\Sigma$ is geodesically convex.
\end{prop}

\begin{proof}
For any point $p \in \Sigma$, the exponential map $\mathrm{exp}_p$ at $p$ is a local diffeomorphism since $\Sigma$ is non-positively curved, and proper by Lemma \ref{lem: exp proper}, hence it is a covering map. Since $\Sigma$ is simply connected, $\mathrm{exp}_p$ is in fact a diffeomorphism. Therefore, there is a unique geodesic from $p$ to any point $q$ of $\Sigma$. Since $p$ was arbitrary, we conclude that there is a unique geodesic between any two points of $\Sigma$, in other words that $\Sigma$ is geodesically convex.
\end{proof}

\subsection{Asymptotics of finite-length geodesics} \label{subsec:asymptotics finite length}
In the following, recall that $\vec{\theta}$ is the null vector in the direction $\theta$ as defined in \eqref{eq:vectheta}. We now prove the following result that describes the asymptotic behaviour of incomplete proper geodesics on an entire convex spacelike surface.


\begin{defi} \label{defi: asymptotic to} Let $V$ be an affine subspace of $\R^{2,1}$, and let $TV$ be the parallel linear subspace. We say that a curve or a sequence in $\R^{2,1}$ is asymptotic to $V$ if it converges in the quotient $\R^{2,1}/TV$ to the point corresponding to $V$. 
\end{defi}


\begin{propx} \label{prop: exists theta} Let $\Sigma$ be a convex entire spacelike surface in $\R^{2,1}$. If $\gamma: [0,T) \to \Sigma$ is a proper geodesic with finite length, then
\begin{enumerate}
\item There is a unique direction ${\theta}_+\in \mathbb S^1$ such that $\dot{\gamma}$ converges projectively to $\vec{\theta}_+$ as $t\to T$.
\item There is a unique null line $L = \{x + r \vec{\theta}_+, r \in \R\}$ in $\R^{2,1}$ such that $\gamma$ is asymptotic to $L$ in the sense of Definition \ref{defi: asymptotic to}.
\item $\langle \gamma(t),\vec\theta_+\rangle\to\phi(\theta_+) < \infty$ as $t\to T$, where $\phi$ is the support function of $\Sigma$. 
\item $\langle \gamma(t),\vec\theta\rangle\to -\infty$ for all $\theta\neq \theta_+$.
\end{enumerate}
\end{propx}

\begin{proof} It is harmless to assume that $\gamma$ is parameterized by unit speed. For $t \in [0, T)$ and $\theta \in \mathbb S^1$, let
\[
\Phi_\theta(t) = \langle \gamma(t), \vec{\theta}\rangle
\]
as in equation \eqref{eqn: phitheta}. By Lemma \ref{lem: phi concave}, $\Phi_\theta(t)$ is concave in $t$.
 

Since $\dot{\gamma}(t)$ is a spacelike vector, at every time $t$ the set (defined similarly to \eqref{eqn: It} in the proof of Lemma \ref{lem: d ext increasing})
\begin{equation} 
\begin{split}
I_t := \{\theta \in \mathbb S^1 | \dot\Phi_\theta(t) \geq 0 \}= \{\theta \in \mathbb S^1 | \langle \dot\gamma(t), \vec{\theta}\rangle \geq 0 \}
\end{split}
\end{equation}
 is a nonempty closed interval. The concavity in $t$ of $\Phi_\theta$ says that $\dot{\Phi}_\theta$ is non-increasing in time, hence for $t' > t$ we have $I_{t'} \subseteq I_t$. Since the intervals $I_t$ are closed, their intersection is nonempty. We claim that the intersection is a single point. Indeed, if it contained two different points $\theta_1$ and $\theta_2$, then $\vec{\theta_1} + \vec{\theta_2}$ is a timelike vector whose inner product with $\dot{\gamma}$ is non-negative and decreasing with $t$, hence is bounded for all $t$. Then the unit vector $\dot{\gamma}$ must lie between two spacelike planes perpendicular to $\vec{\theta_1} + \vec{\theta_2}$ --- in particular, in a compact subset of $\R^{2,1}$ --- so the finite-length curve $\gamma$ must stay inside a compact subset of $\R^{2,1}$, contradicting properness. We conclude that the intersection of the intervals $I_t$ over all $t$ must be a single point. Let ${\theta}_+$ be this point. Since $\dot\gamma^\perp\cap (\mathbb S^1\times\{1\})$ consists of two vectors that converge to $\vec\theta_+$ as $t\to T$, it follows that $\dot{\gamma}$ converges projectively to $\vec{\theta}_+$, which proves the first statement.

For the second statement, we first observe that since $\Phi_{\theta_+}(t)$ is non-decreasing for all $t \in [0,T)$ and concave, it has a finite limit as $t \to T$. This already implies convergence to some null plane in the direction of ${\theta_+}$. To show that $\gamma$ converges to some line in that plane, choose another point ${\theta}_-$ which lies outside all of the intervals $I_t$. Next, let $e_-$ be the positive multiple of $\vec{\theta}_-$ with $\langle e_-, \vec{\theta}_+ \rangle = -1$; let $e_+= \frac{1}{2}\vec{\theta}_+$; and third let $e_0$ be a (spacelike) unit vector perpendicular to both $e_+$ and $e_-$. Using the basis $\{e_-, {e}_0, e_+\}$, we write $\gamma$ as
\begin{equation}\label{eq:dec basis}
\gamma(t) = a(t) e_- + b(t) e_0 + c(t) e_+~.
\end{equation}
By our choice of basis, we have $a(t) = - \Phi_{\theta_+}(t)$ and $c(t)$ is a negative multiple of $\Phi_{\theta_-}$, hence $\dot{a}(t) \leq 0 < \dot{c}(t)$ for all $t$. By our hypotheses on the basis $\{e_-, {e}_0, e_+\}$,
\[
|\dot{\gamma}(t)|^2 = -\dot{a}(t)\dot{c}(t) + \dot{b}(t)^2~,
\]
which is equal to one since we have taken $\gamma$ to be unit speed.  Hence $\dot{b}(t)^2 \leq 1$ for all $t$, so the coordinate $b(t)$ has a finite limit as $t \to T$. Therefore, $\gamma$ converges to a null line parallel to $\vec{\theta}_+$. 

To prove the third statement, we just need to show that the unique null plane containing the line $L$ is a support plane for $\Sigma$, since we have shown that $\gamma$ lies in the future of this plane and is asymptotic to it. For each time $t$, the spacelike surface $\Sigma$ lies in the future of the past null cone $J^-(\gamma(t))$. Since $\gamma$ is asymptotic to $L$ and $c(t)$ tends to infinity, the null cones $J^-(\gamma(t))$ converge in compact subsets to $J^-(L)$, which is a half-space whose boundary is the null plane containing $L$. Hence $\Sigma$ lies in the future of this null plane. But it cannot lie in the future of any parallel null plane in the future of $L$ since $\gamma$ is asymptotic to $L$. Therefore, the null plane containing $L$ is the null support plane for $\Sigma$ in the direction of $\theta_+$. 

Finally, to prove the fourth statement, note that since $\gamma$ is proper and the first two coordinates are bounded, the third coordinate, $c(t)$ must diverge, and since it is increasing, it diverges to positive infinity. Since $c(t)$ is a negative multiple of $\Phi_{\theta_-} = \langle \gamma(t), \vec{\theta}_-\rangle$, this is equivalent to
\begin{equation}\label{eqref:coordinatediverge}
\lim_{t\to T}\langle \gamma(t),{\vec\theta_-}\rangle=-\infty~.
\end{equation}
Now observe that the only assumption we made on $\theta_-$ was that it is contained in the complement of all the intervals $I_t$. But since the intersection of the intervals $I_t$ equals $\{\theta_+\}$, every $\theta\neq \theta_+$ satisfies this assumption up to restricting $\gamma$ to $(T-\epsilon,T)$. Hence \eqref{eqref:coordinatediverge} holds for any choice of $\theta_- \neq \theta_+$.

\end{proof}

In the same setting as Proposition \ref{prop: exists theta}, we now show that the timelike distance from $\gamma$ to any spacelike line contained in its asymptotic null plane converges to zero. This will be important in the proof of Theorem \ref{thm: Lipschitz completeness intro}, in the next section.

\begin{lemma}\label{lem: timelike distance} Let $\Sigma$ be a convex entire spacelike surface in $\R^{2,1}$, and suppose $\gamma: [-T,0) \to \Sigma$ is a proper finite length geodesic in $\Sigma$ asymptotic to a null support plane $P$ of $\Sigma$ in the sense of Definition \ref{defi: asymptotic to}. Let $L$ be any spacelike line in $P$. Then  as $t \to 0$, the timelike distance from $L$ to $\gamma(t)$ tends to zero. 
\end{lemma}

\begin{proof}

Choose a basis $\{e_-, e_0, e_+\}$ of $\R^{2,1}$ such that $e_+$ is a future null vector tangent to $P$, $e_0$ is tangent to $L$, and $e_-$ is a future null vector with
\begin{align*}
\langle e_+, e_-\rangle&= -1/2 \\
\langle e_0, e_0\rangle&=1 \\
\langle e_\pm, e_0\rangle &=0.
\end{align*}
Suppose furthermore that $L$ goes through the origin, up to applying a translation. Write
\begin{equation} \label{eqn: nullnull}
\begin{split}
\gamma(t) = a(t) e_- + b(t) e_0 + c(t) e_+.
\end{split}
\end{equation}

Observe that $a(t)=-2\langle \gamma(t),e_+\rangle=-2\Phi_{e_+}(t)$ and similarly $a(t)=-2\Phi_{e_-}(t)$, hence by Lemma \ref{lem: phi concave} $a(t)$ and $c(t)$ are convex functions.
As in the proof of Proposition \ref{prop: exists theta},
the function $a(t)$ is positive, {nonincreasing} and (by the assumption that $L$ contains the origin) converges to 0 as $t \to 0$. 
Since the graph of $a$ on the interval $[t,0]$ lies below the linear function interpolating its values at the endpoints, we have that 
\begin{equation} \label{eqn: adot}
\begin{split}
\dot{a}(t) \leq - a(t)/|t|
\end{split}
\end{equation}
for all $t < 0$.

Since $a(t)$ is bounded as $t \to 0$ and $b(t)$ is bounded by Proposition \ref{prop: exists theta}, in order for $\gamma$ to be proper we must have $\lim_{t \to 0} c(t) =+ \infty$. Assume from here on that $t$ is close enough to zero such that 
\begin{equation} \label{eqn: cdot}
\begin{split}
c(t) > 0 \textrm{ and } \dot{c}(t) > 0. 
\end{split}
\end{equation}
Then $\gamma(t)$ is in the future of $L$, and the timelike distance to $L$ is $\sqrt{a(t)c(t)}$. 

Assuming without loss of generality that $\gamma$ is unit speed,
\[
1=|\dot{\gamma}(t)|^2 = - \dot{a}(t)\dot{c}(t) + \dot{b}(t)^2~,
\]
and in particular $-1 \leq \dot{a}\dot{c}$. 
Therefore
\begin{align*}
\frac{d}{dt}\left(|t| + \frac{a(t)c(t)}{|t|} \right) &= -1 + \frac{\dot{a}c}{|t|} + \frac{a\dot{c}}{|t|} + \frac{ac}{t^2} \\
&\leq \dot{a}\dot{c} + \frac{\dot{a}c}{|t|} + \frac{a\dot{c}}{|t|} +  \frac{ac}{t^2}\\ 
&= \left(\dot{a} + \frac{a}{|t|}\right)\left(\dot{c} + \frac{c}{|t|}\right) \\
& \leq 0
\end{align*}
where the last inequality follows from \eqref{eqn: adot} and \eqref{eqn: cdot}. Hence $|t| + ac/|t|$ is bounded as $t \to 0$, and so the timelike distance $\sqrt{a(t)c(t)}$ tends to zero. 
\end{proof}

{\begin{remark}
Although we will not use a more quantitative version of Lemma \ref{lem: timelike distance}, the proof shows that if the proper geodesic $\gamma:[-T,0)\to\Sigma$ is parameterized by unit speed, then the timelike distance from $L$ to $\gamma(t)$ is $O(\sqrt{|t|})$.
\end{remark}}

\subsection{Directional completeness}
In light of Proposition \ref{prop: exists theta}, it is natural to give the following definition (recall also Definition \ref{defi:support plane in direction of}).

\begin{defi} \label{defi:incomplete at theta} A convex entire  spacelike surface $\Sigma$ in $\R^{2,1}$ is \emph{incomplete at $\theta \in \mathbb S^1$} if there is a proper finite-length geodesic on $\Sigma$ asymptotic to a null support plane in the direction of $\theta$.  Otherwise, we say that $\Sigma$ is \emph{complete at} $\theta$.
\end{defi}

\begin{remark}\label{rmk:infinite support complete} According to our definition, if $\Sigma$ admits no null support plane in the direction of $\theta$, then $\Sigma$ is complete at $\theta$. 
\end{remark}

The following is an immediate consequence of Proposition \ref{prop: exists theta} and the Hopf-Rinow theorem.

\begin{corx} \label{thm: incomplete at theta} Let $\Sigma$ be a convex entire spacelike surface in $\R^{2,1}$. If $\Sigma$ is complete at every $\theta\in \mathbb S^1$, then $\Sigma$ is complete.
\end{corx}

\begin{proof} The Hopf-Rinow theorem says that if $\Sigma$ is incomplete, it has a proper finite length geodesic. By statements (2) and (3) of Proposition \ref{prop: exists theta}, this geodesic must be asymptotic to some null support plane of $\Sigma$.  Hence $\Sigma$ is incomplete at some direction $\theta$.
\end{proof}


In the following lemma, we show that one could alternatively define completeness at $\theta$ in terms of Cauchy sequences. As a corollary, it is not necessary to require that the proper finite-length curve is a geodesic in Definition \ref{defi:incomplete at theta}. This will be useful below when we want to show (see the proof of Proposition \ref{lem:complete comparison} in Section \ref{sec:comparison}) that a surface is incomplete using a short map from another incomplete surface.

\begin{lemma}\label{lemma:incompleteness by cauchy}
Let $\Sigma$ be a convex entire spacelike surface in $\R^{2,1}$ and let $\theta\in \mathbb S^1$. Suppose that there exists a diverging Cauchy sequence on $\Sigma$ asymptotic to a null support plane in the direction of $\theta$. Then $\Sigma$ is incomplete at $\theta$.
\end{lemma}
\begin{proof}
Suppose that $p_n$ is a diverging Cauchy sequence asymptotic to a null support plane $P$. We need to show that there exists a proper finite-length geodesic asymptotic to $P$. Let $\Phi_{{\theta}}(p) = \langle\vec{\theta}, p\rangle$. By the last statement of Proposition \ref{prop: exists theta}, it is sufficient to find a proper finite-length geodesic on which $\Phi_{{\theta}}$ is bounded below. 

 Fix a point $p_0$ in $\Sigma$ and let $\gamma_n: [0,t_n] \to \Sigma$ be the sequence of unit speed geodesic segments from $p_0$ to $p_n$. Let $T_n$ be their tangent vectors at $p_0$, and let $T_0$ be a subsequential limit of the $T_n$. Let $\gamma: [0, t_0) \to \Sigma$ be the maximal geodesic segment starting at $p_0$ with tangent vector $T_0$. The main point is to prove that $t_0 \leq \liminf_{n \to \infty} t_n$.
 
Let $\Omega \subset T_{p_0} \Sigma$ be the domain of the exponential map at $p_0$. Since $\exp_{p_0}: \Omega \to \Sigma$ is a diffeomorphism by Proposition \ref{prop:geoconv} and $p_n = \exp(t_n T_n)$ is proper in $\Sigma$, the sequence $t_n T_n$ is proper in $\Omega$, so any limit point will lie outside of $\Omega$. Hence the maximal existence time of the geodesic in the limiting direction is less than or equal to the norm of this limit point.

It follows that $\gamma$ is finite length, and proper. To show that $\Phi_{{\theta}}$ is bounded on $\gamma$, first note that since $p_n$ converges to $P$, there is a constant $C$ such that $\Phi_{{\theta}}(p_n) \geq -C$ for all $n$. Next, the concavity of $\Phi_{{\theta}} \circ \gamma_n$ (Lemma \ref{lem: phi concave}) implies that $\Phi_{{\theta}} \circ \exp_{p_0}$ is bounded below by $\min(\Phi_{{\theta}}(p_0), -C)$ on each segment $[0, t_nT_n]$.  Since $\Phi_{{\theta}} \circ \exp_{p_0}$ is continuous on $\Omega$, it has the same bound on the segment $[0, t_0T_0]$. This completes the proof.
\end{proof}

\section{Completeness I: sequentially sublinear condition}\label{sec:Completeness I}

In this section we will prove the first completeness criterion, namely Theorem \ref{thm: Lipschitz completeness intro}. 
We  state it here in a local form, namely in terms of completeness in a given direction $\theta_0$. The statement given in the introduction then follows immediately by Corollary \ref{thm: incomplete at theta}.

\begin{reptheorem}{thm: Lipschitz completeness intro}[Sequentially sublinear condition -- local version] Let $\phi: \mathbb S^1 \to \R \cup \{+\infty\}$ be lower semicontinuous and finite on at least three points. Suppose $\theta_0 \in \mathbb S^1$ is such that $\phi(\theta_0)<+\infty$ and there exists $M > 0$ and a sequence $\theta_i \to \theta_0$ such that
\begin{equation} \label{eqn: M2}
\phi(\theta_i) < \phi(\theta_0) + M |\theta_i - \theta_0|~. \tag{Comp}
\end{equation}
If $\Sigma$ is a convex entire spacelike surface in $\R^{2,1}$ with curvature bounded below and null support function $\phi$, then $\Sigma$ is complete at $\theta_0$.

In particular, if condition \eqref{eqn: M2} holds for every $\theta_0$ at which $\phi$ is finite, then $\Sigma$ is complete.
\end{reptheorem}

The two fundamental properties that we will use in the proof are Lemma \ref{lem: timelike distance} and Proposition \ref{prop: L tilde}, which is proved in Section \ref{subsec:quant est} and in turn relies on Lemma \ref{lem:strip} below. The proof of Theorem \ref{thm: Lipschitz completeness intro} is then completed in Section \ref{subsec:proofA}.

\subsection{An application of the comparison principle}

Let us first quickly recall the (finite) comparison principle.

\begin{prop}\label{prop:finite comparison}
Let $\Omega\subset\R^2$ be an open domain with compact closure, and let $\Sigma_\pm=\mathrm{graph}(f_\pm)$ be spacelike surfaces in $\R^{2,1}$, for $f_\pm\in C^2(\Omega)\cap C^0(\overline\Omega)$. Suppose that the curvature functions of $\Sigma_+$ and $\Sigma_-$ satisfy:
$$K_{\Sigma_-}\leq -C\leq K_{\Sigma_+}\leq 0$$
for some constant $C>0$. If $f_-<f_+$ on $\partial\Omega$, then $f_-<f_+$ on $\Omega$.
\end{prop}
\begin{proof}
Let us first assume that the curvature of $\Sigma_+$ is strictly larger than $-C$. Assume that $f_+-f_-$ takes nonpositive values on $\Omega$, and let $x_{\min}$ be its minimum point, which is in $\Omega$ since $f_+-f_-$ is positive on $\partial\Omega$. Translating $\Sigma_+$ vertically by $f_+(x_{\min})-f_-(x_{\min})$, we obtain a new surface $\Sigma_+'$ which is tangent to $\Sigma_-$ over $x_{\min}$ and contained in the future of $\Sigma_-$. But this contradicts the condition that the curvature of $\Sigma_-$ is strictly larger in absolute value than that of $\Sigma_+$. 

In the general case, namely when we only assume $K_{\Sigma_+}\geq-C$, take $\delta>0$ so that $f_-+\delta\leq f_+$ on $\partial\Omega$.
Then applying a homothety of a factor  $1+\epsilon$ to $\Sigma_+$, the previous case applies. Take the limit as $\epsilon\to 0$, we conclude that $f_-+\delta\leq f_+$ on $\Omega$, so that
$f_-< f_+$ as claimed.
\end{proof}

By applying  the comparison principle, we now prove a fundamental estimate.

\begin{lemma}\label{lem:strip} Let $\Sigma=\mathrm{graph}(f)$ be an entire spacelike surface in $\R^{2,1}$ of curvature $\geq -1$. Suppose that $f(x,y) \geq h>0$ on the boundary of the half-infinite strip $[0,\infty) \times (-h, h)$. Then for $r$ sufficiently large (depending on $h$),  
\begin{equation}\label{eq:inequality z1}
f(r,0) \geq \frac{he^{-2r}}{4}~. 
\end{equation}
\end{lemma}

\begin{proof}
For each value of $r$, we will find a semitrough $\Sigma_-$ whose height on the boundary of the the half-infinite strip (and therefore on the entire half-infinite strip, by convexity) is at most $h$, and whose height above the point $(r,0)$ is approximately $he^{-2r}/2$ up to lower order terms. We would like to conclude by the comparison principle (Proposition \ref{prop:finite comparison}) on the half-infinite strip that $\Sigma$ lies above this semitrough, but we cannot apply it directly since the strip is not bounded. 

Fortunately, there is an easy fix. Apply a $\delta$-Lorentz boost (for $\delta>0$ small) around the $y$-axis, namely the linear isometry
$$L_\delta=\begin{pmatrix}\cosh(\delta) & 0 & \sinh(\delta) \\ 0 & 1 & 0 \\ \sinh(\delta) & 0 & \cosh(\delta)
\end{pmatrix}~,$$
that sends $\Sigma$ to a surface $\Sigma^\delta=L_\delta(\Sigma)$. Let us call $f_\delta$ the function whose graph is $\Sigma^\delta$. Since $L_\delta$ does not decrease the height of all points whose first coordinate  is non-negative, $f_\delta\geq f$ on the half-infinite strip, and therefore  $f_\delta\geq h$ on the boundary of the half-infinite strip. 

Moreover, observe that $f>0$ on the half-infinite strip because it is at least $h$ on its boundary and 1-Lipschitz. Hence the portion of the surface $\Sigma^\delta$ over the half-infinite strip is in the future of the plane of equation $x\sinh(\delta)-z\cosh(\delta)=0$, and therefore $f_\delta>h$ on the segment of the half-infinite strip defined by $x=h/\tanh(\delta)$. We can then apply the standard comparison principle on the truncation $x\leq h/\tanh(\delta)$ of the half-infinite strip to see that $\Sigma^\delta$ lies above the semitrough, and take the limit as $\delta \to 0$ to conclude that $\Sigma$ does as well.

Given a point $(r,0)$ along the positive $x$ axis, it remains only to find a semitrough satisfying the two conditions above. For some $d,\epsilon$ to be given as functions of $r$, let $X_{d,\epsilon}(s,t)$ be a parametrized semitrough given by
\[
X_{d,\epsilon}(s,t) = \begin{bmatrix} t - \coth(t) - d + 1 \\  \sinh(s)/\sinh(t) \\ \cosh(s)/\sinh(t) - \epsilon   \end{bmatrix}
\]
for $s\in\R$ and $t\in(0,+\infty)$. (Observe that $X_{d,\epsilon}(s,t)$ is simply a translate of the semitrough introduced in Example \ref{ex:semitrough}, where $d$ and $\epsilon$ are the parameters of horizontal and vertical translation.)
In order to guarantee that the height of the semitrough along the boundary of the half-infinite strip is at most $h$, it is sufficient for the surface $X_{d,\epsilon}$ to contain the point $(0,h,h)$. Indeed, since the semitrough is symmetric by reflection in the plane $y=0$, if it contains the point $(0, h, h)$ then it also contains $(0, -h, h)$, hence by convexity its height over $\{0\}\times (-h,h)$ is $\leq h$. Moreover, by direct inspection one sees that the height function of the semitrough is decreasing on $\{(x,h)\,|\,x\in\R\}$ as a function of $x$. Observe also that for any $d\in\R$ there exists a unique $\epsilon$ such that $X_{d,\epsilon}$ contains the point $(0, h, h)$, since the vertical translates of an entire graph foliate $\R^{2,1}$. Moreover $\epsilon$ is positive because $X_{d,0}$ is always contained in the future of the $x$-axis. We begin by giving an asymptotic formula for this $\epsilon$ as a function of $d$, valid for large $d$.

By assumption, there is a point $(s_0, t_0)$ such that $X_{d,\epsilon}(s_0,t_0) = (0, h, h)$. Comparing $x$-coordinates, we have
\[
d = t_0 - \coth{t_0} + 1
\]
which we can invert to find
\[
t_0 = d + O(e^{-2d})~.
\]
We can eliminate $s_0$ from the two equations for the $y$- and $z$- coordinates using $\cosh^2 - \sinh^2 = 1$ to find
\[
(h+\epsilon(d))^2 - h^2 = 1/\sinh^2(t_0)~.
\]
Next, taking the positive solution of the quadratic equation for $\epsilon$ and Taylor approximating $\sinh$ yields
\begin{equation}\label{eq:formula epsilon}
\epsilon(d) = 2 e^{-2t_0}/h + O(e^{-4t_0}) = 2 e^{-2d}/h + O(e^{-4d})~.
\end{equation}
Here, and for the remainder of the proof, the implicit constant in $O()$ is allowed to depend on $h$.

The next step is to choose $d(r)$ in order that the height of $X_{d(r), \epsilon(d(r))}$ above the point $(r,0)$ is rather large, approximately $he^{-2r}/2$. For any $d$ and $r$, let $t_1$ be such that $X_{d, \epsilon(d)}(0,t_1) = (r, 0, z_1)$ for some height $z_1$. We now use the same Taylor approximations as above. Comparing $x$-coordinates gives
\[
t_1 = d + r + O(e^{-2(d+r)})~,
\]
and comparing $z$-coordinates, plugging in the formula \eqref{eq:formula epsilon} for $\epsilon(d)$, gives
\[
z_1 = 2 e^{-t_1} + O(e^{-3t_1}) - 2 e^{-2d}/h + O(e^{-4d})
\]
and hence
\[
z_1 = 2 e^{-(d+r)} - 2 e^{-2d}/h + O(e^{-3(d+r)}) + O(e^{-4d})
\]
To make sense of the big-O notation here, recall that we will take $d$ to be a function of $r$. Specifically, we choose $d(r)$ to maximize the highest-order part of $z_1$ (that is to say, everything except for the big-O terms). Doing so, we find that we should take $d(r) = r - \log(h/2)$, which yields
\[
z_1 = h e^{-2r}/2 + O(e^{-4r})~.
\]
It follows that for $r$ large enough, $z_1 \geq h e^{-2r}/4$. The lemma then follows by the comparison principle.
\end{proof}

\subsection{Finding long segments in the future}\label{subsec:quant est}

The following proposition, roughly speaking, serves to show that if a convex entire spacelike surface $\Sigma$ with curvature bounded below is close to a spacelike line $L$ and contained in the future of $L$, then the future of $\Sigma$ contains a long segment parallel to $L$.

\begin{prop} \label{prop: L tilde} Let $\Sigma$ be a convex entire spacelike surface in $\R^{2,1}$ of curvature $\geq -1$, and $L$ a spacelike line containing $\Sigma$ in its future. For any $h > 0$, there exists $\epsilon > 0$ such that for any $\delta < \epsilon$, if there is a pair of points $p \in L$ and $q \in \Sigma$ with timelike distance $\mathrm{dist}(p,q) \leq \delta$, then the spacelike segment $\tilde{L}$ of length $\frac{1}{2}\log(h/4\delta)$ parallel to $L$ at a distance $h$ bisected by the ray from $p$ to $q$ is contained in the future of $\Sigma$ (Figure \ref{fig: L tilde}).

\end{prop}

\begin{figure}[htb]
\centering
\includegraphics[height=6cm]{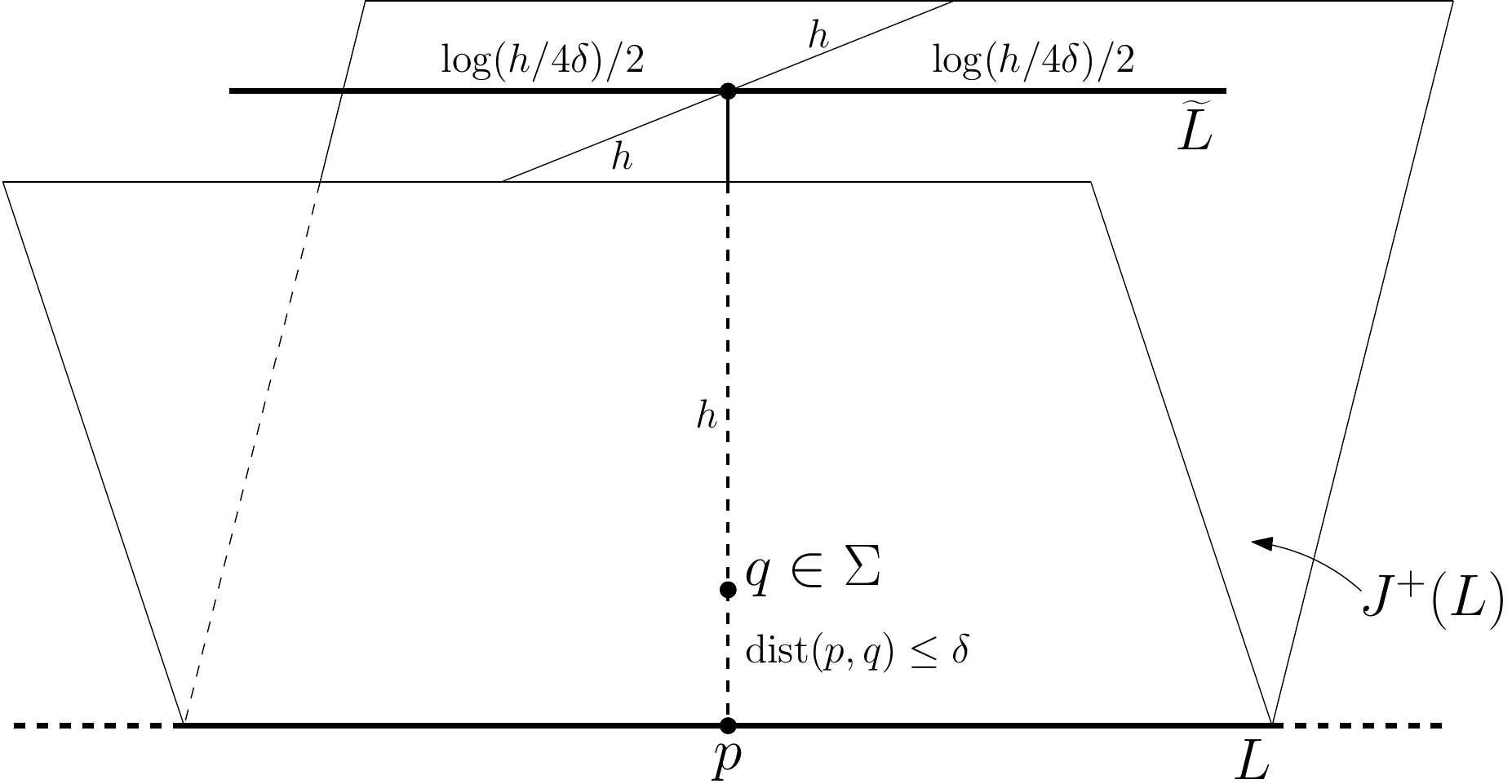}
\caption{The setting of Proposition \ref{prop: L tilde}. The statement says that, if the  convex entire surface $\Sigma$ goes through $q$ and is in the future of $L$, then the (long) parallel line segment $\widetilde L$ is in the future of $\Sigma$. \label{fig: L tilde}}
\end{figure}

\begin{proof}
{Fix $h>0$. Let us choose $r_0>0$ in such a way that inequality \eqref{eq:inequality z1} holds in Lemma \ref{lem:strip} for $r\geq r_0$. We will take $\epsilon=he^{-2r_1}/4$, for some $r_1\geq r_0$ to be chosen at the end of the proof. Any $\delta<\epsilon$ can be written as $\delta=he^{-2r}/4$ for some $r>r_1$. We need to prove, assuming $\mathrm{dist}(p,q) \leq \delta$, that the the spacelike segment $\tilde{L}$ of length $r$ parallel to $L$ at a distance $h$ bisected by the ray from $p$ to $q$ is contained in the future of $\Sigma$.}

Up to an isometry of $\R^{2,1}$, we may suppose that $L$ is the $x$-axis, $p$ is the point $(r,0,0)$, and $q = (r, 0, z)$ for $z \leq \delta$. Then $\tilde{L}$ is the segment from $(r/2, 0, h)$ to $(3r/2, 0, h)$. Let $R$ be the intersection of $J^+(\Sigma)$ with the spacelike plane $z=h$, or equivalently the convex hull of the intersection of $\Sigma$ with this plane. Since $\Sigma$ is contained in the future of $L$, $R$ is contained in the infinite strip $(-\infty, \infty) \times (-h, h) \times \{h\}$. However, if $R$ were contained in the half-infinite strip $[0, \infty) \times (-h, h) \times \{h\}$, then the height of $\Sigma$ would be at least $h$ on the boundary of this strip, and we would satisfy the conditions of Lemma \ref{lem:strip}; this gives a contradiction for $r>r_0$ since we have assumed $z \leq \delta=he^{-2r}/4.$ Therefore, there must be a point $w = (x,y,h) \in R$ with $x < 0$ and $y \in (-h, h)$.

Next, we use the fact that $w \in R$ to show that, for $r$ sufficiently large, half of the line $\tilde{L}$ is also in $R$. Since $\Sigma$ is spacelike and contains the point $q = (r, 0 ,z)$, the disk 
$I^+(q) \cap \{z=h\}$
 must also be contained in $R$. This horizontal disk is centered at $(r,0, h)$ with radius at least $h(1 - e^{-2r}/4)$. We shall show that, for $r$ sufficiently large, the convex hull of this disk together with the point $w$ contains the half of the segment $\tilde{L}$ between $(r/2, 0, h)$ and $(r,0,h)$. 
 
 \begin{figure}[htb]
\centering
\includegraphics[height=5cm]{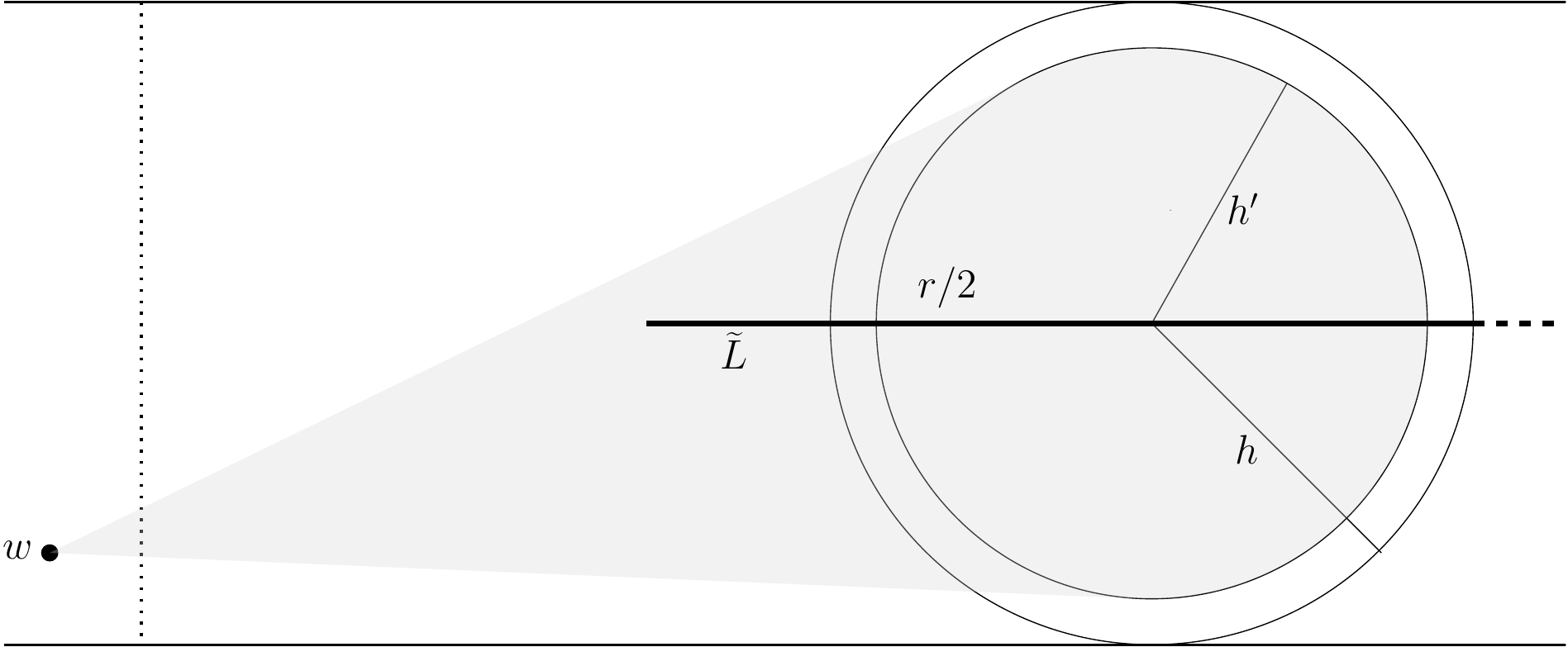}
\caption{{The shaded region represents the convex hull of the point $w$ and the disk of radius $h'$. This convex hull necessarily contains the left half of the line $\tilde L$ if we choose $h'\geq h/\sqrt{1+4h^2/r^2}$.}}
\label{fig:convexhull}
\end{figure}
 
Indeed, observe that the convex hull of a disk $D$ centered at $(r,0,h)$ and radius $h'$ and the point $w$ contains the half segment $\tilde L$ if and only if the line through $w$ and $(r/2,0,h)$ meets the disk $D$ (see Figure  \ref{fig:convexhull}), or equivalently if and only if the distance between the center $(r,0,h)$ of the disk $D$ and the line passing through $w= (x,y,h)$ and $(r/2,0,h)$ is smaller than $h'$. An elementary computation shows that this distance equals 
$$f(x,y):=\frac{(r/2)|y|}{\sqrt{y^2+(r/2-x)^2}}~.$$
Now, recall that $x<0$ and $|y|<h$, and observe that the function $f(x,y)$ satisfies $f(x,-y)=f(x,y)$; that $f(\cdot,y)$ is increasing as a function of $x\in (-\infty,0]$, and $f(x,\cdot)$ is increasing as a function of $y\in [0,h)$. Hence $f(x,y)$ is maximized at $(x,y)=(0,\pm h)$, where it attains the value 
$$\widehat h(r):=f(0,h)=h/\sqrt{1+4h^2/r^2}~.$$
In conclusion, if $h' \geq \widehat h(r)$, then half of the line $\tilde L$ is contained in the convex hull of $w$ and of the disc $D$ of radius $h'$. It turns out that the radius $h(1 - e^{-2r}/4)$ of the disk $I^+(q) \cap \{z=h\}$ that we consider is larger than $\widehat h(r)$ {if $r>r_1$, for $r_1$ sufficiently large}. 

This shows that, for $r$ sufficiently large, half of the line $\tilde L$ is contained in the convex hull of $w$ and of the disc $I^+(q) \cap \{z=h\}$. Repeating the same argument on the other side, we see that all of $\tilde{L}$ lies in the future of $\Sigma$.
\end{proof}


\subsection{Proof of Theorem \ref{thm: Lipschitz completeness intro}}\label{subsec:proofA}

Now we are ready to prove Theorem \ref{thm: Lipschitz completeness intro}, the main theorem of this section. 

\begin{proof}[Proof of Theorem \ref{thm: Lipschitz completeness intro}] It clearly suffices to prove the theorem in the case that $\theta_0 = 0$ and $\phi(0) = 0$. We will prove the contrapositive, namely we will suppose we have an entire hyperbolic surface $\Sigma$ that is incomplete at 0, and then we will show that for any slope $M$ there is some neighborhood of $0$ on which $\phi(\theta) > M |\theta|$. 

Let $P$ be the null support plane of $\Sigma$ in the direction $\theta = 0$. Let $\gamma$ be a geodesic ray of finite length in $\Sigma$ asymptotic to $P$. Also fix a second null support plane $Q$ of $\Sigma$, and let $L$ be the spacelike line $P \cap Q$. Up to an isometry of $\R^{2,1}$ we may suppose that $L$ is the $y$ axis and the unique null line in $P$ to which $\gamma$ converges (by Proposition \ref{prop: exists theta}) goes through the origin of $\R^{2,1}$. We furthermore suppose, by chopping off the beginning of $\gamma$ if necessary, that it is always in the future of $L$ and the coordinates $a$ and $c$ in \eqref{eqn: nullnull} of $\gamma$ are monotonic. This last condition implies that the timelike plane containing $L$ and $\gamma(t)$ converges monotonically in $t$ to the null plane $P$.

For any value of $\epsilon > 0$, by Lemma \ref{lem: timelike distance} there is a time after which the tail of $\gamma$ is entirely within timelike distance $\epsilon$ from $L$. For each point of this tail of $\gamma$, we apply Proposition \ref{prop: L tilde} with $h = 1$ to find a segment parallel to $L$ of length {$l(\epsilon)=-\log(4\epsilon)/2$}, lying in the future of $\Sigma$. By the choice of $h = 1$, each of these segments is contained in the hyperbolic cylinder $\{z^2 - x^2 = 1\}$. The union of these segments over the entire tail of $\gamma$ contains a curved half-infinite rectangular strip $R_\epsilon$ of this cylinder with bounds 
\begin{equation}\label{eq:2vs3}
|y| \leq l(\epsilon)/4
\end{equation} and $y \geq C(\epsilon)$ for some function $C(\epsilon)$ over which we have a priori no control. To explain the bound \eqref{eq:2vs3}, observe that if the curve $\gamma$ was actually contained in the plane $y=0$, then we would have the bound $|y| \leq l(\epsilon)/2$; in general, since the $y$-coordinates tends to zero along $\gamma$, we achieve \eqref{eq:2vs3} after taking $t$ large enough, or equivalently $x$ large enough. Since $R_\epsilon$ lies entirely in the future of $\Sigma$, the null support function of $R_\epsilon$ gives a lower bound for the null support function of $\Sigma$.

It remains to find the null support function $R_\epsilon$. Let $\tilde{R}_\epsilon$ be the doubly infinite strip we obtain by dropping the cutoff on the $y$ coordinate. We may equivalently describe $\tilde{R}_\epsilon$ as the intersection of the hyperbolic cylinder $\{z^2 - x^2 = 1\}$ with the future of a segment of $L$ of length $l(\epsilon)/2 - 2$, centered at $y= 0$. (See Figure \ref{fig:strip} for a picture of this intersection.) Consequently, the null support function of $\tilde{R}_\epsilon$ is at most the null support function of this segment. In fact it is not hard to see that every null support plane of $\tilde{R}_\epsilon$ is also a null support plane of this segment, and hence their null support functions actually coincide. Using the assumption that $L$ is the $y$-axis, one can check that this has null support function $(l(\epsilon)/2 - 2) |\sin(\theta)|/2$. Moreover, in those null directions sufficiently close to 0, every support plane of $\tilde{R}_\epsilon$ is also a support plane of $R_\epsilon$, and so in a neighborhood of zero (whose size depends on $C(\epsilon)$), the null support function of $R_\epsilon$ is exactly $(l(\epsilon)/2 - 2)|\sin(\theta)|/2$. 

Using the approximation $|\sin(\theta)| > |\theta|/2$ for sufficiently small $\theta$, we conclude that for all $M > 0$, there is a neighborhood of zero on which $\phi(\theta) \geq M |\theta|$, where $\phi$ is the support function of $\Sigma$. This is exactly the negation of the condition \eqref{eqn: M2} in the case $\phi(0) = 0$.
\end{proof}

\begin{figure}[htb]
\centering
\includegraphics[height=4cm]{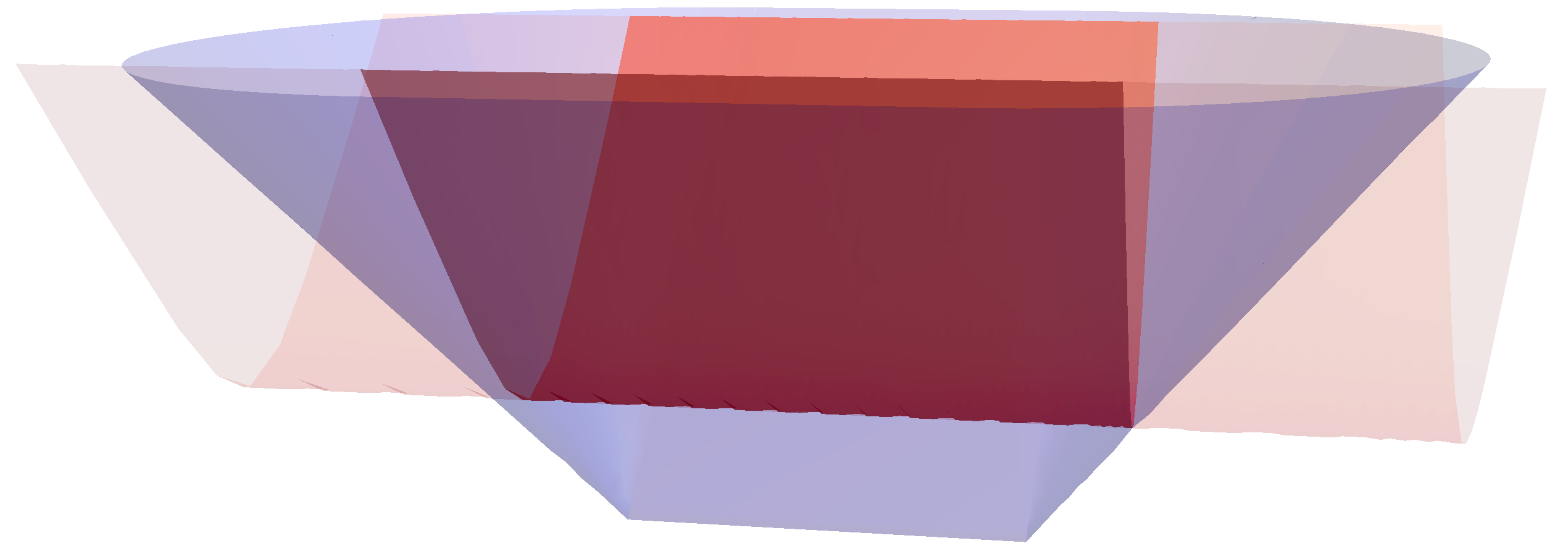}
\caption{The doubly infinite strip $\widetilde R_\epsilon$ is the intersection of the hyperbolic cylinder $\{z^2 - x^2 = 1\}$ with the future $J^+(S)$ of a segment $S$ (the blue region). Hence $\widetilde R_\epsilon$ and $J^+(S)$ have the same null support planes. \label{fig:strip}}
\end{figure}

\begin{example} If $C \subset \mathbb S^1$ is a Cantor set, then Theorem A shows that the entire hyperbolic surface whose null support function is
$$\chi_C=\begin{cases}  0 & \text{ on }C \\
+\infty &\text{ on }\mathbb S^1\setminus C\end{cases}$$
is complete. Cantor sets can arise as limit sets of linear actions of free groups in $\R^{2,1}$, in which case the hyperbolic surface with null support function $\chi_C$ will be invariant by the action.
\end{example}

\begin{remark} \label{rmk:nullvsparabolicComp} It is easy to check using the relationship in Equation \eqref{eq:null affine vs parabolic} (Proposition \ref{prop: ell/par}) that Condition \eqref{eqn: M2} at $\theta_0 = 0$ is equivalent to the same condition for the parabolic null support function $\psi$ at $\mathsf x = 0$, namely
\[
\psi(\mathsf x_i) < \psi(0) + M|\mathsf x_i|~.
\]
\end{remark}

\subsection{A geometric condition} \label{sec: dod}

In this section, we translate the condition \eqref{eqn: M2} into a condition on the domain of dependence of $\Sigma$ (Definition \ref{defi: dod}). For context, we remark that the domain of dependence of a convex entire spacelike surface is always a \emph{regular domain}, which is by definition an open subset of $\R^{2,1}$ that can be written as the intersection of the futures of some number --- possibly infinite, but at least 2 --- of non-parallel null planes \cite{Bonsante}. Its boundary is a convex entire achronal surface that is nowhere spacelike. The null support function of (the boundary of) the domain of dependence of $\Sigma$ is the same as the null support function of $\Sigma$, and conversely the domain of dependence of $\Sigma$ depends only on the null support function.

For reference, we state the negation of the condition \eqref{eqn: M2} at a point $\theta_0$: for all $M > 0$, there exists $\epsilon > 0$ such that $\phi(\theta) \geq \phi(\theta_0) + M |\theta - \theta_0|$ for all $\theta$ with $|\theta - \theta_0| < \epsilon$. By Remark \ref{rmk:infinite support complete} and Theorem \ref{thm: Lipschitz completeness intro}, if $\Sigma$ is incomplete at $\theta_0$ then $\phi(\theta_0)$ is finite and $\phi$ satisfies the negation of condition \eqref{eqn: M2} at $\theta_0$.

\begin{prop} \label{prop: contact set completeness} Suppose $\Sigma$ has a support plane $P$ in the direction $\theta_0$. Condition \eqref{eqn: M2} is equivalent to the condition that there is a null line in $P$ that does not intersect the boundary $\partial \mathcal D$ of the domain of dependence $\mathcal D$ of $\Sigma$.

\end{prop}

\begin{proof} Since both conditions are invariant by the isometry group of $\R^{2,1}$, it suffices to prove the proposition in the case that $\theta_0 = 0$ and $\phi(\theta_0) = 0$. Having done so, it is convenient to use Remark \ref{rmk:nullvsparabolicComp} to substitute the parabolic null support function $\psi$ for the elliptic null support function $\phi$. The null support plane $P$ corresponds to $\psi(0) = 0$.

We begin with an observation about the parabolic null support function of points. From Equations \eqref{eq:defi p} (with $\mathsf y=0$) and \eqref{eq: par null supp fun} for the parabolic null support function, we see that the parabolic null support function of a single point $p=[a,b,c]$ (or equivalently, of its future null cone) is the polynomial 
\begin{equation}\label{eq:par supp point}
f_p(\mathsf x)=(a-c)+2b\mathsf x-(a+c)\mathsf x^2
\end{equation}
 of degree at most 2 in $\mathsf x$. Clearly this construction provides a vector space isomorphism between $\R^{2,1}$ and the space of polynomials of degree at most 2. Moreover, observe that $p$ is in the null plane $P$ exactly when $\langle p,[1,0,1]\rangle=a-c=0$, which is equivalent to $f_p(0) = 0$.  In this case, the linear coefficient $f_p'(0)$ measures the horizontal displacement in $P$, meaning that the sets $\{p\in P\,|\,f_p'(0)=m\}$ are precisely the null lines $L$ in $P$. This provides a bijective correspondence between slopes $m\in \R$ and null lines $L\subset P$. 

With that in mind, we may now prove the proposition. First suppose that every null line in $P$ meets $\partial\mathcal D$ somewhere. Then for any positive number $M$, let $L=\{p\in P\,|\,f_p'(0)=M+1\}$ be the null line corresponding to the slope $m = M+1$.  By assumption, there is a point $p \in L \cap \partial\mathcal D$. Since $\psi$ is the parabolic null support function of $\mathcal D$ --- defined as a supremum over points of $\mathcal D$ --- it follows that $\psi \geq f_p$. But since $f_p(0) = 0$ and $f_p'(0) > M$, there is an $\epsilon>0$ such that $f_p(\mathsf x) \geq M\mathsf x$ for $0 \leq \mathsf x < \epsilon$. Repeating the argument for negative $M$ and taking the minimum of the two values of $\epsilon$, we conclude that $\psi$ satisfies the negation of condition \eqref{eqn: M2} at 0.

Conversely, suppose that the negation of condition \eqref{eqn: M2} holds for $\psi$ at $0$, and let $L$ be any null line in $P$, and $M$ be its corresponding slope. By assumption, there is an $\epsilon$ such that for all $\mathsf x$ with $|\mathsf x| < \epsilon$, there holds $\psi(\mathsf x) \geq M |\mathsf x|$. Now let $q$ be a point in the future of $\partial\mathcal D$, so that $f_q$ is a lower bound for $\psi$. Choose $p \in L$ sufficiently far in the future --- thus making its quadatic term sufficiently negative --- that $f_p(\mathsf x) < f_q(\mathsf x)$ for all $\mathsf x$ with $|\mathsf x| \geq \epsilon$. If we require furthermore that the quadratic term of $f_p$ is at least less than zero, then $f_p(\mathsf x) \leq M|\mathsf x|$ for $|\mathsf x| < \epsilon$ as well. Hence, whether $\mathsf x$ is small or large, we have $f_p(\mathsf x) \leq \psi(\mathsf x)$, and so $p$ is in $J^+(\partial\mathcal D)$. Since $p$ is also in $P$, a support plane for $\partial\mathcal D$, it follows that $p$ must actually be contained in $\partial\mathcal D$. This completes the proof.
\end{proof}

We immediately conclude the following:

\begin{cor}
\label{cor: contact set completeness} Let $\mathcal{D}$ be a regular domain in $\R^{2,1}$. Suppose that every support plane $P$ of $\mathcal{D}$ contains a null line disjoint from $\partial{\mathcal{D}}$. Then any
convex entire spacelike surface with domain of dependence $\mathcal{D}$ and curvature bounded below is complete.
\end{cor}


\section{Comparison results}\label{sec:comparison}


 In this section we prove a tool to deduce incompleteness in a given direction $\theta$ of convex entire  spacelike surfaces from the incompleteness of another one having the same asymptotics at $\theta$. The results of this section will be used to prove Theorem \ref{thm: Holder incomplete} and Theorem \ref{thm: xlogx incomplete}, which provide incompleteness criteria, but also in the proof of Theorem \ref{thm: xloglogx completeness}.

\begin{prop}\label{lem:complete comparison} Let $\Sigma_+$ and $\Sigma_-$ be convex entire spacelike  surfaces in $\R^{2,1}$ and let ${\theta}\in \mathbb S^1$. Suppose that:
\begin{enumerate}
\item $\Sigma_+$ contained in $J^+(\Sigma_-)$, and
\item $\Sigma_-$ and $\Sigma_+$ have the same null support plane in the direction of ${\theta}$.
\end{enumerate}
If $\Sigma_+$ is incomplete at ${\theta}$, then $\Sigma_-$ is also incomplete at ${\theta}$.
\end{prop}

\subsection{An important corollary}

Before proving the lemma, we note that in the case of hyperbolic surfaces, or more generally under certain bounds on the curvature, we can restate our assumptions purely in terms of support functions. Given two functions $\phi_1,\phi_2$ on the circle, we say that $\phi_1$ \emph{touches $\phi_2$ from below at $\theta\in \mathbb S^1$} if $\phi_1\leq \phi_2$ and $\phi_1(\theta)=\phi_2(\theta) < \infty$.

First, observe that if conditions (1) and (2) in Proposition \ref{lem:complete comparison}  hold, then denoting by $\phi_\pm$ the null support functions of $\Sigma_\pm$, $\phi_+$ touches $\phi_-$ from below at ${\theta}$. Indeed, if $\Sigma_+$ is in the future of $\Sigma_-$, the support functions $\phi_\pm$ of $\Sigma_\pm$ satisfy $\phi_+ \leq \phi_-$ everywhere; the condition (2) is then equivalent to the equality $\phi_+(\theta)=\phi_-(\theta)$. 
The converse holds under a suitable curvature assumption, based on an application of the maximum principle. The proof follows immediately from \cite[Proposition 3.11]{Bonsante:2019aa}.

\begin{prop}\label{prop:comp_principle}
Let $\Sigma_+$ and $\Sigma_-$ be entire  spacelike surfaces in $\R^{2,1}$ with null support functions $\phi_+$ and $\phi_-$, and let ${\theta}\in \mathbb S^1$. Suppose that the curvature functions of $\Sigma_+$ and $\Sigma_-$ satisfy:
$$K_{\Sigma_-}\leq -C\leq K_{\Sigma_+}\leq 0$$
for some constant $C>0$, and that $\phi_+$ touches $\phi_-$ from below at ${\theta}$.
Then
\begin{enumerate}
\item $\Sigma_+$ contained in the future of $\Sigma_-$, and
\item $\Sigma_-$ and $\Sigma_+$ have the same null support plane in the direction of ${\theta}$.
\end{enumerate}
\end{prop}

Hence from Proposition \ref{lem:complete comparison} and Proposition \ref{prop:comp_principle} we obtain the following corollary, which is the statement we will concretely apply in the next sections:

\begin{corx} \label{cor:complete comparison} 
Let $\Sigma_+$ and $\Sigma_-$ be entire spacelike  surfaces in $\R^{2,1}$ with null support functions $\phi_+$ and $\phi_-$, and let ${\theta}\in \mathbb S^1$. Suppose that the curvature functions of $\Sigma_+$ and $\Sigma_-$ satisfy:
$$K_{\Sigma_-}\leq -C\leq K_{\Sigma_+}\leq 0$$
for some constant $C>0$. If $\phi_+$ touches $\phi_-$ from below at ${\theta}$ and $\Sigma_+$ is incomplete at $\theta$, then $\Sigma_-$ is also incomplete at $\theta$.
\end{corx}

\subsection{Proof of Proposition \ref{lem:complete comparison}}

Let us now provide the proof of Proposition \ref{lem:complete comparison}. We first need to show that we can find a 1-Lipschitz map from $\Sigma_+$ to $\Sigma_-$.

\begin{lemma}\label{prop:contraction} Let $\Sigma_+$ and $\Sigma_-$ be entire spacelike  surfaces in $\R^{2,1}$ with $\Sigma_+$ convex. Then there exists a 1-Lipschitz map $\Pi:\Sigma_+\cap J^+(\Sigma_-)\to\Sigma_-$. Moreover, if $\Sigma_+$ and $\Sigma_-$ have the same null support plane $P$ and $\Sigma_+$ has a diverging Cauchy sequence $p_n$ asymptotic to $P$ which is in the future of $\Sigma_-$, then $\Pi(p_n)$ is a diverging Cauchy sequence in $\Sigma_-$ asymptotic to $P$. 
\end{lemma}
\begin{proof}
We essentially use the Lorentzian analog of the fact that closest point projection onto a convex body in Euclidean space is distance decreasing, but we must be slightly careful because the sign reversal means that we need to project from the wrong side, where the closest point may not be unique. Instead, we start with any foliation $\mathcal{F}=\{\mathcal F_t\}_{t\in \R}$ of $\R^{2,1}$ by  convex entire spacelike surfaces, including $\Sigma_+=\mathcal F_0$ as a leaf. For concreteness, take $\mathcal{F}$ to be the foliation by vertical translates of $\Sigma_+$. Let $V$ be the unique past timelike vector field such that 
\begin{enumerate}
\item $V$ is normal to $\mathcal{F}$, and
\item $\mathcal{F}$ is invariant by the flow of $V$, meaning that the time $t_0$ flow of $V$ maps $\mathcal F_{t}$ to $\mathcal F_{t+t_0}$.
\end{enumerate}
 If $\mathcal F_0$ is the graph of the function $f:\R^2\to\R$, then the leaves of $\mathcal{F}$ are the level sets of the function $F(x,y,z)=z-f(x,y)$, and then $V = - \grad F / \langle\grad F,\grad F\rangle$. Starting from any point in $\Sigma_+\cap J^+(\Sigma_-)$, if we flow by the vector field $V$ we will eventually hit $\Sigma_-$ because  $\Sigma_-$ is convex and entire (this is essentially the observation that the domain of dependence of a convex entire surface contains its future). Moreover, each flow line will intersect $\Sigma_-$ exactly once, and transversally, since $V$ is timelike and $\Sigma_-$ is spacelike. Let $\Phi(x,t): \Sigma_+ \times \R_{>0} \to \R^{2,1}$ be the flow of $V$, and $T: \Sigma_+\cap J^+(\Sigma_-) \to \R_{>0}$ be the intersection time with $\Sigma_-$, which is smooth by transversality. We want to show that if $x$ is in $\Sigma_+\cap J^+(\Sigma_-)$, then the smooth map $\Pi: \Sigma_+\cap J^+(\Sigma_-) \to \Sigma_-$ defined by $\Pi(x) = \Phi(x,T(x))$ is 1-Lipschitz. It is harmless to assume that $x$ is in  $\Sigma_+\cap I^+(\Sigma_-)$, so that a small neighbourhood of $x$ in $\Sigma_+$ is contained in the future of $\Sigma_-$. If $W$ is tangent to $\Sigma_+$ at $x$, then by the chain rule,
\[
d \Pi (W) = \frac{\partial \Phi}{\partial x}(W) + \frac{\partial \Phi}{\partial t} dT(W)
 \]
 Since $\frac{\partial \Phi}{\partial x}(W)$ is tangent to the leaves of the foliation, and $\frac{\partial \Phi}{\partial t} = V$ is orthogonal to the leaves of the foliation and timelike, we conclude
 \[
 |d \Pi(W)|^2 \leq \bigg| \frac{\partial \Phi}{\partial x}(W) \bigg|^2
 \]
It remains to show that the right hand side is less than or equal to $|W|^2$, which will follow from convexity of the foliation. If we extend $W$ to a vector field $\tilde{W}$ on $\R^{2,1}$ which is invariant by the flow of $V$, then the derivative of the square length of $\tilde{W}$ with respect to $t$ is
\[
\mathcal{L}_V \langle \tilde{W},\tilde{W}\rangle = ( \mathcal{L}_V g ) (\tilde{W}, \tilde{W}) =  - 2|V| \II (\tilde{W},\tilde{W}) \leq  0
\]
 where $g$ is the Minkowski metric, $\mathcal{L}$ the Lie derivative, $\II=(1/2)\mathcal L_Ng$ the second fundamental form of the foliation with respect to the future unit normal $N$, and $|V|$ the positive length defined by $V=-|V|N$. The final inequality is the convexity of the foliation. Integrating this inequality gives  
 \[
 \bigg| \frac{\partial \Phi}{\partial x}(0,T)(W) \bigg|^2 \leq  \bigg| \frac{\partial \Phi}{\partial x}(0,0)(W) \bigg|^2 = |W|^2
 \]
 Putting this together, we find that $\Pi$ is 1-Lipschitz.
 
 For the ``moreover" part, let $p_n$ a diverging Cauchy sequence in $\Sigma_+\cap J^+(\Sigma_-)$ such that 
 \begin{equation}\label{eq:cauchy1}
\langle p_n, \vec{\theta} \rangle \to \phi_+({\theta}).
\end{equation}
Since $\Pi$ is Lipschitz, the sequence $\Pi(p_n) \subset \Sigma_-$ is still Cauchy. Let us show that the sequence $\langle \Pi(p_n), \vec{\theta}\rangle$ converges to $\phi_-({\theta})$. Since $\Pi(p_n)$ lies in the timelike past of $p_n$ and $\vec{\theta}$ is future-directed, we have
\begin{equation}\label{eq:cauchy2}
\langle \Pi(p_n), \vec{\theta} \rangle \geq \langle p_n, \vec{\theta} \rangle.
\end{equation}
But since $\Sigma_-$ lies in the future of the null support plane of $\Sigma_+$ at $\theta$, we also have 
\begin{equation}\label{eq:cauchy3}
\langle \Pi(p_n), \vec{\theta}\rangle \leq \phi_+({\theta})=\phi_-(\theta).
\end{equation}
Putting together \eqref{eq:cauchy1}, \eqref{eq:cauchy2} and \eqref{eq:cauchy3}, we conclude that 
\begin{equation}\label{eq:limit product}
\langle \Pi(p_n), \vec{\theta} \rangle \to \phi_-({\theta})~.
\end{equation}

To conclude the proof, it only remains to show that the Cauchy sequence $\Pi(p_n)$ is diverging. This is easily done by a contradiction argument: if $\Pi(p_n)$ had a converging subsequence, then the limit point $q_\infty$ would be in $\Sigma_-$ because $\Sigma_-$ is closed, and in $P$ because $\langle q_\infty,\vec\theta\rangle=\phi_-(\theta)$ by \eqref{eq:limit product}. But then the null support plane $P$ would touch $\Sigma_-$ at $q_\infty$, contradicting Proposition \ref{prop achronal acausal}.
In conclusion, we have shown that $\Pi(p_n)$ is a diverging Cauchy sequence in $\Sigma_-$ asymptotic to $P$.
\end{proof}

We are now ready to conclude the proof of Proposition \ref{lem:complete comparison}.

\begin{proof}[Proof of Proposition \ref{lem:complete comparison}]   
Suppose $\Sigma_+$ is in the future of $\Sigma_-$. Let $\Pi:\Sigma_+\to\Sigma_-$ the 1-Lipschitz map from Lemma \ref{prop:contraction}. Since $\Sigma_+$ is incomplete at ${\theta}$, there is a Cauchy sequence $p_n$ in $\Sigma_+$ asymptotic to a null support plane $P$ in the direction of $\theta$. 
By the second part of Lemma \ref{prop:contraction}, $\Pi(p_n)$ is a diverging Cauchy sequence in $\Sigma_-$ asymptotic to $P$. 
By Lemma \ref{lemma:incompleteness by cauchy}, $\Sigma_-$ is incomplete at ${\theta}$.
\end{proof}


\section{Completeness II: subloglogarithmic condition}\label{sec:Completeness II}

The main result of this section, namely Theorem \ref{thm: xloglogx completeness} below, provides another sufficient condition for an entire spacelike surface of curvature bounded below to be complete. 

Let us give the statement of Theorem \ref{thm: xloglogx completeness}, which again we give in a local form; the statement that appears in the introduction  follows again by Corollary \ref{thm: incomplete at theta}.




\begin{reptheorem}{thm: xloglogx completeness}[Subloglogarithmic condition -- local version]
Let $\phi:\mathbb S^1\to\R\cup\{+\infty\}$ be lower semicontinuous and  finite on at least three points, and let $\lambda>0$. Suppose  $\theta_0 \in \mathbb S^1$ is such that $\phi(\theta_0)<+\infty$  and there is a one-sided neighbourhood $U$ of $\theta_0$ with
\begin{equation} \label{eqn: lxloglogx}
\phi(\theta) \leq \phi(\theta_0) + \frac{\lambda}{4}|\theta - \theta_0| \log(-\log|\theta-\theta_0|) \tag{Comp'}
\end{equation}
for every $\theta \in U$.
If $\Sigma$ is a convex entire spacelike surface in $\R^{2,1}$ with curvature bounded below by $-\lambda^2$ and null support function $\phi$, then $\Sigma$ is complete at $\theta_0$.

In particular, if condition \eqref{eqn: lxloglogx} holds for every $\theta_0$ at which $\phi$ is finite, then $\Sigma$ is complete.
\end{reptheorem}

\begin{remark} Note that in the statement of Theorem \ref{thm: xloglogx completeness}, we need an explicit lower bound on the curvature to conclude that the surface is complete, whereas in Theorem \ref{thm: Lipschitz completeness intro} we only needed that the curvature had some lower bound. A priori, it might be possible that there exist two convex entire spacelike surfaces with curvature bounded below, having the same null support function $\phi$, one of which is complete and the other incomplete. 

For instance, by the results in \cite{Bonsante:2019aa}, for every $\kappa>0$ there is a convex entire spacelike surface of constant negative curvature $-\kappa$ with a given null support function $\phi$, and these surfaces foliate a regular domain in $\R^{2,1}$. We are unable to rule out the possibility that some of these surfaces become incomplete $\kappa$   sufficiently large, for instance for $\phi$ a function as in Theorem \ref{thm: xloglogx completeness}.
\end{remark}

\begin{remark}
Theorem \ref{thm: xloglogx completeness} also shows that the sufficient condition for completeness from Theorem \ref{thm: Lipschitz completeness intro} is not a necessary condition. 
The point is that the function  $ |\theta - \theta_0| \log|\log|\theta-\theta_0|| $ does not satisfy the sequentially sublinear condition \eqref{eqn: M2} of the previous section. Neither though does the sequentially sublinear condition \eqref{eqn: M2} imply the subloglogarithmic condition \eqref{eqn: lxloglogx}, since the first condition applies even to subsequences. 
\end{remark}


\subsection{Preliminary steps}\label{sec:thmBprel}

To prove Theorem \ref{thm: xloglogx completeness}, we can assume after an isometry that $\theta_0=0$ and $\phi(0)=0$, in which case the condition \eqref{eqn: lxloglogx} becomes:
$$\phi(\theta) \leq \frac{\lambda|\theta| \log(-\log|\theta|)}{4}~.$$

We will find the computations easier if we work with parabolic null support functions, so as a preliminary step we show this condition essentially doesn't change when the null support function is replaced by the parabolic null support function, for the point $\theta_0 = 0$.

\begin{lemma}\label{lemma:mouse0} 
Let $s$ be a homogeneous function, and let $\phi$ and $\psi$ be respectively its elliptic dehomogenization and parabolic dehomogenization, with point at infinity $\pi$.
If 
\[
\phi(\theta) \leq \frac{\lambda|\theta| \log(-\log|\theta|)}{4}
\]
for all $\theta$ in some (one-sided) neighbourhood $U$ of 0, then 
\begin{equation} \label{eqn: par xloglogx}
\begin{split}
\psi(\mathsf x) \leq \frac{ \lambda|\mathsf x| \log(-\log|\mathsf x|)}{2}
\end{split}
\end{equation}
for all $\mathsf x$ in some (one-sided) neighbourhood $U'$ of 0.
\end{lemma}

\begin{proof}
The second statement of Proposition \ref{prop: ell/par}, gives us the relationships $\mathsf x = \tan(\theta/2)$ and
\[
\phi(\theta) = \frac{\psi(\mathsf x)}{1+\mathsf x^2}~.
\]
Assume $\mathsf x$ is small enough that $3|\mathsf x|/2 \leq |\theta| \leq 2|\mathsf x|$
\begin{align*}
\psi(\mathsf x) &= (1+\mathsf x^2) \phi(\theta) \\
&\leq (1 + \mathsf x^2) \frac{\lambda|\theta| \log(-\log|\theta|)}{4} \\
&\leq (1+\mathsf x^2) \frac{\lambda|\mathsf x| \log(-\log|\mathsf x| - \log(3/2))}{2} \\
&\leq \frac{\lambda|\mathsf x|}{2}(1+\mathsf x^2) \left(\log(-\log|\mathsf x|) + \frac{\log(3/2)}{\log|\mathsf x|}\right) 
\end{align*}
where in the last line we used that $\log(a+b)=\log a+\log(1+b/a)\leq \log a +(b/a)$ for any $a>0$ and $b\in (-a,+\infty)$.
Developing the cross terms in the product, we obtain:
$$\psi(\mathsf x)\leq \frac{\lambda|\mathsf x|\log(-\log|\mathsf x|)}{2}+\frac{\lambda|\mathsf x|}{2}\left(\mathsf x^2\log(-\log|\mathsf x|)+\frac{\log(3/2)}{\log|\mathsf x|}+\frac{\log(3/2)\mathsf x^2}{\log|\mathsf x|}\right)~.$$
We claim that the second term on the right hand side is negative for $|x|$ small, which concludes the proof. To see this more easily, set $|\mathsf x|=e^{-t}$. Then the sum  in the last bracket equals
$e^{-2t}\log(t)-\log(3/2)(1+e^{-2t})/t$, which is clearly negative for $t$ large. 
\end{proof}

As a last preliminary step, we observe that it suffices to prove Theorem \ref{thm: xloglogx completeness} assuming $\lambda=1$. This follows from the following lemma, whose proof is straightforward from the definition of support function.

\begin{lemma}\label{lemma:rescale par null}
Let $\phi$ and $\phi'$ be the  null support functions of two convex spacelike  surfaces $\Sigma$ and $\Sigma'$. If $\Sigma'=(1/\lambda)\Sigma$, then $\phi'=\phi/\lambda$. The same holds for the parabolic null support functions.
\end{lemma}

Indeed, from Lemma \ref{lemma:rescale par null}, if $\Sigma$ has curvature bounded below by $-\lambda^2$ and satisfies \eqref{eqn: par xloglogx}, then $\Sigma'=(1/\lambda)\Sigma$ has curvature bounded below by $-1$ and satisfies \eqref{eqn: par xloglogx} with $\lambda=1$. Clearly $\Sigma$ is complete in a given direction $\theta_0$ if and only if $\Sigma'$ is. Hence it is harmless to assume $\lambda=1$ in Theorem \ref{thm: xloglogx completeness}.

\subsection{Proof of Theorem \ref{thm: xloglogx completeness}}\label{subsec:proofthmB}

In the proof of Theorem \ref{thm: Lipschitz completeness intro}, we used two fundamental qualitative observations:
\begin{enumerate}
\item the timelike distance to a spacelike line contained in a null support plane $P$ tends to zero along incomplete geodesics asymptotic to $P$ (Lemma \ref{lem: timelike distance}), and
\item to each point in $\Sigma$ at which this timelike distance is small, there corresponds a segment in the future of $\Sigma$ with large spacelike length (Proposition \ref{prop: L tilde}).
\end{enumerate} 

We will achieve Theorem \ref{thm: xloglogx completeness} by turning the qualitative observations (1) and (2)  into more quantitative estimates.

\begin{proof}[Proof of Theorem \ref{thm: xloglogx completeness}]
In Section \ref{sec:thmBprel} we observed already that we can assume $\theta_0=0$, which corresponds to $\mathsf x=0$ in parabolic coordinates, that the parabolic null support function $\psi$ satisfies \eqref{eqn: par xloglogx}, and that $\lambda=1$. Moreover, up to applying a reflection in the plane $\{y=0\}$ to the surface $\Sigma$,  we can assume that the one-sided neighbourhood $U$ contains an interval of the form $[0,\epsilon)$. Hence we assume that \eqref{eqn: par xloglogx} holds for small positive $\mathsf x$.

Suppose that $\Sigma$ is incomplete at $\mathsf x=0$, and let $\gamma$ be a proper geodesic ray asymptotic to the support plane of $\Sigma$ in the direction of $\mathsf x=0$, which we will denote $P$. Again allowing ourselves the freedom of an isometry of $\R^{2,1}$, we may suppose that the null line in $P$ to which $\gamma$ is asymptotic (Proposition \ref{prop: exists theta}) goes through the origin. We moreover assume that the $y$-axis is the intersection of $P$ with some other null support plane of $\Sigma$. In particular, $\Sigma$ is contained in the future of the $y$-axis.

 The next step is to define a certain spacelike surface with $P$ as a null support plane and which is complete at $P$. It is easiest to describe this surface in Lorentzian-cylindrical coordinates $(\rho,y,\alpha)$  defined on the future of the $y$-axis by

\begin{equation} \label{eqn: cylindrical}
\begin{split}
\begin{bmatrix}  x \\ y \\ z   \end{bmatrix} = \begin{bmatrix}  \rho \sinh(\alpha) \\ y \\ \rho \cosh(\alpha)   \end{bmatrix}~.
\end{split}
\end{equation}

Observe that $\rho$ is the timelike distance to the $y$-axis, and $\alpha \to +\infty$ as you approach $P$. Let $\Sigma_\epsilon$ be the spacelike surface defined by $\rho(\alpha) = \epsilon (1 + \alpha^2)^{-\frac{1}{2}}$ for $\epsilon > 0$ to be chosen. We first show that $\Sigma_\epsilon$ is spacelike and complete at $P$.

Since $\Sigma_\epsilon$ is invariant by translation in the $y$ direction, it suffices to show that the curve $\rho = \epsilon (1 + \alpha^2)^{-\frac{1}{2}}$ in $\R^{1,1}$ is spacelike with infinite length as $\alpha \to +\infty$. Using the equation \eqref{eqn: cylindrical}, we find that the square norm of the derivative of this curve is 
\[
\rho^2 - \left(\frac{d\rho}{d\alpha}\right)^2 = \epsilon^2\left(\frac{1}{1 + \alpha^2} - \frac{\alpha^2}{(1 + \alpha^2)^3}\right)
\]
and we see that this is spacelike for all $\alpha$. For $\alpha$ large, this is bigger than $\epsilon^2/(2\alpha^2)$, so integrating its square root to infinity, we see that the length is infinite. 

We will now use the completeness of $\Sigma_\epsilon$ at $P$ to show that there is a divergent sequence of points $\gamma(t_n)$ on the curve $\gamma$ that lie in the past of $\Sigma_\epsilon$. Otherwise, beyond a certain point $\gamma$ would lie entirely in the future of $\Sigma_\epsilon$, and using Lemma \ref{prop:contraction} we could project this tail of $\gamma$ to a proper curve on $\Sigma_\epsilon$ of length at most the length of $\gamma$, which is still asymptotic to $P$. Hence using Lemma \ref{lemma:incompleteness by cauchy}, this would contradict the completeness of $\Sigma_\epsilon$ at $P$.

Let us now consider the points $\gamma(t_n)$ in the cylindrical coordinates defined above. By the convergence of $\gamma$ to the null line through the origin, the $y$-coordinates $y_n$ of these points tend to zero. Furthermore, the $\alpha$-coordinates $\alpha_n$ go to infinity, and by constructions the $\rho$-coordinates satisfy $0 < \rho_n < \epsilon (1 + \alpha_n^2)^{-\frac{1}{2}}<\epsilon \alpha_n^{-1}$. Define $q_n = \gamma(t_n)$ and let $p_n$ be the point in the $y$-axis with the same $y$-coordinate as $q_n$. 

We now apply Proposition \ref{prop: L tilde} to each of the pairs of points $(p_n, q_n)$ to produce line segments $\tilde{L}_n$ in the future of $\Sigma$. Taking $h=1$, and using $\mathrm{dist}(p_n, q_n) \leq \epsilon \alpha_n^{-1}$, the proposition gives a segment $\tilde{L}_n$ of length 
\[
r=\frac{1}{2}\log\left(\frac{\alpha_n}{4\epsilon}\right)~.
\]
Setting $C = -\log(4\epsilon)/4$, this implies $r/2 = \log(\alpha_n)/4 + C$. In summary, the endpoints of the segment $\tilde{L}_n$ have $\rho$-coordinate 1, $\alpha$-coordinate $\alpha_n$, and $y$-coordinate $y_n \pm (\log(\alpha_n)/4 + C)$. Let us call $w_n$ the endpoint with $y$-coordinate $y_n +\log(\alpha_n)/4 + C$. 

These segments $\tilde{L}_n$ lie in the future of $\Sigma$, and therefore give lower bounds on the (parabolic) null support function $\psi$ of $\Sigma$. To calculate these lower bounds, we need to get our hands just a little bit dirty with parabolic coordinates.

Recall from Equation \eqref{eq:par supp point} that the parabolic null support function of the point $p=[a, b, c]$ (equivalently, of its future null cone) equals
\[
f_p(\mathsf x) = (a-c) + 2b\mathsf x - (a+c)\mathsf x^2~.
\]
Hence the parabolic null support function of the point $w_n$ with cylindrical coordinates $(1, \alpha_n, y_n + \log(\alpha_n)/4{+C})$ is
\[
f_{w_n}(\mathsf x) = - e^{-\alpha_n} + 2\left(y_n + \frac{\log(\alpha_n)}{4} + C\right)\mathsf x - e^{\alpha_n}\mathsf x^2~.
\]
For each $n$, we have $f_{w_n}(\mathsf x) \leq \psi(\mathsf x)$ for all $\mathsf x$, but we only use this inequality at the point $\mathsf x_n = e^{-\alpha_n}$, for which we have
\[
\psi(\mathsf x_n) \geq f_{w_n}(\mathsf x_n) = 2\mathsf x_n\left(-1 +  y_n + \frac{\log(-\log \mathsf x_n)}{4} + C\right)~.
\]
Since the $y_n$ tend to zero, we can take $\epsilon$ small enough so that $C > \max_n y_n + 1$, and conclude
\[
\psi(\mathsf x_n) > \frac{\mathsf x_n \log(-\log \mathsf x_n)}{2}~,
\]
and this gives a contradiction. Hence the theorem is proved. 
\end{proof}

\section{Incompleteness I: power function condition}\label{sec:Holder incomplete}

The purpose of this section is to prove Theorem \ref{thm: Holder incomplete}, which we restate below with a slightly strengthened conclusion with respect to the introduction, namely that $\Sigma$ is incomplete \emph{at $\theta_0$}.

\begin{reptheorem}{thm: Holder incomplete}[Power function condition -- local version]  Let $\phi: \mathbb S^1 \to \R \cup \{+\infty\}$ be lower semicontinuous and finite on at least three points. Suppose $\theta_0 \in \mathbb S^1$ is such that $\phi(\theta_0)<+\infty$ and there exists a neighborhood $U$ of $\theta_0$ and constants $\epsilon>0$, $0 < \alpha < 1$ with
\begin{equation}\label{eq:inc power}
\phi(\theta) - \phi(\theta_0) > \epsilon|\theta - \theta_0|^{\alpha} \tag{Inc}
\end{equation}
for every $\theta \in U$. 
If $\Sigma$ is a convex entire spacelike  surface in $\R^{2,1}$ with null support function $\phi$ and curvature bounded above by a negative constant, then $\Sigma$ is incomplete at $\theta_0$.
\end{reptheorem}

\subsection{Outline of the strategy}\label{subsec:outline}
\label{subsec:easy}
Before entering the details of the proof of Theorem \ref{thm: Holder incomplete}, we remark that an easier incompleteness criterion can be immediately deduced from Corollary \ref{cor:complete comparison}. That is, we prove the following statement. 

\begin{cor}\label{cor easy}
Suppose $\phi: \mathbb S^1 \to \R \cup \{+\infty\}$ is a lower semicontinuous function, finite on at least three points. Suppose that there exist $\theta_0 \in \mathbb S^1$ at which $\phi$ has a \emph{two-sided jump}: that is, $\phi(\theta_0)<+\infty$ and there are a neighborhood $U$ of $\theta_0$ and a constant $\epsilon>0$ such that
\begin{equation}\label{eq:weak inc}
\phi(\theta) - \phi(\theta_0) > \epsilon
\end{equation}
for every $\theta \in U$. If $\Sigma$ is a convex entire spacelike  surface in $\R^{2,1}$ with null support function $\phi$ and curvature bounded above by a negative constant, then $\Sigma$ is incomplete at $\theta_0$.
\end{cor}

\begin{remark}
Theorem \ref{thm: Holder incomplete}, which is the main goal of this section, is a much stronger result in this direction. In fact, it infers incompleteness of $\Sigma$ under the condition \eqref{eq:inc power}, which is much weaker than \eqref{eq:weak inc}.  
\end{remark}

As a warm-up towards Theorem \ref{thm: Holder incomplete}, let us prove Corollary \ref{cor easy}.

\begin{proof}[Proof of Corollary \ref{cor easy}]
We can assume (up to an isometry and a rescaling) that $\theta_0=\pi$, $\phi(\theta_0)=-\epsilon$ (which implies $\phi>0$ in a punctured neighbourhood of $\pi$) and that the surface $\Sigma$ has curvature bounded above by $-1$. The strategy of the proof consists in applying Corollary \ref{cor:complete comparison} to $\Sigma_-=\Sigma$ and to $\Sigma_+$ a parabolic-invariant  entire surface of curvature $-1$ as in Example \ref{ex:parabolic}, which is incomplete at $\pi$, and has null support function $\phi_+$ equal to $-\epsilon$ at $\pi$ and identically zero elsewhere. Hence in particular, we have $\phi_+(\pi)=\phi(\pi)=-\epsilon$ and $\phi_+\leq \phi$ in a neighbourhood of $\pi$. By identity \eqref{eq:null affine vs parabolic}, the parabolic null support function $\psi_+$ of $\Sigma_+$ is also identically zero.

We are however not yet in the conditions to  apply Corollary \ref{cor:complete comparison}, since for that we  need  the inequality $\phi_+\leq \phi$ to hold globally on $\mathbb S^1$. This is easily fixed by replacing  $\Sigma_+$ with a translate of $\Sigma_+$ in the direction $\vec{\pi}$. Indeed, the parabolic null support function $\psi$ of $\Sigma$ is lower semi-continuous and positive near infinity, so it is bounded below by some constant $-2c$. By Equations \eqref{eq:defi p} (with $\mathsf y=0$) and \eqref{eq: par null supp fun}, translating $\Sigma_+$ by the vector $c\vec{\pi}$ adds the constant $-2c$ to its parabolic null support function, and it clearly does not change the value of the elliptic null support function at $\pi$. Hence, the parabolic null support function $\psi'_+$ of the surface $\Sigma_+' = \Sigma_+ + c \vec{\pi}$ is equal to $-2c$, and so its elliptic null support function is also less than or equal to $\phi$ everywhere. So, we apply corollary \ref{cor:complete comparison} to $\Sigma$ and $\Sigma_+'$ and conclude that $\Sigma$ is incomplete at $\pi$.
\end{proof}

To achieve Theorem \ref{thm: Holder incomplete}, we will apply a very similar strategy. However, we construct a more complicated surface to serve as a ``barrier'' $\Sigma_+$, a role that in the proof of Corollary \ref{cor easy} above is played by the surface of Example \ref{ex:parabolic}. Let us outline the strategy that we will implement.

The idea of our proof is to construct, for any choice of constants $\epsilon, \alpha$ and $C$, an explicit  entire surface $\Sigma_+$, which:
\begin{enumerate}
\item is spacelike
\item  has curvature between $-C$ and 0,
\item is incomplete at $\theta_0$, and
\item  has elliptic null support function $\phi_+$ satisfying 
\begin{enumerate}
\item $\phi_+(\pi) = 0$
\item $\phi_+(\theta) - \phi_+(\theta_0) \leq \epsilon |\theta - \theta_0|^{\alpha}$ on some neighborhood of $\theta_0$, and
\item $\phi_+$ is bounded above.
\end{enumerate}
 \end{enumerate}
and to invoke the comparison principle Corollary \ref{cor:complete comparison}. 
Indeed, as long as $\phi_+$ satisfies these four conditions, then if $\Sigma$ is any surface with null support function $\phi$ satisfying the inequality \eqref{eq:inc power} on some neighborhood of $\theta_0$, then a suitable translate of $\Sigma_+$ in the null direction $\vec{\theta_0}$ will have null support function less than $\phi$ and will still be incomplete at $\theta_0$ (as in the proof of Corollary \ref{cor easy}). Hence by the curvature assumptions, we can conclude by Corollary \ref{cor:complete comparison} that $\Sigma$ is incomplete.

This approach will require a few preliminary steps, which are proved in Subsection \ref{subsec:criteria},
in order to reformulate the conditions on $\Sigma_+$ into conditions on its parabolic support function. For the impatient reader, we point out that the explicit expression for the parabolic support function $u = s_{\mathrm{par}}$ of $\Sigma_+$ is given in Subsection \ref{subsec:explicit}, more precisely  in equation \eqref{eqn: u}. 
The motivation behind the somewhat complicated formula in equation \eqref{eqn: u} will be more clear from the preliminary lemmas of Subsection \ref{subsec:criteria}.

\subsection{Preliminary steps}\label{subsec:criteriaXX}\label{subsec:criteria}
In this section we will establish some standard relationships between geometric properties of a surface and its parabolic support functions.
In particular we will explicitly compute the curvature of a surface at the pre-image of the Gauss map in terms of the parabolic support function under some natural regularity assumptions. 
Most of the content of this section is  known in the literature (see for instance \cite{bonfill}), 
but we give a self contained exposition since the computations in terms of the parabolic support function 
are less common.

Recall from Corollary \ref{lem: convex duality} that a convex entire achronal surface is associated to a proper closed convex function $u$ on the upper half-space. 
The first lemma of this section gives a sufficient condition in terms of $u$ for the corresponding surface $\Sigma_+$ to be acausal instead of merely achronal. For the surface obtained from our specific function $u$ defined in \eqref{eqn: u}, we will then apply Corollary \ref{cor:spacelike iff acausal}  to ensure that $\Sigma_+$ is spacelike, which is the first condition in the list of Section \ref{subsec:outline}.

To simplify the notation, we will associate to a proper closed function $u:\overline{\mathcal H}\cup\{\infty\}\to\R\cup\{+\infty\}$, which is convex with respect to the affine structure of the plane $x+z=2$ (see Corollary \ref{cor: convex duality 2} for this definition), an ``extension'' $\overline u:\overline{\mathcal H}\cup\{\infty\}$ to the one-point compactification $\overline{\mathcal H}\cup\{\infty\}$ of $\overline{\mathcal H}$, defined by:
\begin{equation}\label{eq:extension infty}
\overline{u}(\infty) = \liminf_{\mathsf x^2 + \mathsf y^2 \to \infty} \frac{u(\mathsf x,\mathsf y)}{1 + \mathsf x^2 + \mathsf y^2}~.
\end{equation}

\begin{remark}\label{remark theta0 pi}
The definition of $\overline{u}(\infty)$ in \eqref{eq:extension infty} is of course motivated by Proposition \ref{prop: ell/par} (see \eqref{eqn: u of infinity}): if $u=s_\mathrm{par}$ is the parabolic support function of a convex entire achronal surface $\Sigma$, then $\overline{u}(\infty)$ equals the value of the (elliptic) null support function of $\Sigma$ at $\pi$. We remark that in the remainder of this section we will assume (up to applying an isometry of $\R^{2,1}$) that $\theta_0 = \pi$ and $\phi(\theta_0) = 0$ in the statement of Theorem \ref{thm: Holder incomplete}; hence we will be particularly interested in the behaviour of surfaces $\Sigma$ in the direction $\pi$ in relation with the null support function around $\pi$.
\end{remark}

\begin{lemma} \label{lem: spacelike} 
Let $\Sigma$ be a convex entire achronal surface in $\R^{2,1}$ and let 
$$u=s_{\mathrm{par}}: \overline{\mathcal{H}} \to \R \cup \{+\infty\}$$ be its parabolic support function. Suppose that for every function $f$ of the form 
\begin{equation}\label{eq:quadratic}
f(\mathsf x,\mathsf y) = a + b \mathsf x + c (\mathsf x^2 + \mathsf y^2)
\end{equation} such that $f \leq u$ on $\overline{\mathcal{H}}$, it holds that $\overline{f} < \overline{u}$ on $\R \cup \{\infty\}$. Then $\Sigma$ is acausal.
\end{lemma}

\begin{proof}
Suppose $\Sigma$ were not acausal. Then we claim that it is touched by a null support plane at some point $q$. Indeed, it contains points $p$ and $q$ such that the segment from $p$ to $q$ is lightlike. Notice that $p$ is contained in the past of every spacelike plane through $q$, so $\Sigma$ can have no spacelike support plane containing $q$. Therefore the support plane through $q$ must be lightlike (and indeed its tangent space is the orthogonal to the vector $q-p$).

 Let $f$ be the parabolic support function of the point $q$ (equivalently, of its future null cone). From Equation \eqref{eq:defi p} and \eqref{eq: par supp fun}, $f$ is a quadratic polynomial of the form \eqref{eq:quadratic} for some $a,b,c\in\R$. Since $\Sigma$ is achronal, the future null cone of $q$ lies in its (causal) future so the elliptic support function of $q$ is less than or equal to the elliptic support function of $\Sigma$. But, since both are supported by the null plane $P$, their support functions are equal in the null direction corresponding to $P$. Translated into the parabolic support functions using Proposition \ref{prop: ell/par}, this says that $f \leq u$ on $\overline{\mathcal{H}}$ but $f = u$ at the point of $\R \cup \{\infty\}$ corresponding to $P$.
\end{proof}

{Following the list in Section \ref{subsec:outline}, we then} need to be able to {show that $\Sigma_+$ is $C^2$ and to} verify property (2)  of the surface $\Sigma_+$, namely that its curvature is bounded below. In order to do this, we will need to parametrize $\Sigma_+$ by the inverse of its Gauss map. More precisely, under a mild regularity assumption on the support function,
we will point out a continuous map $\ginv:\mathcal H\to\Sigma$ which has the property that the plane though $\ginv(\mathsf x,\mathsf y)$ orthogonal to $\zeta(\mathsf x,\mathsf y)$ is a support plane for $\Sigma$. (Recall that the map $\zeta$ has been introduced in \eqref{eq:defi p}.)

\begin{lemma}\label{lm:c1u}
Let $\Sigma$ be a convex entire acausal surface in $\R^{2,1}$ and let 
$$u=s_{\mathrm{par}}:\overline{\mathcal{H}}\to\mathbb R\cup\{+\infty\}$$ be its parabolic support function.
Assume that $u$ is $C^1$ in a neighbourhood $U$ of  a point $(\mathsf{x}_0,\mathsf{y}_0)\in\mathcal H$. Then
\begin{enumerate}
\item There exists a unique map $\ginv: U\to\mathbb R^{2,1}$ which satisfies the following equations:
\begin{align} 
\langle \ginv, \zeta \rangle &= u \label{eqn: def v 1}\\
\langle \ginv, \zeta_{\mathsf x} \rangle &= u_{\mathsf x} \label{eqn: def v 2} \\  
\langle \ginv, \zeta_{\mathsf y} \rangle &= u_{\mathsf y}   \label{eqn: def v 3}
\end{align}
\item $\ginv$ is continuous and $\ginv(\mathsf{x}_0,\mathsf{y}_0)$ is a point of $\Sigma$ and the affine plane through $\ginv(\mathsf{x}_0,\mathsf{y}_0)$ orthogonal to $\zeta(\mathsf{x}_0,\mathsf{y}_0)$ meets $\Sigma$ only
 at $\ginv(\mathsf{x}_0,\mathsf{y}_0)$.
\item  If moreover $u\circ\zeta^{-1}$ is strictly convex on $U$, then $\ginv$ is a homeomorphism onto an open subset $V$ of $\Sigma$.
 In this case  $\Sigma$ is $C^1$ and spacelike on $V$, and 
 \begin{equation}\label{eq:c1}
 G\circ \ginv=\bar\zeta
 \end{equation}
  where  $G:V\to\mathbb H^2$ is the Gauss map and $\bar\zeta:=\frac{\zeta}{|\zeta|}$.
\end{enumerate}

\end{lemma}
\begin{proof}
The fact that $\ginv$ is well-defined and continuous follows from the observation that $\zeta(\mathsf{x},\mathsf{y}), \zeta_{\mathsf x}(\mathsf{x},\mathsf{y}), \zeta_{\mathsf y}(\mathsf{x},\mathsf{y})$ form a basis of $\mathbb R^{2,1}$ for every $(\mathsf{x},\mathsf{y})\in\mathcal H$.

Let $\ginv_0=\ginv(\mathsf{x}_0,\mathsf{y}_0)$ and consider the function $F: \mathcal H\to\mathbb R$ defined by 
$$F(\mathsf x,\mathsf y)=u(\mathsf x,\mathsf y)-\langle \ginv_0,\zeta(\mathsf x,\mathsf y)\rangle~,$$
Notice that $F\circ\zeta^{-1}$ is convex and $F$ vanishes at $(\mathsf{x}_0,\mathsf{y}_0)$ by \eqref{eqn: def v 1}. Moreover by \eqref{eqn: def v 2} and \eqref{eqn: def v 3}, $F$ has a critical point at 
$(\mathsf{x}_0,\mathsf{y}_0)$. It follows that $F(\mathsf x,\mathsf y)\geq 0$ for all $(\mathsf x,\mathsf y)\in\mathcal H$, that is  $\langle \ginv_0,\zeta(\mathsf x,\mathsf y)\rangle\leq u(\mathsf x,\mathsf y)$, with equality at $\zeta(\mathsf{x}_0,\mathsf{y}_0)$. Now, recall from Proposition \ref{prop: convex duality} and Corollary \ref{cor: convex duality 2} that the entire surface $\Sigma$ is entirely determined by its parabolic support function $u$: more precisely, $\Sigma$ is the boundary of the convex domain in $\R^{2,1}$ consisting of the points $p$ such that $\langle p,\zeta(\mathsf x,\mathsf y)\rangle \leq u(\mathsf x,\mathsf y)$ for all $(\mathsf x,\mathsf y)\in \mathcal H$. Hence $\ginv_0\in\Sigma$.

Let now $p\in\mathbb R^{2,1}$ be a point in the affine plane defined by $u(\mathsf x_0,\mathsf y_0)=\langle p,\zeta(\mathsf x_0,\mathsf y_0)\rangle$, with $p\neq \ginv_0$. 
Then at least  one among \eqref{eqn: def v 2} and  \eqref{eqn: def v 3} does not hold so there is a vector $W\in T_{(\mathsf x_0,\mathsf y_0)}\mathcal H$ such that
\[
    du(W)<\langle p,d\zeta(W)\rangle
\] 
and this implies that there is some point $(\mathsf x_1,\mathsf  y_1)\in\mathcal H$ such that
\[
    u(\mathsf x_1,\mathsf  y_1)<\langle p,\zeta(\mathsf x_1,\mathsf  y_1)\rangle
\]
showing that $p\notin\Sigma$.

Assume now that $u\circ\zeta^{-1}$ is strictly convex on $U$. Hence the function $F\circ\zeta^{-1}$  is strictly convex, since the difference between $u\circ\zeta^{-1}$ and $F\circ\zeta^{-1}$ is an affine function. We have already observed that  $F(\mathsf x,\mathsf y)\geq 0$ and $F(\mathsf x_0,\mathsf y_0)= 0$. Hence strict convexity of $F\circ\zeta^{-1}$ implies that the zero set of $F$ only contains $(\mathsf x_0,\mathsf y_0)$. 
Since the zero set of $F$ equals the set of $(\mathsf x,\mathsf y)\in\mathcal H$  such that $\zeta(\mathsf x,\mathsf y)$ is orthogonal to a support plane through $\ginv_0$, this proves that $\ginv$ is an injective function, and by the Invariance of domain Theorem, $V=\ginv(U)$ is an open subset of $\Sigma$.
Moreover each point in $V$ admits only one support plane, hence the surface $\Sigma$ is $C^1$ on $\ginv(U)$.
Since those support planes are spacelike, $V$ is a spacelike surface.
Finally using that the plane through $\ginv(\mathsf x,\mathsf y)$ orthogonal to $\zeta(\mathsf x,\mathsf y)$ is a support plane for $\Sigma$ we conclude that $G\circ \ginv=\bar\zeta$.
\end{proof}

Now for a $C^2$ function $u$ on $\mathcal H$,  we give a characterization of the convexity of  $u\circ\zeta^{-1}$ in terms of the derivatives of $u$.
\begin{lemma}\label{lm:convexu}
Let $u$ be a $C^2$ function on a domain $U$ in the upper half plane $\mathcal{H}$, then $u\circ\zeta^{-1}$ is locally convex if and only if
the matrix 
\begin{equation}\label{AA}
A = \frac{1}{2}\begin{bmatrix} \mathsf yu_{\mathsf x\mathsf x} - u_{\mathsf y} & \mathsf yu_{\mathsf x\mathsf y}  \\ \mathsf yu_{\mathsf x\mathsf y} & \mathsf yu_{\mathsf y\mathsf y} - u_{\mathsf y}  \end{bmatrix}
\end{equation}
is positive semi-definite at every point of $U$. Moreover if $A$ is positive definite, then $u\circ\zeta^{-1}$ is strictly convex on $\zeta(U)$. 
\end{lemma}
\begin{proof}
Consider for every $(\mathsf{x}_0,\mathsf{y}_0)\in U$ a vector $\ginv_0=\ginv(\mathsf{x}_0,\mathsf{y}_0)\in\mathbb R^{2,1}$ which satisfies the system of equations $(\ref{eqn: def v 1},~\ref{eqn: def v 2},~\ref{eqn: def v 3})$. This vector exists since $\zeta,\zeta_x,\zeta_y$ are linearly indipendent.
As observed before, both the function $F_{(\mathsf x_0, \mathsf y_0)}(\mathsf x, \mathsf y)= u(\mathsf x,\mathsf y)-\langle \ginv_0, \zeta(\mathsf x,\mathsf y))$ and its differential vanish at $(\mathsf x_0,\mathsf y_0)$. Since $F_{(\mathsf x_0, \mathsf y_0)}\circ\zeta^{-1}$ and $u\circ\zeta^{-1}$ differ by an affine function, $u\circ\zeta^{-1}$  is locally convex if and only if for every $(\mathsf x_0,\mathsf y_0)$ the function $F_{(\mathsf x_0, \mathsf y_0)}$ attains a local minimum at 
$(\mathsf x_0, \mathsf y_0)$.
By differentiating \eqref{eq:defi p} we have that
\[
  \zeta_{\mathsf{xx}}=\zeta_{\mathsf{yy}}=\frac{1}{\mathsf y}\zeta_{\mathsf{y}}\,,\qquad\zeta_{\mathsf{xy}}=0
\]
so a simple computation shows that the Hessian of $F_{(\mathsf x_0, \mathsf y_0)}$ at $(\mathsf x_0,\mathsf y_0)$ with respect to the basis $\frac{\partial\,}{\partial\mathsf x},\frac{\partial\,}{\partial\mathsf y}$ is given by 
$\frac{2}{\mathsf y_0}A(\mathsf x_0,\mathsf y_0)$ and the result follows.
\end{proof}

In order to relate a higher regularity of the support function to a higher regularity of the surface $\Sigma$, the following standard observation will be useful.

\begin{lemma}\label{lm:gaussc1}
Let $\Sigma$ be a  spacelike surface in $\mathbb R^{2,1}$, and $G:\Sigma\to\mathbb H^2$ be its Gauss map.
If $G$ is $C^1$ around a point $p_0\in\Sigma$, then $\Sigma$ is a $C^2$-submanifold around $p_0$.
\end{lemma}
\begin{proof}
Recall that a spacelike surface is $C^1$ by definition. Hence around $p_0$ the surface $\Sigma$ is the graph of a $C^1$-function $f$ defined on some open subset $\Omega$ of $\mathbb R^2$.
The map $\iota(x,y)=(x,y,f(x,y))$ furnishes a $C^1$-parameterization of $\Sigma$ around $p_0$. In this parameterization we have
\[
    G(\iota(x,y))=\frac{1}{\sqrt{1-f_x^2-f_y^2}}\begin{bmatrix}f_x\\f_y\\1\end{bmatrix}~.
\]
Since the standard immersion of $\mathbb H^2$ into $\mathbb R^{2,1}$ is smooth, we have that the functions
\[
  g_1(x,y)=\frac{f_x}{\sqrt{1-f_x^2-f_y^2}},\quad  g_2(x,y)=\frac{f_y}{\sqrt{1-f_x^2-f_y^2}},\quad  g_3(x,y)=\frac{1}{\sqrt{1-f_x^2-f_y^2}},
\]
are $C^1$. Since $f_x=g_1/g_3$, and $f_y=g_2/g_3$ the result immediately follows.
\end{proof}

\begin{lemma}\label{lm:c2}
Let $\Sigma$ be a convex spacelike  surface in $\mathbb R^{2,1}$, and let $u:\mathcal{H}\to\mathbb R\cup\{+\infty\}$ be its parabolic support function.
Assume that $u$ is $C^2$ in a neighborhood $U$ of  $(\mathsf{x}_0, \mathsf{y}_0)$, and that $A(\mathsf{x}_0, \mathsf{y}_0)$ is  positive definite, where $A$ is the matrix defined in \eqref{AA}.
Denote by $\ginv: U\to\Sigma$ the map defined in Lemma \ref{lm:c1u}. Then
\begin{itemize}
\item $\Sigma$ is $C^2$ around $p_0=\ginv(\mathsf{x}_0,\mathsf{y}_0)$;
\item The representative matrix of the shape operator $\mathsf B(p_0)$ with respect to the basis of $T_{p_0}\Sigma=T_{\bar\zeta_0}\mathbb H^2= \{\bar\zeta_{\mathsf x}(\mathsf{x}_0, \mathsf{y}_0), \bar\zeta_{\mathsf y}(\mathsf{x}_0, \mathsf{y}_0)\}$ is given by $A^{-1}(\mathsf{x}_0, \mathsf{y}_0)$.
\end{itemize}
In particular its intrinsic curvature on $\ginv(U)$ is strictly negative and  satisfies
\begin{equation} \label{eqn: det B}
\begin{split}
K(\ginv(\mathsf x,\mathsf y))^{-1} =- \frac{1}{4}\left((\mathsf yu_{\mathsf x\mathsf x} - u_{\mathsf y})(\mathsf yu_{\mathsf y\mathsf y} - u_{\mathsf y}) - \mathsf y^2u^2_{\mathsf x\mathsf y}\right)
\end{split}
\end{equation}
\end{lemma}
\begin{proof}
By Lemma \ref{lm:c1u} part (2), the surface $\Sigma$ is strictly convex on $\ginv(U)$. 
Moreover by Lemma \ref{lm:convexu} the function $u\circ\zeta^{-1}$ is strictly convex on $\zeta(U)$ so, applying Lemma \ref{lm:c1u} part (3),  $\Sigma$ is $C^1$ over $\ginv(U)$.
Notice from its definition in Lemma \ref{lm:c1u} part (1) that the map $\Gamma$ is $C^1$, moreover by \eqref{eq:c1} we have that $T_{\ginv(\mathsf{x},\mathsf{y})}\Sigma=T_{\bar\zeta(\mathsf{x},\mathsf{y})}\Hyp^2$ as subspaces of $\mathbb R^{2,1}$.

We claim that $A(\mathsf{x},\mathsf{y})$ is the representative matrix of the endomorphism $d_{(\mathsf{x},\mathsf{y})}\Gamma\circ(d_{({\mathsf{x},\mathsf{y}})}\bar\zeta)^{-1}$ with respect to the basis
$\bar\zeta_{\mathsf x},\bar\zeta_{\mathsf y}$. 
In fact,  differentiating the second and third equations defining $\ginv$, namely \eqref{eqn: def v 2} and \eqref{eqn: def v 3},  gives (omitting the evaluation point in the sequel for the sake of simplicity):
\begin{equation}
\begin{split}
\langle \ginv_{\mathsf x}, \zeta_{\mathsf x} \rangle + \langle \ginv, \zeta_{{\mathsf x}{\mathsf x}} \rangle &= u_{{\mathsf x}{\mathsf x}} \\
\langle \ginv_{\mathsf x}, \zeta_{\mathsf y} \rangle + \langle \ginv, \zeta_{{\mathsf x}{\mathsf y}} \rangle &= u_{{\mathsf x}{\mathsf y}} \\
\langle \ginv_{\mathsf y}, \zeta_{\mathsf x} \rangle + \langle \ginv, \zeta_{{\mathsf y}{\mathsf x}} \rangle &= u_{{\mathsf y}{\mathsf x}} \\
\langle \ginv_{\mathsf y}, \zeta_{\mathsf y} \rangle + \langle \ginv, \zeta_{{\mathsf y}{\mathsf y}} \rangle &= u_{{\mathsf y}{\mathsf y}} \\
\end{split}
\end{equation}
Verifying from the definition of $\zeta$ that $\zeta_{{\mathsf x}{\mathsf y}} = 0$ and $\zeta_{\mathsf x\mathsf x} = \zeta_{{\mathsf y}{\mathsf y}} = \zeta_{\mathsf y}/{\mathsf y}$ and using \eqref{eqn: def v 3}, we see that
\begin{equation}
\begin{split}
\begin{bmatrix} \langle \ginv_{\mathsf x}, \zeta_{\mathsf x} \rangle & \langle \ginv_{\mathsf x}, \zeta_{\mathsf y} \rangle  \\ \langle \ginv_{\mathsf y}, \zeta_{\mathsf x} \rangle & \langle \ginv_{\mathsf y}, \zeta_{\mathsf y} \rangle  \end{bmatrix} = \begin{bmatrix} u_{{\mathsf x}{\mathsf x}} - u_{\mathsf y}/{\mathsf y} & u_{{\mathsf x}{\mathsf y}}  \\ u_{{\mathsf y}{\mathsf x}} & u_{{\mathsf y}{\mathsf y}} - u_{\mathsf y}/{\mathsf y}  \end{bmatrix}=\frac{2}{\mathsf{y}}A=\frac{4}{|\zeta|}A
\end{split}
\end{equation}
Now we use that $\bar\zeta_{\mathsf x}\in \frac{\zeta_{\mathsf{x}}}{|\zeta|}+\mathrm{Span}(\zeta)$, $\bar\zeta_{\mathsf y}\in \frac{\zeta_{\mathsf{y}}}{|\zeta|}+\mathrm{Span}(\zeta)$, and  that $\langle\ginv_x,\zeta\rangle=\langle\ginv_y,\zeta\rangle=0$ (which is obtained by combining \eqref{eqn: def v 1}, \eqref{eqn: def v 2} and \eqref{eqn: def v 3}), to deduce that:

\begin{equation}\label{eq:cc}
\begin{split}
\begin{bmatrix} \langle \ginv_{\mathsf x}, \bar\zeta_{\mathsf x} \rangle & \langle \ginv_{\mathsf x}, \bar\zeta_{\mathsf y} \rangle  \\ \langle \ginv_{\mathsf y}, \bar\zeta_{\mathsf x} \rangle & \langle \ginv_{\mathsf y}, \bar \zeta_{\mathsf y} \rangle  \end{bmatrix} = \frac{4}{|\zeta|^2} A  
\end{split}
\end{equation}
Since $\bar\zeta_{\mathsf x}, \bar\zeta_{\mathsf y}$ is an orthogonal basis of $T_{\bar\zeta}\mathbb H^2$ and $||\bar\zeta_{\mathsf x}||^2=||\bar\zeta_{\mathsf y}||^2=4/|\zeta|^2$, 
the coordinates of a vector of $T_{\bar\zeta}\mathbb H^2$ with respect to this basis are given by the scalar product with the elements of the basis multiplied by $|\zeta|^2/4$.
Since $d\Gamma\circ(d\bar\zeta)^{-1}$ sends $\bar\zeta_{\mathsf x}$ to $\Gamma_{\mathsf x}$ and  $\bar\zeta_{\mathsf y}$ to $\Gamma_{\mathsf y}$, from \eqref{eq:cc}
we see that its representative matrix is given by $A$.

A consequence of this computation is that the differential of $\Gamma$ is injective under the hypothesis that $A$ is invertible, so $\Gamma$ realizes a $C^1$-diffeomorphism onto its image.
Since $G\circ\ginv=\bar\zeta$ we  see that $G=\bar\zeta\circ\ginv^{-1}$, so that the restriction of $G$ to $\ginv(U)$ is $C^1$. We deduce that $\ginv(U)$ is $C^2$ by Lemma \ref{lm:gaussc1}.

Finally using that $dG=d\bar\zeta\circ (d\ginv)^{-1}=(d\ginv\circ (d\bar\zeta)^{-1})^{-1}$ we conclude that the representative matrix of $dG$ with respect to the basis $\bar\zeta_{\mathsf x}, \bar\zeta_{\mathsf y}$ is
$A^{-1}$.
By the Gauss equation $K(\ginv(\mathsf x,\mathsf y))=\det A^{-1}$ and the result follows.
\end{proof}

Before finally giving the formula for our function $u$, we need two additional lemmas in order to deal with the conditions (3) and (4) in the list of Section \ref{subsec:outline}. The first lemma gives a formula for the extension of $u$ to $\infty$ and a condition for incompleteness of the corresponding surface under the additional assumption that $u$ is symmetric about the $\mathsf y$ axis. 

\begin{lemma} \label{lem: u symmetric incomplete}
Suppose that $u: \mathcal{H} \to \R$ is the parabolic support function of a convex entire spacelike surface $\Sigma$, and that $u(\mathsf x,\mathsf y)=u(-\mathsf x,\mathsf y)$. Suppose furthermore that the restriction of $u$ to the $\mathsf y$ axis is $C^2$ and strictly convex, and satisfies
\[
\int_1^\infty \left( u_{\mathsf y\mathsf y}-\frac{u_{\mathsf y}}{\mathsf y}\right) d\mathsf y < \infty~.
\]
Then $\Sigma$ is incomplete at $\theta_0 = \pi$ and 
$$\overline{u}(\infty) = \sup_{(0,\mathsf y) \in \mathcal{H}} \frac{u_\mathsf y(0,\mathsf y)}{2\mathsf y}~.$$
\end{lemma}

\begin{proof} Because of the bijection established in Corollary \ref{lem: convex duality}, $\Sigma$ also has the corresponding reflection symmetry, across the timelike plane $y=0$ spanned by the vectors $\vec{0}$ and $\vec{\pi}$. Since $\Sigma$ is convex, it follows that its null support planes in these directions contains the null support lines of its intersection with this timelike plane. Since $u$ is $C^2$ and strictly convex on the $\mathsf y$ axis, we can parametrize the intersection of $\Sigma$ with this plane by $\ginv(0,\mathsf y)$ defined by equations $(\ref{eqn: def v 1},~\ref{eqn: def v 2},~\ref{eqn: def v 3})$. Then since $\overline{u}(\infty)$ is the value of the elliptic support function of $\Sigma$ at ${{\theta_0}=}{\pi}$ by Proposition \ref{prop: ell/par} and since $\vec{\pi} = \zeta_{\mathsf y}/2\mathsf y$, we have
\begin{equation}
\begin{split}
\overline{u}(\infty) &= \sup_{\mathsf y} \langle \ginv(0,\mathsf y), \vec{\pi} \rangle \\
&= \sup_{\mathsf y} \langle \ginv(0,\mathsf y), \frac{\zeta_{\mathsf y}}{2\mathsf y} \rangle \\
&= \sup_{\mathsf y} \frac{u_{\mathsf y}(0,\mathsf y)}{2\mathsf y}~.
\end{split}
\end{equation}

For the incompleteness part of the lemma, it suffices to show that the length of $\ginv(0,\mathsf y)$ is finite as $\mathsf y$ approaches infinity, since we have already argued that this curve approaches the null support plane of $\Sigma$ in the direction $\vec{\pi}$. By the reflection symmetry, $\Gamma_{\mathsf y}$ is a multiple of $\overline{\zeta}_{\mathsf y}$, so its length is $\langle \Gamma_{\mathsf y}, \overline{\zeta}_{\mathsf y} \rangle / |\overline{\zeta}_{\mathsf y}|$. From equation \eqref{eq:cc}, which clearly remains true in the $\mathsf y$ direction even though $u$ is not differentiable in the $\mathsf x$ direction, this is equal to 
\[
\frac{1}{2} \left(u_{\mathsf{yy}} - \frac{u_{\mathsf y}}{\mathsf y} \right).
\]
Its integral is the length of the curve $\ginv(0, \mathsf y)$, so if the integral from one to infinity is finite, then the surface $\Sigma$ is incomplete at $\pi$.
%
\end{proof}

Finally, using the relationship \eqref{eq:null affine vs parabolic} between the elliptic and parabolic null support functions, we can translate the conditions on $\phi_+$ to conditions on the parabolic null support function $\psi_+$ of $\Sigma_+$. 

 
\begin{lemma}\label{lemma:mouse}
 Let $s$ be a homogeneous function, and let $\phi$ and $\psi$ be respectively its elliptic dehomogenization and parabolic dehomogenization, with point at infinity $\pi$.
 If 
\begin{equation} \label{eqn: psi+}
\begin{split}
\psi(\mathsf x) \leq \epsilon |\mathsf x|^{2-\alpha}
\end{split}
\end{equation}
in the complement of a compact set of $\R$, then
$$\phi(\theta) \leq \epsilon |\theta - \pi|^{\alpha}$$
in a neighbourhood of $\pi$.
 \end{lemma}
\begin{proof}
From Proposition \ref{prop: ell/par}, we have $\mathsf x = \tan(\theta/2)$ and
\[
\phi(\theta) = \frac{\psi(\mathsf x)}{1+\mathsf x^2}\leq \epsilon\frac{ |\mathsf x|^{2-\alpha}}{1+\mathsf x^2}< \epsilon |\mathsf x|^{-\alpha} ~.
\]
Using $1/\mathsf x = \tan(\pi/2 - \theta/2)$ and a linear bound $|t| \geq |\tan(t/2)| $ for $t$ sufficiently small, we get
$\phi(\theta)<\epsilon|\theta-\pi|^{\alpha}$ in a neighbourhood of $\pi$.
\end{proof}

\subsection{The explicit surface $\Sigma_+$}\label{subsec:explicit}

We are now ready to present the formula for the parabolic support function $u$ of $\Sigma_+$. For constants $0 < \beta, \gamma < 1$ and $M > 0$ to be determined, let
\begin{equation} \label{eqn: u}
u(\mathsf x,\mathsf y) = - M\mathsf y^\beta + \epsilon |\mathsf x|^{2- \alpha}(1+\mathsf y)^{-\gamma}~.
\end{equation}
Clearly $u(\mathsf x,0)$ satisfies \eqref{eqn: psi+}. Since $u$ is symmetric about the $\mathsf y$ axis {and its restriction to the $\mathsf y$ axis is $C^2$ and strictly convex}, we can use Lemma \ref{lem: u symmetric incomplete} to check incompleteness and the value of $\overline{u}(\infty)$, once we know that $u$ is convex. We also remark that $u$ is, inconveniently, not $C^2$ on the $\mathsf y$-axis.

Let $F$ be the operator
\begin{equation} \label{eqn: def of F}
\begin{split}
F(u) = -4K^{-1} = (\mathsf yu_{\mathsf x\mathsf x} - u_{\mathsf y})(\mathsf yu_{\mathsf y\mathsf y} - u_{\mathsf y}) - \mathsf y^2u^2_{\mathsf x\mathsf y}
\end{split}
\end{equation}
so that $-4/F(u)$ is the curvature of the surface constructed in Lemma \ref{lm:c2}.

\begin{lemma} \label{lem: F subsolution} 
Let $u$ be the function defined in \eqref{eqn: u}, and $F$ the operator defined in \eqref{eqn: def of F}.
\begin{itemize}
\item If $\gamma <1- \alpha$, then $F(u)>0$ for all points $(\mathsf x,\mathsf y)$ with $\mathsf x \neq 0$.
\item If $\gamma \leq \beta(1- \alpha)$, then for any $C>0$ there exists $M>0$ such that $F(u) \geq 4/C$ for all points $(\mathsf x,\mathsf y)$ with $\mathsf x \neq 0$.
\end{itemize}
\end{lemma}

\begin{proof} Let us compute
\begin{equation}\label{eq:first derivatives}
u_{\mathsf x}=\epsilon(2-\alpha)\mathsf x|\mathsf x|^{-\alpha}(1+\mathsf y)^{-\gamma}\qquad u_{\mathsf y}=-M\beta \mathsf y^{-1+\beta} - \epsilon\gamma |\mathsf x|^{2-\alpha}(1+\mathsf y)^{-1-\gamma}
\end{equation}
and
\begin{equation}\label{eq:second derivatives}
\begin{split}
u_{\mathsf x\mathsf x}&=\epsilon(2-\alpha)(1-\alpha) |\mathsf x|^{-\alpha}(1+\mathsf y)^{-\gamma}\\
u_{\mathsf x\mathsf y}&=-\epsilon(2-\alpha)\gamma\mathsf x|\mathsf x|^{-\alpha}(1+\mathsf y)^{-1-\gamma}\\
 u_{\mathsf y\mathsf y}&=M\beta(1-\beta) \mathsf y^{-2 + \beta} + \epsilon\gamma(1+\gamma) |\mathsf x|^{2-\alpha}(1+\mathsf y)^{-2-\gamma}~.
 \end{split}
\end{equation}
Hence we have:
\begin{equation*}
\begin{split}
F(u) =& (\mathsf y \orange{u_{\mathsf x\mathsf x}} - \blue{u_{\mathsf y}})(\mathsf y \red{u_{\mathsf y\mathsf y}} - \blue{u_{\mathsf y}}) - \mathsf y^2 \purple{u_{\mathsf x\mathsf y}}^2 \\
=& \left(\mathsf y\orange{\epsilon(2-\alpha)(1-\alpha) |\mathsf x|^{-\alpha}(1+\mathsf y)^{-\gamma}} + \blue{M\beta \mathsf y^{-1+\beta} + \epsilon\gamma |\mathsf x|^{2-\alpha}(1+\mathsf y)^{-1-\gamma}}\right) \cdot\\
& \left(\mathsf y(\red{M\beta(1-\beta) \mathsf y^{-2 + \beta} + \epsilon\gamma(1+\gamma) |\mathsf x|^{2-\alpha}(1+\mathsf y)^{-2-\gamma}}) + \right.\\
 & \qquad \left.\blue{M\beta \mathsf y^{-1+\beta} + \epsilon\gamma |\mathsf x|^{2-\alpha}(\mathsf y+1)^{-1-\gamma}}\right) \\
& - \mathsf y^2\left(\purple{-\epsilon(2-\alpha)\gamma\mathsf x|\mathsf x|^{-\alpha}(1+\mathsf y)^{-1-\gamma}}\right)^2
\end{split}
\end{equation*}

We drop the final term in each of the two parenthetical expressions and combine two equivalent terms in the second, and recolor:
\begin{equation*}
\begin{split}
F(u) \geq& \left(\red{\mathsf y\epsilon(2-\alpha)(1-\alpha) |\mathsf x|^{-\alpha}(\mathsf y+1)^{-\gamma}} + \blue{M\beta \mathsf y^{-1+\beta}} \right) \cdot \\
& \left( \blue{M\beta(2-\beta) \mathsf y^{-1+\beta}} + \red{\mathsf y\epsilon\gamma(1+\gamma) |\mathsf x|^{2-\alpha}(1+\mathsf y)^{-2-\gamma}}\right) \\
&- \red{\mathsf y^2\left(\epsilon(2-\alpha)\gamma|\mathsf x|^{1-\alpha}(1+\mathsf y)^{-1-\gamma}\right)^2} 
\end{split}
\end{equation*}
and now expand, dropping one of the four terms of the product:
\begin{equation*}
\begin{split}
F(u) \geq & \red{((1-\alpha)(1+\gamma) - (2 - \alpha)\gamma) \left(\mathsf y^2 \epsilon^2 (2-\alpha)\gamma|\mathsf x|^{2-2\alpha}(1+\mathsf y)^{-2-2\gamma}\right)} \\
& + \purple{\epsilon(2 - \alpha)(1 - \alpha) M\beta(2-\beta)|\mathsf x|^{- \alpha}(1+\mathsf y)^{-\gamma} \mathsf y^\beta} \\
& + \blue{M^2 \beta^2(2-\beta)\mathsf y^{-2+2\beta}}
\end{split}
\end{equation*}

As long as $\gamma <1- \alpha$, the first term is positive, and hence $F(u)$ is strictly positive. Assuming this, we now try to refine our choice of constants so that $F(u) \geq 4/C$. First of all, if $\mathsf y\leq 1$, the last term is at least $M^2\beta^2$, so for any $\beta$ we can choose $M$ to make this at least $4/C$. Hence, so long as we are free to choose $M$ sufficiently large, we need only bound $F(u)$ from below under the assumption that $\mathsf y \geq 1$. In this case, $(1+\mathsf y)^{-1} \geq \frac{1}{2}\mathsf y^{-1}$, and so for some $\delta(\epsilon, \alpha, \gamma)$ small we have

\begin{equation*}
\begin{split}
F(u) \geq& \delta\left(\frac{\alpha}{2-\alpha}|\mathsf x|^{2-2\alpha}\mathsf y^{-2\gamma} + \frac{2-2\alpha}{2-\alpha}M \beta |\mathsf x|^{-\alpha}\mathsf y^{\beta - \gamma}\right) \\
\geq & \delta \mathsf y^{(-2\gamma\alpha+2(1-\alpha)(\beta - \gamma))/(2-\alpha)}(M\beta)^{(2-2\alpha)/(2-\alpha)}
\end{split}
\end{equation*}
where we have used the weighted arithmetic-geometric inequality $\lambda a + (1-\lambda) b \geq a^{\lambda}b^{1-\lambda}$ in the last line. If we choose $\gamma$ such that
\[
-2\gamma\alpha+2(1-\alpha)(\beta - \gamma) \geq 0
\]
then for $M$ sufficiently large, $F(u) \geq 4/C$ for all $\mathsf y \geq 1$. A little algebra shows that it suffices to take $\gamma \leq \beta(1-\alpha)$.
\end{proof}
\
\begin{lemma}
{If $\gamma<1-\alpha$}, the function $u$ defined in \eqref{eqn: u} is the parabolic support function of a strictly convex entire   spacelike surface $\Sigma_+$.
\end{lemma}
The proof below shows that $\Sigma_+$ is a $C^1$ surface, as in the definition of spacelike surfaces. We will prove later (Lemma \ref{lem: exists sigma+}) that $\Sigma_+$ is in fact $C^2$, but this will require a little additional work.
\begin{proof}
By Corollary \ref{lem: convex duality}, in order to prove the existence of an achronal surface $\Sigma_+$ with parabolic support function $u$, we need to check that  the function $u\circ\zeta^{-1}$ is a  convex function of the intersection of the affine  plane $P_0$ of equation $x+z=2$ with the future cone  $J^+$. Since the intersection is convex, it is sufficient to prove that $u\circ\zeta^{-1}$ is locally convex around any point $(x,y,z)\in P_0\cap I^+$; then the function $u$ defined in \eqref{eqn: u}, which is clearly continuous up to the boundary of $\mathcal H$, is the unique closed convex extension to  $P_0\cap J^+$.

The strict convexity of $u\circ\zeta^{-1}$ at  $(x,y,z)\in P_0\cap I^+$ for $y\neq 0$ follows from Lemma \ref{lem: F subsolution}, which implies that the matrix \eqref{AA} is positive definite, and Lemma \ref{lm:convexu}.

To prove the strict convexity  at points of the form $(x,0,x+2)$, notice that $u(\mathsf x,\mathsf y)\geq  u_0(\mathsf x, \mathsf y)$ where
$u_0(\mathsf x, \mathsf y)=-M \mathsf{y}^ \beta$, with $u=u_0$ along $\mathsf x=0$.
We observe that $F(u_0)>0$ so $u_0\circ\zeta^{-1}$ is strictly convex again by Lemma \ref{lm:convexu}. 
We conclude that $u\circ\zeta^{-1}$ is strictly convex at the points of $P_0\cap I^+$ with $y=0$ as well. 

If $\Gamma:\mathcal H\to\Sigma_+$ is the map defined in Lemma \ref{lm:c1u}, we conclude that $\Sigma_+$ is $C^1$ and spacelike on the open subset  $\Gamma(\mathcal H)$, which consists of points admitting
spacelike support planes. 
In order to prove that $\Gamma$ is surjective, we show that $\Sigma_+$ is {acausal} as a consequence of Lemma \ref{lem: spacelike}.
Suppose $f$ is a polynomial of the given form with $f \leq u$ on $\overline{\mathcal{H}}$. Since $u$ has a singularity like $-\mathsf y^\beta$ near $\mathsf y=0$ and $f$ is smooth, we must certainly have $f < u$ on $\R$. Now at $\infty$, one checks $\overline{f}(\infty) = c$, and by Lemma \ref{lem: u symmetric incomplete}, we have $\overline{u}(\infty) = - \lim_{\mathsf y \to \infty} M \beta \mathsf y^{\beta - 2} = 0$. So if these were equal, then $f$ is independent of $\mathsf y$, so cannot be less than or equal to $-M\mathsf y^\beta$ on the $\mathsf y$ axis. We conclude that $f < u$ on $\R \cup \{\infty\}$, and so by Lemma \ref{lem: spacelike} that $\Sigma_+$ is  {acausal}. Hence by Corollary \ref{cor:spacelike iff acausal}, every point of $\Sigma_+$ has a spacelike support plane, that is, $\Gamma$ is surjective. This concludes the proof.
\end{proof}

By symmetry, $\Gamma(\{\mathsf x=0\})=\{y=0\}$ and so $\Gamma(\mathcal H\setminus\{\mathsf x=0\})=\Sigma_+\setminus\{y=0\}$. Since $u$ is $C^2$ on $\mathcal H\setminus\{\mathsf x=0\}$ and $F(u)>0$, Lemma \ref{lm:c2} implies that
$\Sigma_+\setminus\{y=0\}$ is $C^2$. We will prove now that $\Sigma_+$ is in fact $C^ 2$ everywhere.
First we prove the following lemma that roughly states that the shape operator of $\Sigma_+$ extends everywhere.
In order to make this statement more precise, notice that $T\Sigma_+\to\Sigma_+$ is a continuous vector bundle and a priori the shape operator $B$ is a section of $(T\Sigma_+)^*\otimes(T\Sigma_+)$
defined only on $\Sigma_+\setminus\{y=0\}$.

\begin{lemma}\label{lm:Bextends}
If $\gamma<1-\alpha$, the shape operator continuously extends to a section of $(T\Sigma_+)^*\otimes(T\Sigma_+)$ on the whole $\Sigma_+$.
\end{lemma}
\begin{proof}
We consider the continuous parameterization $\Gamma:\mathcal H\to \Sigma_+$ given by the inverse of the Gauss map.
Notice that for any $(\mathsf x,\mathsf y)\in\mathcal H$ the vectors  $\bar{\zeta}_{\mathsf x}(\mathsf x, \mathsf y)$,
$\bar\zeta_{\mathsf y}(\mathsf x, \mathsf y)$ form a basis of $T_{\Gamma(\mathsf x,\mathsf y)}\Sigma_+=T_{\bar\zeta(\mathsf x,\mathsf y)}\mathbb H^2$.
In this way a continuous frame of $T\Sigma_+$ is defined everywhere.
By Lemma \ref{lm:c2} the representative matrix of the shape operator $B$ with respect to this frame at a point $\Gamma(\mathsf x,\mathsf y)$ is given by
$ A(\mathsf x, \mathsf y)^{-1}$ where 
\begin{equation}\label{eq:nnnn}
 A(\mathsf x,\mathsf y)=\frac{1}{2}\begin{bmatrix} \mathsf yu_{\mathsf x\mathsf x} - u_{\mathsf y} & \mathsf yu_{\mathsf x\mathsf y}  \\ \mathsf yu_{\mathsf x\mathsf y} & \mathsf yu_{\mathsf y\mathsf y} - u_{\mathsf y}  \end{bmatrix}
(\mathsf x, \mathsf y)
\end{equation}

So in order to conclude it is sufficient to check that the entries of the inverse of \eqref{eq:nnnn} extend at $\mathsf x=0$.
From \eqref{eq:first derivatives} and \eqref{eq:second derivatives}, we observe that $u_{\mathsf x\mathsf y}(\mathsf x,\mathsf y)\rightarrow 0$ as $\mathsf x\to 0$, while $u_{\mathsf y}(\mathsf x,\mathsf y)\to -M\beta \mathsf y^{-1+\beta}$ and $u_{\mathsf y\mathsf y}(\mathsf x,\mathsf y)\to M\beta(1-\beta)\mathsf y^{-2+\beta}$, which implies that $\mathsf yu_{\mathsf y\mathsf y} - u_{\mathsf y}\to M\beta(2-\beta) \mathsf y^{-1+\beta}$.
On the other hand $u_{\mathsf x\mathsf x}(\mathsf x,\mathsf y) \sim |\mathsf x|^{-\alpha}$ diverges as $\mathsf x\to 0$, so we simply see that the matrix $A^{-1}(\mathsf x,\mathsf y)$ converges to the matrix
\[
\begin{bmatrix} 0 & 0\\0 & \frac{2}{M\beta(2-\beta)\mathsf y^{-1+\beta}}\end{bmatrix}\,.
\]
as $(\mathsf x, \mathsf y)\to (0, \mathsf y)$. 
\end{proof}

We will need to use an elementary lemma for which we give a short proof for the convenience of the reader.
\begin{lemma}\label{lm:c11}
Let $E=\{(x,y)\in\mathbb R^2\,|\, y= 0\}$ be the $x$-axis, and $f\in C^0(\mathbb R^2)\cap C^1(\mathbb R^2\setminus E)$.
Let us assume that $df$ extends to a continuous $1$-form $\omega=h(x,y)dx+k(x,y)dy$  on $\mathbb R^2$.
Then $f\in C^1(\mathbb R^2)$ and $df=\omega$ everywhere.
\end{lemma}
\begin{proof}
Let us consider the continuous function on $\mathbb R^2$
\[
K(x,y)=f(x,0)+\int_0^y k(x,t)dt\,
\]
so that  $K_y(x,y)=k(x,y)$ for all $(x,y)\in\mathbb R^2$.
Since $(K-f)_y$ vanishes on $\mathbb R^2\setminus E$, the function $K-f$ is constant on the half-lines of the form $\{x\}\times(0,+\infty)$ and on the half-lines of the form $\{x\}\times(-\infty,0)$.
Since $K-f$ is continuous and vanishes on $E$, it vanishes everywhere, which implies that $f$ admits $y$-derivative everywhere, that coincide with $k(x,y)$.

Similarly let us consider the function
\[
 H(x,y)=f(0,y)+\int_0^x h(s,y)ds\,
\]
so that $H_x(x,y)=h(x,y)$ everywhere.
We have that $H-f$ vanishes on $\mathbb R^2\setminus E$. Since $H-f$ is continuous on $\mathbb R^2$ we conclude that $f$ coincides with $H$ everywhere.
\end{proof}

\begin{lemma} \label{lem: exists sigma+} {If $\gamma<1-\alpha$, the function $u$ given by \eqref{eqn: u} is the parabolic support function of a $C^2$ convex entire  spacelike surface $\Sigma_+$. If moreover $\gamma \leq \beta(1- \alpha)$ and $M$ is chosen as in the second  point of Lemma \ref{lem: F subsolution}, then $\Sigma_+$ has curvature greater than or equal to $-C$.}
\end{lemma}
\begin{proof}
Let $f:\mathbb R^2\to\mathbb R$ be the $C^1$-function whose graph is $\Sigma_+$, 
and consider the corresponding $C^1$-parametrization  $\sigma(x,y)=(x,y, f(x,y))$.
Notice that $(G\circ\sigma)_x=B(\sigma(x,y))\cdot \sigma_x$.
Since by Lemma \ref{lm:Bextends} the shape operator $B$ extends to a continuous section of $(T\Sigma_+)^*\otimes(T\Sigma_+)$,  $\sigma_x$ extends to $y=0$ we conclude that
the function $(x,y)\to (G(\sigma(x,y)))_x$ continuously extends at $y=0$.
Similarly we have that $(x,y)\to (G(\sigma(x,y))_y$ continuously extends at $y=0$. 

 Lemma \ref{lm:c11} implies that $G\circ\sigma:\mathbb R^2\to\mathbb H^2$ is $C^1$.
 Since $\sigma$ is a $C^1$-parametrization of $\Sigma_+$, the map  $G:\Sigma_+\to\mathbb H^2$ is  $C^1$ as well, so by Lemma \ref{lm:gaussc1}, $\Sigma_+$ is $C^2$.


{Finally, under the choices of constants as in the second  point of Lemma \ref{lem: F subsolution} we have $F(u) \geq 4/C$. Therefore the curvature of $\Sigma_+$ is greater than or equal to $-C$ away from its intersection with the $y=0$ plane, and thus it is greater than or equal to $-C$ everywhere by continuity.}
\end{proof}

\subsection{Proof of Theorem \ref{thm: Holder incomplete}}\label{sec: conclusion of the proof}

At this point, we have assembled all the necessary components for a proof of the main theorem of this section, and so we end by simply putting them in order.

\begin{proof}[Proof of Theorem \ref{thm: Holder incomplete}]
{As already observed at the beginning of Section \ref{subsec:outline} (Remark \ref{remark theta0 pi}), we can assume $\theta_0=\pi$ and $\phi(\pi)=0$.}
Let $\Sigma_+$ be the surface whose existence is given by Lemma \ref{lem: exists sigma+}, where $\gamma \leq \beta(1- \alpha)$ and the constant  $M$ is chosen as in the second  point of Lemma \ref{lem: F subsolution}. 
From \eqref{eq:first derivatives} and \eqref{eq:second derivatives} we find $u_{\mathsf y\mathsf y}(0,\mathsf y) - u_{\mathsf y}(0,\mathsf y)/\mathsf y = M\beta(2-\beta)\mathsf y^{-2+\beta}$; moreover we have $u_{\mathsf y}(0,\mathsf y)/2\mathsf y = - M \beta \mathsf y^{-2+\beta}/2$. Hence Lemma \ref{lem: u symmetric incomplete} shows that $\Sigma_+$ is incomplete at $\pi$ and its elliptic null support function $\phi_+$ has value 0 at $\pi$. Moreover, since its parabolic null support function is  
$$\psi_+(\mathsf x)=\epsilon|\mathsf x|^{2-\alpha}~,$$
by Lemma \ref{lemma:mouse} {and the condition \eqref{eq:inc power}, we have that $\phi_+(\theta)\leq \epsilon|\theta-\pi|^{\alpha} < \phi(\theta)$ for $\theta$ in a neighbourhood $U$ of $\pi$. In addition, we have $\phi_+(\pi)=\phi(\pi)=0$.}

{Analogously to the proof of Corollary \ref{cor easy}, we remark that we are not yet in the hypothesis of Corollary \ref{cor:complete comparison}, since we  need  the inequality $\phi_+\leq \phi$ to hold globally on $\mathbb S^1$. This issue is solved by replacing  $\Sigma_+$ with a translate in the direction of $\vec{\pi}$, which has the effect of changing its parabolic null support function $\psi_+$ by adding the constant $-2c$ and leaving the   value of the elliptic null support function $\phi_+$ at $\pi$ unchanged. Since $\psi_+$ is continuous, hence bounded above, and $\psi$ is lower semicontinuous, hence bounded below, by choosing $c$ large enough, the translate $\Sigma'_+$ of $\Sigma_+$ satisfies $\psi'_+<\psi$, and therefore $\phi_+'\leq \phi$ on the whole $\mathbb S^1$. By Corollary \ref{cor:complete comparison} applied to $\Sigma$ and $\Sigma_+'$, we conclude that $\Sigma$ is incomplete at $\pi$.}
\end{proof}

\section{Incompleteness II: one-sided superlogarithmic condition}\label{sec:xlogx incomplete}
The goal of this section is to prove Theorem \ref{thm: xlogx incomplete}, which we state here as usual in a slightly stronger version with respect to the introduction, using directional incompleteness.

\begin{reptheorem}{thm: xlogx incomplete} [One-sided superlogarithmic condition -- local version]  Let $\phi: \mathbb S^1 \to \R \cup \{+\infty\}$ be lower semicontinuous and finite on at least three points. Suppose $\theta_0 \in \mathbb S^1$ is such that $\phi(\theta_0)<+\infty$ and there exist a neighborhood $U$ of $\theta_0$ and a constant $\epsilon > 0$ such that 
\begin{equation} \label{eq:one sided log}\tag{Inc'}
\begin{cases}
\phi(\theta)= +\infty & \text{ if }\theta\text{ is on one side of }\theta_0  \\ 
\phi(\theta)\geq \phi(\theta_0) + \epsilon|(\theta - \theta_0)\log|\theta - \theta_0|| &  \text{ if }\theta\text{ is on the other one side of }\theta_0.
\end{cases}
\end{equation}
for every $\theta \in U$. 

If $\Sigma$ is a convex entire spacelike  surface in $\R^{2,1}$ with null support function $\phi$ and curvature bounded above by a negative constant, then $\Sigma$ is incomplete at $\theta_0$.
\end{reptheorem}

To prove the incompleteness criterion given in Theorem \ref{thm: xlogx incomplete},
we will construct an incomplete entire hyperbolic surface  $\Sigma_\lambda$ (depending on a real parameter $\lambda)$ invariant under a one-parameter glide hyperbolic group,  study its null support function, and apply the comparison principle (Corollary \ref{cor:complete comparison}).


\subsection{The invariant surface $\Sigma_\lambda$}

Let us fix  $\lambda\in\R$, and consider the one-parameter group $A^\lambda=\{A^\lambda_s\}_{s\in\mathbb R}$ of affine isometries of Minkowski space defined by 
$$A^\lambda_s\begin{bmatrix}x_1\\x_2\\x_3\end{bmatrix}=\begin{bmatrix}x_1+\lambda s\\ \cosh(s)x_2+\sinh(s) x_3\\ \sinh(s) x_2+\cosh(s)x_3\end{bmatrix}~.$$
Let us define the map $X^\lambda:(0,+\infty)\times\R \to \R^{2,1}$ given by
\[
     X^\lambda(t,s)=\begin{bmatrix}  \sqrt{1+\lambda^2}( t -\coth(t) ) +\lambda s\\ \sqrt{1+\lambda^2}{\sinh(s)}/{\sinh(t)} \\ \sqrt{1+\lambda^2}{\cosh(s)}/{\sinh(t)} \end{bmatrix}
\]
We have that
\[
A^\lambda_{s_0}X^\lambda(t,s)=X^\lambda(t, s+s_0)~.
\]
\begin{remark}
When $\lambda=0$ in the above expression, we recover the embedding $X=X^0$ of the semitrough that we described in Example \ref{ex:semitrough}. We have already observed there that the semitrough is complete. We will see that $X^\lambda$ is a proper spacelike immersion, whose image $\Sigma_\lambda$ is an entire spacelike surface of constant curvature $-1$ which is complete if and only if $\lambda=0$. The incompleteness of $\Sigma_\lambda$ for $\lambda\neq 0$ is the fundamental property that we will apply to prove Theorem \ref{thm: xlogx incomplete}.
\end{remark}

{\begin{remark}\label{rmk sign lambda}
Observe that $\Sigma_{-\lambda}=R(\Sigma_\lambda)$, where $R=\mathrm{diag}(1,-1,1)$ is the reflection in the plane $y=0$. Indeed, we immediately see from the definition that $X^{-\lambda}(t,s)=R\circ X^\lambda(t,-s)$.
\end{remark}}

\begin{lemma}\label{lemmaXproper}
For any $\lambda\in\R$, the image of $X^\lambda$ is an entire spacelike surface.
\end{lemma}
\begin{proof}
An explicit computation shows that
\[
\partial_t X^\lambda(t,s)=\sqrt{1+\lambda^2}\frac{\cosh(t)}{\sinh^2(t)}\begin{bmatrix}\cosh(t)\\ -\sinh(s)\\ -\cosh(s)\end{bmatrix}\qquad 
\partial_s X^\lambda(t,s)=\frac{1}{\sinh(t)}\begin{bmatrix}\lambda \sinh(t) \\ \sqrt{1+\lambda^2}\cosh(s)\\  \sqrt{1+\lambda^2}\sinh(s)\end{bmatrix}
\]
so that the coefficients of the pull-back of the Minkowski product are
\begin{align*}
E&=\langle \partial_t X^\lambda, \partial_t X^\lambda\rangle=(1+\lambda^2)\coth^2(t)\\
F&=\langle \partial_t X^\lambda, \partial_s X^\lambda\rangle=\lambda\sqrt{1+\lambda^2}\coth^2(t)\\
G&=\langle \partial_s X^\lambda, \partial_s X^\lambda\rangle=\lambda^2\coth^2(t)+\frac{1}{\sinh^2(t)}\,.
\end{align*}
Hence we get
\begin{equation}\label{eq:detIlambda}
EG-F^2=(1+\lambda^2)\frac{\coth^2(t)}{\sinh^2(t)}>0
\end{equation}
which shows that the pull-back metric is Riemannian, and the map $X^\lambda$ is a spacelike immersion.

Let us prove that $X^\lambda$ is proper. By contradiction let us assume there exists a diverging sequence $(t_n, s_n)\in\mathbb R_+\times\mathbb R$ such that $X^\lambda(t_n,s_n)$ is converging in $\mathbb R^{2,1}$.
First let us assume that $s_n\to+\infty$.
Since $X^\lambda(t_n,s_n)$ stays in a compact region of $\R^{2,1}$, from boundedness of the second or third coordinate we deduce that $e^{s_n-t_n}$ is bounded. That is, that there exists $C>0$ such that $t_n>s_n-C$. Therefore $\coth(t_n)$ is bounded above by some constant $C'$, and the first component of $X_\lambda(t_n,s_n)$ is  larger than $(\sqrt{1+\lambda^2}+\lambda)s_n- (C+C')\sqrt{1+\lambda^2}$ for $n$ sufficiently large. Since this quantity tends to $+\infty$, we get a contradiction.

Similarly, if $s_n\to-\infty$, from the third component we obtain that $e^{-s_n-t_n}$ is bounded, hence $t_n>-s_n-C$, and the third component would be larger than $(\sqrt{1+\lambda^2}-\lambda)(-s_n)- (C+C')\sqrt{1+\lambda^2}$, which is impossible. 

Finally let us consider the case where $s_n$ is bounded and $t_n$ is diverging in $\R_+$.
If $t_n\to 0$ then the third component of $X(t_n, s_n)$ diverges.
On the other hand, if $t_n\to+\infty$ then the first component diverges.

We have thus shown that $X^\lambda$ is a proper spacelike immersion. By \cite[Proposition 1.10]{Bonsante:2019aa}, $X^\lambda$ is a proper embedding, and its image is an entire spacelike surface.
This concludes the proof.
\end{proof}

Let us call $\Sigma_\lambda$ the image of the embedding $X^\lambda$, which is an entire spacelike surface. We now show that it is hyperbolic:

\begin{lemma}\label{lemma:curv -1}
For any $\lambda\in\R$, the surface $\Sigma_\lambda$ has constant curvature $-1$.
\end{lemma}
\begin{proof}
A direct computation shows that the future pointing unit normal vector of $\Sigma_\lambda$ at $X^\lambda(t,s)$ is 
\begin{equation}\label{eq:gaussmap}
   N(t,s)=\begin{bmatrix}{-\sqrt{1+\lambda^2}}/{\sinh(t)}\\ \sqrt{1+\lambda^2}\coth(t)\sinh(s)+\lambda\cosh(s)\\ \sqrt{1+\lambda^2}\coth(t)\cosh(s)+\lambda\sinh(s)\end{bmatrix}
\end{equation}
On the other hand the Hessian of $X^\lambda$ is the following $\R^{2,1}$-valued symmetric form:
\[
\begin{split}
D^2X_\lambda&=(\partial_{tt}X_\lambda) dt^2+2 (\partial_{ts}X_\lambda) dtds+(\partial_{ss}X_\lambda) ds^2=\\
&=\sqrt{1+\lambda^2}\left(\begin{bmatrix}-2\frac{\coth(t)}{\sinh^2(t)}\\\frac{1+\cosh^2(t)}{\sinh^3(t)}\sinh(s)\\\frac{1+\cosh^2(t)}{\sinh^3(t)}\cosh(s)\end{bmatrix}dt^2
-2\begin{bmatrix}0\\\frac{\coth(t)}{\sinh(t)}\cosh(s)\\ \frac{\coth(t)}{\sinh(t)}\sinh(s)\end{bmatrix}dtds +\begin{bmatrix}0\\ \frac{1}{\sinh(t)}\sinh(s)\\ \frac{1}{\sinh(t)}\cosh(s)\end{bmatrix}ds^2
\right)
\end{split}
\]
Recall that the second fundamental form is defined as  {$\II=-\langle D^2 X_\lambda, N\rangle$}. It follows that the coefficients of the second fundamental form are given by

\begin{align*}
e&=g=(1+\lambda^2)\frac{\coth(t)}{\sinh(t)}\\
f&=\lambda\sqrt{1+\lambda^2}\frac{\coth(t)}{\sinh(t)}~.\\
\end{align*}
We have
\[
eg-f^2=\frac{\coth^2(t)}{\sinh^2(t)}[(1+\lambda^2)^2-\lambda^2(1+\lambda^2)]=(1+\lambda^2)\frac{\coth^2(t)}{\sinh^2(t)}
\]
By the Gauss equation and Equation \eqref{eq:detIlambda} we have that the curvature of $\Sigma$ is given by
\[
K=-\frac{eg-f^2}{EG-F^2}=-1
\]
as claimed.
\end{proof}

By Proposition \ref{prop:Gauss map properties}, $\Sigma_\lambda$ is either convex or concave; the proof of Lemma \ref{lemma:curv -1} actually showed that the second fundamental form is positive definite, hence $\Sigma_\lambda$  is convex (Remark \ref{rmk convention}).

\begin{lemma}\label{lm:xlnx-inc}
The surface $\Sigma_\lambda$ is incomplete at $\theta_0=\pi/2$ if $\lambda<0$, and at $\theta_0=-\pi/2$ if $\lambda>0$.
\end{lemma}

\begin{proof}
{By Remark \ref{rmk sign lambda}, it suffices to prove the statement for $\lambda>0$.}
Let us consider the proper path $\gamma:[1,+\infty)\to\Sigma$ given by
\[
   \gamma(\tau)=X^\lambda(\lambda\tau, -\sqrt{1+\lambda^2}\tau)
\]
We claim that $\gamma$ has finite length.
Indeed, a direct computation shows that
$$\dot\gamma(\tau)=\lambda\partial_tX^\lambda(\lambda\tau, -\sqrt{1+\lambda^2}\tau)-\sqrt{1+\lambda^2}\partial_sX_\lambda(\lambda\tau, -\sqrt{1+\lambda^2}\tau)~.$$
Using the coefficients of the first fundamental form that we derived in the proof of Lemma \ref{lemmaXproper}, we have
\begin{align*}
||\dot\gamma(\tau)||^2&=\lambda^2E(\lambda\tau, -\sqrt{1+\lambda^2}\tau)+(1+\lambda^2)G(\lambda\tau, -\sqrt{1+\lambda^2}\tau)-2\lambda\sqrt{1+\lambda^2}F(\lambda\tau, -\sqrt{1+\lambda^2}\tau)\\
&=2\lambda^2(1+\lambda^2)\coth^2(\lambda\tau)+\frac{1+\lambda^2}{\sinh^2(\lambda\tau)}-2\lambda^2(1+\lambda^2)\coth^2(\lambda\tau)=
\frac{1+\lambda^2}{\sinh^2(\lambda\tau)}~.
\end{align*}
This shows that 
$$\int_1^{+\infty}||\dot\gamma(\tau)||d\tau<+\infty~.$$

It remains to show that the direction of incompleteness is $\theta_0=-\pi/2$. One sees directly from the definition of $X^\lambda$ that the null planes $z=\pm y$ are support planes for $\Sigma^\lambda$ (for all $\lambda$). Moreover, we can check that
$$\langle \gamma(\tau),(0,-1,1)\rangle=-\sqrt{1+\lambda^2}\frac{e^{-\sqrt{1+\lambda^2}\tau}}{\sinh(\lambda\tau)}\to 0$$
as $\tau\to+\infty$. Hence we can apply
 Lemma \ref{lemma:incompleteness by cauchy}, and conclude that for $\lambda>0$, $\Sigma_\lambda$ is incomplete at $-\pi/2$.
 \end{proof}


\subsection{Null support function}

We will now compute the null support function of the surface $\Sigma_\lambda$. We will use again the parabolic null support function. However, since we have shown in Lemma \ref{lm:xlnx-inc} that $\Sigma_\lambda$ is incomplete at $-\pi/2$, we will use a variation of the map $\zeta$ defined in Equation \eqref{eq:defi p}, so as to make the point $\theta_0=-\pi/2$  correspond to $\mathsf x_0=0$, and the point $\pi/2$ correspond to the direction at infinity. Namely, we just compose the map $ \zeta:\overline{\mathcal H}\to\R^{2,1}$ with a rotation of angle $-\pi/2$. That is, let us define:
\begin{equation}\label{eq:defi q}
{\xi}(\mathsf x,\mathsf y) := \begin{bmatrix}  2\mathsf x \\ -1+ \mathsf x^2 + \mathsf y^2 \\ 1+ \mathsf x^2 + \mathsf y^2 \end{bmatrix}
\end{equation}
We then define the parabolic null support function \emph{with point at infinity $\pi/2$} analogously to \eqref{eq: par null supp fun} with respect to $ \xi$ instead of $ \zeta$, and we will denote it by $\ph$ hereafter:
\begin{equation} \label{eq: par null supp fun2}
\ph(\mathsf x) := \sup_{{p} \in \Sigma} \langle {p}, {\xi}(\mathsf x,0) \rangle~.
\end{equation}

Notice that $ \xi(\mathsf x,0)$  parameterizes all the lightlike directions $\vec\theta$ except $(0,1,1)$, which is however the projective limit of ${\xi}(\mathsf x,0)$ as $\mathsf x\to\infty$. Observe that with this convention, $\mathsf x=\tan(\theta+\pi/2)$. 

We are interested in the parabolic null support function of $\Sigma_\lambda$ near the point $\mathsf x_0=0$, which corresponds to $\theta_0=-\pi/2$ and to $ \xi(0,0)=(0,-1,1)$. The fundamental result we have to achieve is the following, which holds for any $\lambda\in\R$, although  we are mostly concerned with the case where $\Sigma_\lambda$ is incomplete at $-\pi/2$, namely (by Lemma \ref{lm:xlnx-inc}) when $\lambda>0$.

\begin{prop}\label{pr:xlnx}
For any $\lambda\in\R$, let $\ph_\lambda$ be the parabolic null support function of the surface $\Sigma_\lambda$, with point at infinity $\pi/2$. Then
\[
\ph_\lambda(\mathsf x)=\left\{\begin{array}{ll}-2\lambda |\mathsf x|\log |\mathsf x| & \textrm{if } \mathsf x<0\\
0 & \textrm{if } \mathsf x=0\\
+\infty & \textrm{if } \mathsf x>0.\end{array}\right.
\]
\end{prop}

Before proving Proposition \ref{pr:xlnx} we notice that the invariance of $\Sigma^\lambda$ under the action of the one-parameter group $A^\lambda=\{A^\lambda_s\}_{s\in\mathbb R}$ gives a strong constraint on its null support function.
\begin{lemma}\label{lm:eqvv}
Let $\ph_\lambda$ be the parabolic null support function of any convex entire achronal surface in $\R^{2,1}$ invariant under the one-parameter group $A^\lambda$, with point at infinity $\pi/2$. Then for every $s\in\mathbb R$ and $\mathsf x\in\mathbb R$ the following identity holds:
\[
\ph_\lambda(e^{s}\mathsf x)=e^{s}(\ph_\lambda(\mathsf x)+2\lambda s \mathsf x)
\]
\end{lemma}
\begin{proof}
Let us call $\Sigma$ the invariant surface. Denote by $L_s$ and $\tau_s$  the linear and the translational part of $A_s$ respectively, namely
$$L_s=\begin{bmatrix}1&0&0\\ 0&\cosh(s)&\sinh(s)\\ 0&\sinh(s) &\cosh(s)\end{bmatrix}~,\qquad \tau_s=\begin{bmatrix}\lambda s\\ 0\\ 0\end{bmatrix}~.$$
By the invariance  under the action of the group $A^\lambda$, we have
\begin{align*}
    \ph_\lambda(\mathsf x)&=\sup_{p\in \Sigma}\langle  \xi(\mathsf x,0),p\rangle\\
    &=\sup_{p\in \Sigma}\langle  \xi(\mathsf x,0), A_s(p)\rangle\\
    &= \sup_{p\in \Sigma}\langle  \xi(\mathsf x,0), L_s(p)+\tau_s\rangle\\
    &= \sup_{p\in \Sigma}\langle  \xi(\mathsf x,0), L_s(p)\rangle+\langle \xi(\mathsf x,0),\tau_s\rangle\\
    &= \sup_{p\in \Sigma}\langle  L_{-s}(\xi(\mathsf x,0)), p\rangle+2\lambda\mathsf x s
\end{align*}
By a direct computation,
\[
L_{-s}( \xi(\mathsf x,0))=e^{s} \xi( e^{-s}\mathsf x,0)~.
\]
Hence we obtain
\[
  \ph_\lambda(\mathsf x)= e^{s}\ph_\lambda(e^{-s}\mathsf x)+2\lambda \mathsf x s\,
\]
and the result follows.
\end{proof}
\begin{proof}[Proof of Proposition \ref{pr:xlnx}]
{We have already observed in the proof of Lemma \label{lm:xlnx-inc} that the null planes $z=\pm y$ are support planes of $\Sigma_\lambda$. This implies that the elliptic null support function of $\Sigma^\lambda$ at $\theta_0=\pm \pi/2$ equals $0$. Comparing the elliptic and parabolic null support functions as in Proposition \ref{prop: ell/par}, one has $\ph_\lambda(0)=0$.}

Moreover, from Equation \eqref{eq:gaussmap}, the image of the Gauss map is contained in the half plane defined by $x\geq 0$ in the hyperboloid model of $\Hyp^2$, and therefore the elliptic support function takes the value $+\infty$ on the half-disc $\{(x,y,1)\,|\,x\geq 0\}$. 
{By convexity}, the elliptic null support function equals $+\infty$ at every $\theta\in (-\pi/2,\pi/2)$, that is $\ph_\lambda(\mathsf x)=+\infty$ for every $\mathsf x>0$.

Finally, we claim that $\ph_\lambda(-1)=0$.  From Lemma \ref{lm:eqvv}, this will imply that
$$\ph_\lambda(-e^{s})=-2\lambda se^{s}
$$
and the statement will be proved.

By definition of $\ph_\lambda$, since $\xi(-1,0)=(-2,0,2)$, we have 
\begin{equation}\label{eq:8am}
\ph_\lambda(-1)=\sup_{p\in\Sigma_\lambda}\langle p,(-2,0,2)\rangle~.
\end{equation}
As an additional preliminary remark, we observe that it suffices to show that $\ph_\lambda(-1)=0$ for $\lambda\geq 0$. Indeed, by Remark \ref{rmk sign lambda}, $\Sigma_{-\lambda}=R(\Sigma_\lambda)$ where $R$ is the reflection in the plane $y=0$. Hence
$$\ph_{-\lambda}(-1)=\sup_{p\in\Sigma_\lambda}\langle R(p),(-2,0,2)\rangle=\sup_{p\in\Sigma_\lambda}\langle p,(-2,0,2)\rangle=\ph_\lambda(-1)$$
where we have used that $R$ fixes $(-2,0,2)$.

We are now ready to provide the final computation. From \eqref{eq:8am} and the parameterization of $\Sigma_\lambda$, we have
\begin{align*}
   \ph_\lambda(-1)&= 
   \sup_{(t,s)\in\mathbb R_+\times\mathbb R}\langle X^\lambda(t,s),(-2,0,2)\rangle\\
   &=
   2\sup_{(t,s)\in\mathbb R_+\times\mathbb R}\left(\sqrt{1+\lambda^2}\left(\coth(t)-\frac{\cosh(s)}{\sinh(t)}-t\right)-\lambda s\right)\,.
\end{align*}
For every $t\in\mathbb R_+$ let us consider the function $F_t:\mathbb R\to\mathbb R$ defined by 
$$F_t(s)=\sqrt{1+\lambda^2}\left(\coth(t)-\frac{\cosh(s)}{\sinh(t)}-t\right)-\lambda s~.$$
This function has a global maximum at the point $s_{\max}=s_{\max}(t)$ such that 
\begin{equation}\label{eq:smax}
\sinh(s_{\max})=-({\lambda}/{\sqrt{1+\lambda^2}})\sinh(t)~.
\end{equation}
 The corresponding maximum value  is
\begin{align*}
G(t):&=\sup_{s} F_t(s)=F(s_{\max}(t)) \\
&=\sqrt{1+\lambda^2}\left(\coth(t)-\frac{\cosh (s_{\max}(t))}{\sinh (t)}-t\right)-\lambda s_{\max}(t)\\
&=\sqrt{1+\lambda^2}(\coth(t)-t)+\lambda(\coth(s_{\max}(t))-s_{\max}(t))\\
&=\sqrt{1+\lambda^2}g(t)+\lambda g(s_{\max}(t))~,
\end{align*}
where we have set $g(t):=\coth(t)-t$. We will show that $G'(t)<0$. For this, let us first observe that $g'(t)=-\coth^2(t)$ and, differentiating \eqref{eq:smax}, we have the relation
$$\cosh(s_{\max}(t))s_{\max}'(t)=-\frac{\lambda}{\sqrt{1+\lambda^2}}\cosh (t)~.$$
Using these two relations, we obtain:
\begin{align*}
G'(t)&=-\sqrt{1+\lambda^2}\frac{\cosh^2(t)}{\sinh^2(t)}-\lambda\frac{\cosh^2 (s_{\max}(t))}{\sinh^2 (s_{\max}(t))}(s_{\max}'(t))\\
&={\cosh (t)}\left(-{\sqrt{1+\lambda^2}}\frac{\cosh(t)}{\sinh^2(t)}+\frac{\lambda^2}{{\sqrt{1+\lambda^2}}}\frac{\cosh (s_{\max}(t))}{\sinh^2(s_{\max}(t))}\right)\\
&=\sqrt{1+\lambda^2}\frac{\cosh (t)}{\sinh^2(t)}(-\cosh(t)+\cosh(s_{\max}(t)))<0
\end{align*}
where in the last inequality we have used that, from \eqref{eq:smax}, $|s_{\max}(t)|<t$.

This shows that $\psi_\lambda(-1)=2\lim_{t\to 0^+}G(t)$. To compute this limit, observe that, by  \eqref{eq:smax},
\[
   s_{\max}(t)=-\frac{\lambda}{\sqrt{1+\lambda^2}} t+O(t^2)~.
\]
From the expression of $G(t)$, we then have:
\begin{align*}
G(t)&=\sqrt{1+\lambda^2}\left(\frac{\cosh(t)-\cosh (s_{\max}(t))}{\sinh(t)}\right)-\sqrt{1+\lambda^2}t-\lambda s_{\max}(t)\\
&=\sqrt{1+\lambda^2}\left(\frac{ct^2+O(t^3)}{t+O(t^3)}\right)+O(t)~.
\end{align*}
This shows that  $\lim_{t\to 0^+} G(t)=0$ and thus concludes the proof.
\end{proof}

\subsection{The surface $\Sigma_+$}
By Proposition \ref{pr:xlnx}, the invariant surface has the correct local behaviour in order to prove Theorem \ref{thm: xlogx incomplete}, using the comparison result of Corollary \ref{cor:complete comparison}. The first step is, analogously to Lemmas \ref{lemma:mouse0} and \ref{lemma:mouse}, comparing the elliptic and parabolic null support functions.

\begin{lemma}\label{lemma:mouse2}
 Let $s$ be a homogeneous function, and let $\phi$ and $\ph$ be respectively its elliptic dehomogenization and parabolic dehomogenization, the latter with point at infinity $\pi/2$. If
$$
\ph(\mathsf x) \leq \epsilon |\mathsf x\log|\mathsf x||
$$
for all $\mathsf x$ in some (one-sided) neighbourhood of  $\mathsf x_0=0$, then
$$\phi(\theta) \leq \epsilon |(\theta - \theta_0)\log|\theta - \theta_0||$$
for all $\theta$ in some (one-sided) neighbourhood $\theta_0=-\pi/2$.
 \end{lemma}
\begin{proof}
The proof is very similar to Lemma \ref{lemma:mouse0}, except that we consider the point at infinity $\pi/2$. Analogously to Proposition \ref{prop: ell/par}, we have $\mathsf x = \tan((1/2)(\theta-\theta_0))$, hence for $\mathsf x$ close to $\mathsf x_0=0$, the linear bound $|\mathsf x|\leq |\theta-\theta_0|$ holds. Thus
\[
\phi(\theta) = \frac{\ph(\mathsf x)}{1+\mathsf x^2}\leq  \frac{\epsilon}{1+\mathsf x^2}|\mathsf x\log|\mathsf x||\leq \epsilon|(\theta - \theta_0)\log|\theta - \theta_0||~.
\]
This concludes the proof.
\end{proof}

So, let $\phi$ be a function as in the hypothesis of Theorem \ref{thm: xlogx incomplete}. Of course, up to applying an isometry of $\R^{2,1}$, we can assume $\theta_0=-\pi/2$ and $\phi(\theta_0)=0$. Let $\phi_\lambda$ the elliptic null support function of the invariant surface $\Sigma_\lambda$, for $\lambda=\epsilon/2$. We thus have $\phi_\lambda(\theta_0)=\phi(\theta_0)$ and $\phi_\lambda\leq \phi$ in a neighbourhood of $\theta_0=-\pi/2$. However, in order to apply Corollary \ref{cor:complete comparison}, with $\phi_\lambda$ playing the role of $\phi_+$, we would need the inequality $\phi_\lambda\leq \phi$ to hold globally. 

This is an issue that we have already taken care of in the proof of Theorem \ref{thm: Holder incomplete}, by a translation of the surface. In this setting, however, there is one more step to do. Indeed, we have that $\phi_\lambda$ equals $+\infty$ in the whole interval $(-\pi/2,\pi/2)$, while the function $\phi$ is by hypothesis equal to $+\infty$ only in a one-sided neighbourhood (we can assume it is a right neighbourhood) of $\theta_0=-\pi/2$. Translating the surface has an effect of adding a bounded function to its null support function, hence will not change this situation. We will instead modify the surface $\Sigma_\lambda$ by applying a parabolic isometry fixing the direction $\vec{\theta_0}$, which has the effect of ``shrinking'' the interval on which $\phi_\lambda$ is infinite. We provide the details of this argument in the following proof.

\begin{prop}\label{prop:barrier xlogx}
For any $\epsilon,\kappa>0$ and $\theta_0,\theta_1\in  \mathbb S^1$, there exists a convex entire spacelike surface $\Sigma_+$ of constant curvature $-\kappa$ which is incomplete at $\theta_0$ and whose null support function $\phi_+$ satisfies
\begin{itemize}
\item $\phi_+(\theta_0)=0$;
\item the set $\{\phi_+=+\infty\}$ equals the right open interval between $\theta_0$ and $\theta_1$;
\item $\phi_+$ is continuous on the subset $\{\phi_+<+\infty\}$ where it is finite;
\item $\phi_+(\theta)\leq \epsilon|(\theta-\theta_0)\log|\theta-\theta_0||$ in a left one-sided neighbourhood of $\theta_0$.
\end{itemize}
\end{prop}
\begin{proof}
First, observe that it suffices to prove the statement for $\kappa=1$, since applying a dilatation has the effect of multiplying the support function by a constant (Lemma \ref{lemma:rescale par null}). 

Now, assuming $\theta_0=-\pi/2$, we will prove the proposition by applying to the surface $\Sigma_\lambda$ studied above (for $\lambda>0$ to be chosen later) a parabolic linear isometry $L$ that fixes the direction $\vec{\theta_0}=(0,-1,1)$. Let $\Sigma_\lambda'=L(\Sigma_\lambda)$ and $\phi_\lambda'$ be the elliptic null support function of $\Sigma_\lambda'$. Clearly $L$ fixes the plane $z=-y$, which is a null support plane of $\Sigma_\lambda$, and it is therefore also a null support plane of $\Sigma_\lambda'$. Hence $\phi_\lambda'(\theta_0)=\phi_\lambda(\theta_0)=0$. Moreover, the action induced by $L$ on the unit circle $\mathbb S^1$ clearly maps $\{\theta\,|\,\phi_\lambda(\theta)=+\infty\}=(-\pi/2,\pi/2)$ to $\{\theta\,|\,\phi_\lambda'(\theta)=+\infty\}$. Choosing $L$ suitably, we can therefore achieve the latter set to be equal to the right interval between $\theta_0=-\pi/2$ and $\theta_1$, for any $\theta_1\in \mathbb S^1$. 

Observe that by Lemma \ref{lemma:mouse2}, the null support function $\phi_\lambda$ is smaller than  $\lambda|(\theta+\pi/2)\log|\theta+\pi/2||$ on the left of $-\pi/2$. It remains to show that $\phi_\lambda'$ still has the same behaviour in a left one-sided neighbourhood of $-\pi/2$. This is done, avoiding unnecessary computations, by considering the parabolic null support function \emph{with point at infinity $-\pi/2$.}

For this purpose, let us introduce yet another map $ \widehat\xi:\overline{\mathcal H}\to\R^{2,1}$ by:
\begin{equation}\label{eq:defi q new}
{\widehat \xi}(\mathsf x,\mathsf y) = \begin{bmatrix}  2\mathsf x \\ 1- \mathsf x^2 - \mathsf y^2 \\ 1+ \mathsf x^2 + \mathsf y^2 \end{bmatrix}~.
\end{equation}
This map parameterizes all timelike and lightlike directions in Minkowski space, except $(0,-1,1)$, which corresponds to $\vec{\theta_0}$. Observe that $\widehat \xi=R\circ \xi$, where $\xi$ is the map we have used in this section (Equation \eqref{eq:defi q}) and $R$ is the reflection in the plane $y=0$. Hence, defining the parabolic null support function $\widehat \psi$ \emph{with respect to $-\pi/2$} as
$$
\widehat \ph_\lambda(\mathsf x) = \sup_{{p} \in \Sigma_\lambda} \langle {p}, \widehat{\xi}(\mathsf x,0) \rangle~,
$$
we immediately obtain from the identity $\Sigma_{-\lambda}=R(\Sigma_\lambda)$ (Remark \ref{rmk sign lambda}) that
$$\widehat \ph_\lambda(\mathsf x)= \sup_{{p} \in \Sigma_\lambda} \langle {p}, \widehat{\xi}(\mathsf x,0) \rangle= \sup_{{p} \in \Sigma_\lambda} \langle {p}, R\circ {\xi}(\mathsf x,0) \rangle= \sup_{{p} \in \Sigma_\lambda} \langle R(p), {\xi}(\mathsf x,0) \rangle=\ph_{-\lambda}(\mathsf x)~.$$
Therefore from Proposition \ref{pr:xlnx}, we have:
$$\widehat \ph_\lambda(\mathsf x)=2\lambda |x|\log|x|$$
for $\mathsf x<0$. Similarly to Remark \ref{rmk:parabolic equivariance} (or by a direct check using \eqref{eq:defi q new}), we have $L\circ \widehat \xi(\mathsf x,\mathsf y)=\widehat \xi(\mathsf x+c,\mathsf y)$. 
Hence applying a parabolic isometry fixing the direction $\vec\theta_0$ to a surface has the effect of precomposing the parabolic null support function with point at infinity $\theta_0$ by a translation:
$$\widehat\ph'_{\lambda}(\mathsf x):=\sup_{{p} \in \Sigma'_\lambda=L(\Sigma_\lambda)} \langle {p}, \widehat{\xi}(\mathsf x,0) \rangle=\sup_{{p} \in \Sigma_\lambda} \langle {p}, \widehat{\xi}(\mathsf x+c,0) \rangle=\widehat\ph_\lambda(\mathsf x+c)~.$$
In particular, 
$$\widehat \ph'_\lambda(\mathsf x)=2\lambda |\mathsf x+c|\log|\mathsf x+c|~.$$
Since 
$$\lim_{t\to +\infty}\frac{(t+c)\log(t+c)}{t\log t}=1~,$$
we have that, for any $\epsilon>2\lambda$ and for $|\mathsf x|$ large,
$$\widehat \ph'_{\lambda}(\mathsf x)\leq \epsilon|\mathsf x|\log|\mathsf x|$$
This shows the null support function of $L(\Sigma_\lambda)$ is smaller than that of $\Sigma_{\epsilon/2}$ in a one-sided neighbourhood of $\theta_0=-\pi/2$. Going back to the parabolic null support function with point at infinity $-\pi/2$, this means (from Proposition \ref{pr:xlnx}) that $\ph'_\lambda(\mathsf x)\leq -\epsilon|\mathsf x|\log|\mathsf x|$ for $\mathsf x$ in a left neighbourhood of $0$. Using Lemma \ref{lemma:mouse2}, the proof is complete.
\end{proof}

\subsection{Proof of Theorem \ref{thm: xlogx incomplete}}
We are finally ready to conclude the proof of Theorem \ref{thm: xlogx incomplete}.

\begin{proof}[Proof of Theorem \ref{thm: xlogx incomplete}]
Let $\phi$ be the null support function as in the hypothesis of the theorem. Assume as usual that  $\phi(\theta_0)=0$, and let's say that the curvature of $\Sigma$ is bounded above by $-\kappa$. By Proposition \ref{prop:barrier xlogx} (possibly applying a reflection) let be $\phi_+$ the null support function of a convex entire spacelike surface $\Sigma_+$ of curvature $-\kappa$, with the following properties:
\begin{itemize}
\item $\Sigma_+$ is incomplete at $\theta_0$,
\item $\phi_+(\theta_0)=\phi(\theta_0)=0$,
\item the set where $\phi_+=+\infty$ is a proper subset of the set where $\phi=+\infty$,
\item $\phi_+$ is continuous on the subset of $ \mathbb S^1$ where it is finite,
\item $\phi_+\leq\phi$ in a one-sided neighbourhood of $\theta_0$, on the side of $\theta_0$ where both are finite.
\end{itemize}
Exactly as in the proof of  Theorem \ref{thm: Holder incomplete}, we must now modify $\Sigma_+$ further so that $\phi_+\leq\phi$ holds globally, so that by Corollary \ref{cor:complete comparison} we will deduce that $\Sigma$ is incomplete at $\theta_0$. To achieve this, it now suffices to translate $\Sigma_+$ in the direction of $\vec\theta_0$. Indeed, $\phi_+$ is continuous on the complement of the open interval where it is equal to $+\infty$, whereas $\phi$ is lower-semicontinuous, hence bounded below.  The same of course holds for the parabolic null support functions $\widehat \psi_+$ and $\widehat \psi$ with point at infinity $\theta_0$. This implies that $\widehat\psi-\widehat\psi_+$ is bounded below on $\{\widehat\psi_+<+\infty\}$. Recall from  the proof of  Theorem \ref{thm: Holder incomplete} that translating $\Sigma_+$ in the direction $c\vec\theta_0$ has the effect of adding the constant $-2c$ to $\widehat \psi_+$, and does not change the elliptic null support function at $\theta_0$. Hence choosing $c$ large enough, one can make sure that  $\widehat \psi_+\leq \widehat\psi$, and therefore also $\phi_+\leq \phi$ on the whole $\mathbb S^1$. The proof is then concluded by Corollary \ref{cor:complete comparison}.
\end{proof}

\bibliographystyle{alpha}
\bibliography{bs-bibliography_comp.bib}

\end{document}